\documentclass{amsart}

\title{Contact homology of orbit complements and implied existence}
\author{Al Momin}
\address{Purdue University Dept. of Mathematics \\  150 N. University St. \\  West Lafayette IN, 47906 U.S.A. \\amomin@math.purdue.edu}

\usepackage{amssymb, graphicx, subfigure}

\newcommand{\C}{\mathbb{C}}
\newcommand{\R}{\mathbb{R}}
\newcommand{\Z}{\mathbb{Z}}
\newcommand{\Q}{\mathbb{Q}}

\newcommand{\M}{\mathcal{M}}
\newcommand{\calN}{\mathcal{M}}

\newcommand{\J}{\mathcal{J}}
\newcommand{\B}{\mathcal{B}}
\newcommand{\E}{\mathcal{E}}

\newcommand{\qi}{\mathbf{i}}

\newcounter{newcounter}[section]

\numberwithin{equation}{section}
\numberwithin{newcounter}{section}
\numberwithin{figure}{section}
\numberwithin{footnote}{section}

\newtheorem{corollary}[newcounter]{Corollary}
\newtheorem{definition}[newcounter]{Definition}
\newtheorem{example}[newcounter]{Example}
\newtheorem{lemma}[newcounter]{Lemma}
\newtheorem{proposition}[newcounter]{Proposition}
\newtheorem{remark}[newcounter]{Remark}
\newtheorem{theorem}[newcounter]{Theorem}

\begin{document}

\renewcommand{\thesubfigure}{(\arabic{subfigure})}

\begin{abstract}
For Reeb vector fields on closed $3$-manifolds, cylindrical contact homology is used to show that the existence of a set of closed Reeb orbit with certain knotting/\-linking properties implies the existence of other Reeb orbits with other knotting/\-linking properties relative to the original set.  We work out a few examples on the $3$-sphere to illustrate the theory, and describe an application to closed geodesics on $S^2$ (a version of a result of Angenent in \cite{Angenent}).
\end{abstract}

\maketitle

\section{Introduction}

Let $V$ be a closed $3$-manifold with a contact form $\lambda$ and associated Reeb vector field $X_{\lambda}$.  In this article we will be concerned with the following question about closed orbits:

\begin{quote}
\textbf{General Question:}  If one has a Reeb vector field with a known set of closed Reeb orbits $L$, can one deduce the existence of other closed Reeb orbits from knowledge about $L$?
\end{quote}

\noindent Let us call an affirmative answer to such a question an \emph{implied existence} result.

We will see that in some instances one can obtain an affirmative answer to this question.  These results agree with affirmative answers to analogous questions for certain similar conservative systems in low dimensions studied in e.g.~\cite{Angenent, MR1974892, GVBVVW}, but the present methods differ significantly.  %
Specifically, to exhibit implied existence we study cylindrical contact homology on the complement of $L$ for non-degenerate contact forms $\lambda$ imposing as few conditions on the orbit set as we can manage.  We will show in detail one approach to such a theory inspired by the intersection theory of \cite{Siefring1} for the necessary compactness arguments\footnote{It has been pointed out to the author that other approaches are possible e.g.~using convexity for compactness arguments, see e.g.~\cite{2010arXiv1004.2942C}.}.  We use established techniques \cite{2008arXiv0809.5088C, Bourgeois_thesis} to compute this homology in some cases, and use these computations as a tool to affirmatively answer the above question in certain cases.

\subsection{A version of cylindrical contact homology on Reeb orbit complements}

In this paper we denote by $V$ a $3$-dimensional manifold.  A one-form $\lambda$ on $V$ is called a contact form if $\lambda \wedge d\lambda$ is everywhere non-zero.  Such a one-form uniquely determines its \emph{Reeb vector field} $X_{\lambda}$ by the equations $\lambda(X_{\lambda}) \equiv 1$, $d\lambda(X_{\lambda},\cdot) \equiv 0$.  It also determines a distribution $\xi = \ker \lambda$, which is a contact structure, and two forms $\lambda_{\pm}$ induce the same contact structure if and only if $\lambda_+ = f \cdot \lambda_-$ for some nowhere vanishing function.  The Conley-Zehnder index of a closed orbit for this vector field is a measure of the rotation of the flow around the closed orbit - see Proposition \ref{prop-Linear Growth} for one characterization of the Conley-Zehnder index.  The Conley-Zehnder index usually depends on a choice of trivializations, but in many cases we will consider, e.g.~the \emph{tight} $3$-sphere, there is a global trivialization which is used to define Conley-Zehnder indices independently of choices.

\subsubsection{The hypotheses $(E)$ and $(PLC)$}

We introduce two types of technical hypotheses below, which are not mutually exclusive.  They simplify the construction of contact homology on $V \backslash L$ which we use to deduce implied existence.  We will give some examples of forms on the $3$-sphere later which we hope will help to clarify these hypotheses.

We will consider pairs $(\lambda,L)$ in which $\lambda$ is a non-degenerate contact form on $V$ and $L$ is a link composed of closed orbits of the Reeb vector field of $\lambda$.  Sometimes it may be convenient to refer to the subset of closed Reeb orbits with image contained in $L$; we will abuse language and denote this subset of closed Reeb orbits by $L$ again\footnote{i.e.~$L$ refers to both a closed embedded submanifold which is tangent to the Reeb vector field $X$, as well as the set of solutions $x: \R / (T \cdot \Z) \rightarrow V, \dot{x} = X(x)$ modulo reparametrization with image contained in $L$ (this includes multiple covers of components of $L$).}.

Following standard terminology\footnote{e.g.~\cite{Hasselblatt-Katok}; see also Proposition \ref{prop-Linear Growth} for another characterization of \emph{elliptic}.}, an orbit is \emph{elliptic} if the eigenvalues of its linearized Poincar\'{e} return map (a linear symplectic map on $\xi$) are \emph{non-real}.
\begin{definition}
 We will say $(\lambda,L)$ satisfies the \emph{``ellipticity''} condition (abbreviated by $(E)$) if 
\begin{itemize}
 \item each orbit in $L$ is non-degenerate elliptic (including multiple covers)
 \item each contractible Reeb orbit $y$ not in $L$ ``links'' with $L$ in the sense that for any disc with boundary $y$, $L$ intersects the interior of the disc.
\end{itemize}
\end{definition}
We shall see these hypotheses force the compactness of moduli of holomorphic cylinders in $V \backslash L$ necessary to define cylindrical contact homology.  Before we continue, let us note a very simple example:

\begin{example}  \label{example-1}
Consider $\R^4 = \R^2 \times \R^2$ with the usual polar coordinates $(r_i,\theta_i)$ ($i = 1,2$) on each $\R^2$ factor of $\R^4$ and let $E$ be the ellipsoid determined by the equation
\[
 \frac{r_1^2 }{a} +  \frac{r_2^2}{b} = 1
\]
\noindent where $a,b$ are positive constants.  The one form
\[
 \lambda_0 = \sum_{i = 1}^2 \frac{1}{2} r_i^2  d\theta_i
\]
\noindent restricted to $E$ defines a contact form which we denote $\lambda$.  If the ratio $a/b$ is irrational, then there are precisely two geometrically distinct closed Reeb orbits, $P' = S \cap \C \times \{ 0 \}$, $Q' = S \cap \{ 0 \} \times \C$.  They are both non-degenerate and elliptic with Conley-Zehnder indices 
\[
 \mathrm{CZ}(P'^k) = 2 \left \lfloor k \left( 1 + \frac{a}{b} \right) \right \rfloor + 1, \qquad \mathrm{CZ}(Q'^k) = 2 \left \lfloor k \left( 1 + \frac{b}{a} \right) \right \rfloor + 1 
\]
Since the linking numbers $\ell(P'^k,Q') = k, \ell(P',Q'^k) = k$, the pairs $(\lambda, P')$, $(\lambda, Q')$, $(\lambda, P'\sqcup Q')$ all satisfy $(E)$.
\end{example}

\noindent See Example \ref{example-2} in section \ref{sec-examples} for a more interesting class of examples.

There is another way to control compactness of holomorphic cylinders.  Again assume $L$ is a link of closed orbits for $\lambda$, and let $[a]$ be the homotopy class of a loop $a$ in the complement.
\begin{definition} \label{def-PLC}
 We say that $(\lambda,L,[a])$ satisfies the \emph{``proper link class''} condition $(PLC)$ if 
\begin{itemize}
\item for any connected component $x \subset L$, no representative $\gamma \in [a]$ can be homotoped to $x$ inside $V \backslash L$ i.e.~there is no homotopy $I:[0,1] \times S^1 \rightarrow V$ (with $I(0,\cdot) = \gamma$ and $I(1,\cdot) = x$) such that $I([0,1) \times S^1) \subset V \backslash L$.  We will call such a $[a]$ a \emph{proper link class} for $L$.
\item for every disc $F$ with boundary $\partial F = y$ a closed (non-constant, but possibly multiply covered) Reeb orbit (possibly in $L$), there is a component $x$ of $L$ that intersects the interior of $F$.
\end{itemize}
\end{definition}

See Figure \ref{fig-PLC} for an example, and Figure \ref{fig-notPLC} for a counter-example.

\begin{figure}
  \centering
\subfigure[The homotopy class of $a$ is a proper link class for the two component link $L$ above.] 
{
    \includegraphics[scale=.34]{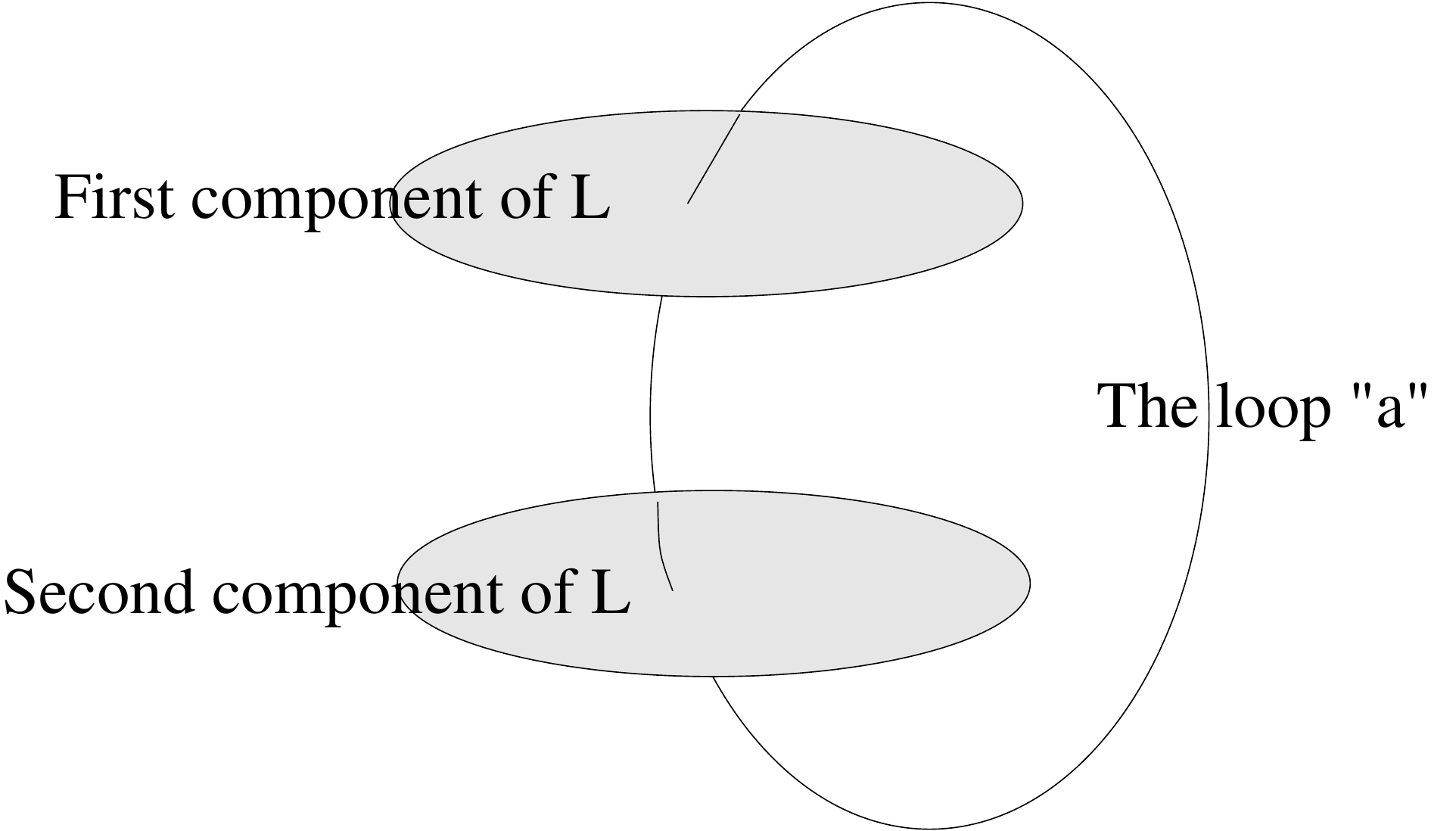}
    \label{fig-PLC}
} 
\subfigure[Here $a$ does not represent a proper link class for $L$.]
{
    \includegraphics[scale=.31]{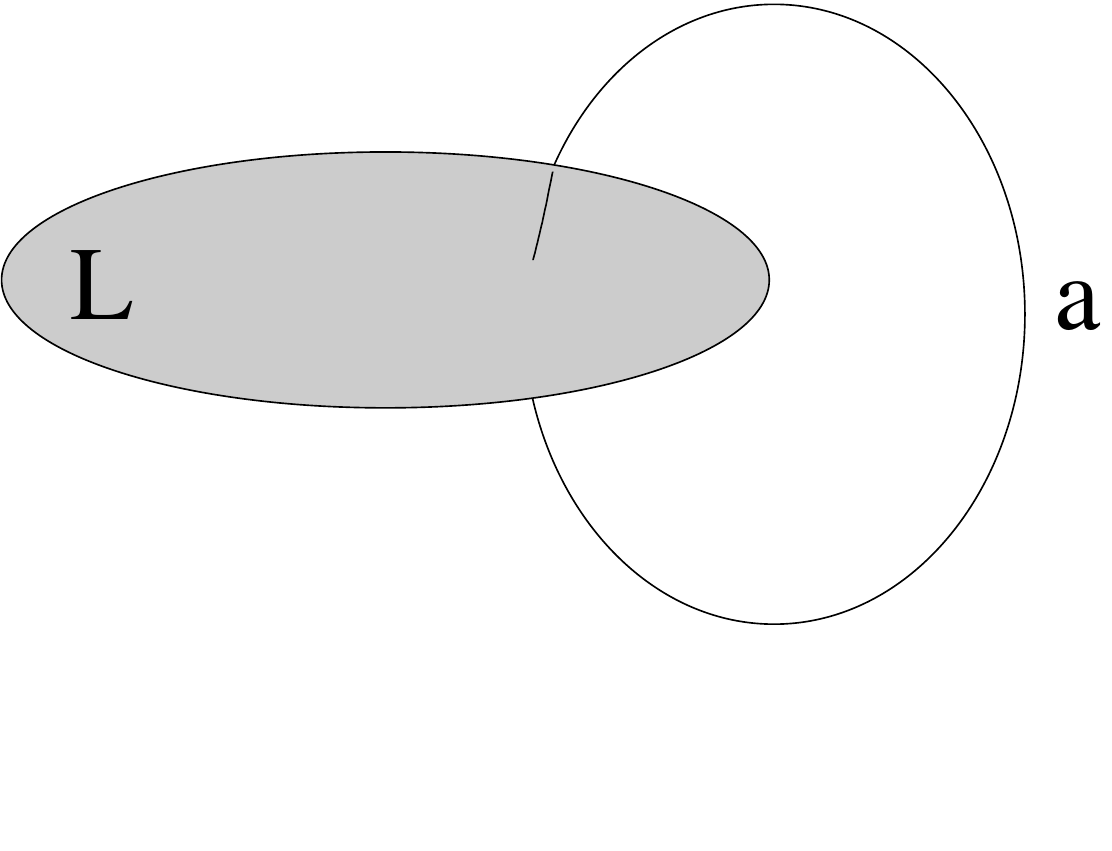}
    \label{fig-notPLC}
}

\caption{Example and counter-example of proper link classes in Definition \ref{def-PLC}.}

\end{figure}

The class $[a]$ is meant to be analogous to a \emph{proper braid class} studied in \cite{MR1974892, GVBVVW}.  Both conditions $(E)$ and $(PLC)$ contain a \emph{``no contractible orbits''} hypothesis.  Such a hypothesis is necessary to preclude ``bubbling'', which is an obstruction to defining cylindrical contact homology in general \cite{EGH}\footnote{The necessary hypotheses can actually be weakened with greater care.  We remark that the ``no contractible orbits'' hypothesis in $(PLC)$ is more restrictive.}.
We point to Examples \ref{example-2}, \ref{example-3} below for concrete examples of orbits and (Morse-Bott) contact forms on the tight $S^3$ satisfying conditions $(E)$ and $(PLC)$ respectively.

\subsubsection{An invariant, contact homology}

We define a relation on pairs $(\lambda,L)$ as above.  Suppose we have two such pairs $(\lambda_{\pm},L_{\pm})$.  We write $(\lambda_+,L_+) \sim (\lambda_-,L_-)$ if $\ker (\lambda_+) = \ker (\lambda_-)$ and $L_+ = L_-$.  We say $(\lambda_+,L) \geq (\lambda_-,L)$ if they are related by $\sim$ and the Conley-Zehnder indices of the orbits in $L$ (including multiple covers) are always greater or equal when considered as orbits for $\lambda_+$ than when considered as orbits for $\lambda_-$.  If $(\lambda_0,L) \geq (\lambda_1,L)$ and $(\lambda_1,L) \geq (\lambda_0,L)$ then we write $(\lambda_0,L) \equiv (\lambda_1,L)$.

In Section \ref{sec-CCHint} we associate with a pair $(\lambda,L)$ satisfying $(E)$ and the homotopy class $[a]$ of a loop $a$, or with a triple $(\lambda,L,[a])$ satisfying $(PLC)$, an invariant $CCH^{[a]}_*([\lambda] \mbox{ rel } L)$.  It is invariant in the sense that it depends only on the equivalence class $\equiv$ in case $(\lambda,L)$ satisfy $(E)$ (that is, it depends on the Conley-Zehnder indices of the elliptic orbits in $L$), and in the case $(PLC)$ it depends only on equivalence classes of $\sim$ (i.e.~only on the contact structure $\xi$).  $CCH^{[a]}_*([\lambda] \mbox{ rel } L)$ will be an isomorphism class of graded vector spaces, so it makes sense to say whether or not it is zero.

It follows easily from the constructions of the invariants in Section \ref{sec-CCHint} that:

\begin{theorem} \label{thm-forcing1}
 Given a homotopy class $[a]$, if the pair $(\lambda,L)$ satisfies $(E)$ or the triple $(\lambda,L,[a])$ satisfies $(PLC)$, and $CCH^{[a]}_*([\lambda] \mbox{ rel } L) \neq 0$, then there is a closed Reeb orbit in the homotopy class $[a]$.
\end{theorem}

\subsection{A general implied existence result}
\label{sec-impliedexistence}

The non-degeneracy assumption on the form $\lambda$ in Theorem \ref{thm-forcing1}, and the implicit assumptions about contractible orbits are unfortunate, but can be weakened considerably by stretching-the-neck and compactness arguments given in Section \ref{sec-degeneracies}.

\begin{theorem}
\label{thm-forcing2}
Suppose $\lambda$ is a contact form with a closed orbit set $L$.  Suppose either
\begin{itemize}
\item $L$ is non-degenerate and elliptic, or
\item \begin{enumerate}
	\item $L$ is such that every disc $F$ with boundary $\partial F  \subset L$ and $[\partial F] \neq 0 \in H_1(L)$ has an interior intersection with $L$, and
       \item $[a]$ is a proper link class relative to $L$.
      \end{enumerate}
\end{itemize}
If $[a]$ is simple and $CCH^{[a]}_*([\lambda] \mbox{ rel } L) \neq 0$, then there is a closed Reeb orbit in the homotopy class $[a]$.
\end{theorem}

To clarify a possible point of confusion in the statement of Theorem \ref{thm-forcing2}, when we say $CCH^{[a]}_*([\lambda] \mbox{ rel } L) \neq 0$ we mean that there is a $\lambda' \equiv \lambda$ resp.~ $\lambda' \sim \lambda$ satisfying $(E)$ resp.~$(PLC)$ for which $CCH_*^{[a]}$ can be computed and is non-zero.

When $[a]$ is a proper link class, this theorem requires no non-degeneracy hypotheses whatsoever on $\lambda$, and in the ellipticity case requires only non-degeneracy of the orbits in $L$.  For a concrete example where this is applicable see Example \ref{example-3}, and Section \ref{example-5}.

\subsection{Some applications in $S^3$}

The results here are corollaries of Theorem \ref{thm-forcing2} in Section \ref{sec-impliedexistence} and the examples and computations of Section \ref{sec-examples}.

One application is the existence of analogs of ``$(p,q)$-type orbits'' if a certain pair of closed Reeb orbits exists and satisfy a ``non-resonance'' condition (namely the pair violates the ``resonance'' condition introduced in \cite{MR2342700} when the differential of cylindrical contact homology in the tight $S^3$ vanishes).  In the following statement $\ell(\cdot, \cdot)$ denotes the linking number of two knots in $S^3$.

\begin{corollary} \label{cor-forcing1}
Let $\lambda$ be a tight contact form on the $3$-sphere.  Suppose that there is a pair of periodic orbits $L_1, L_2$ such that $L_1 \sqcup L_2$ is the Hopf link with self-linking $0$\footnote{By Theorem $4.1$ in \cite{Etnyre_Horn-Morris}, all such links are transversely isotopic.  The pair $H_1 \sqcup H_2$ of Example \ref{example-2} is one such example.}.  Suppose $L_1$ and $L_2$ (and all multiple covers) are non-degenerate elliptic, and let $\theta_1, \theta_2$ be the unique irrational numbers satisfying 
\begin{align*}
 \mathrm{CZ}(L_1^k) &= 2 \left \lfloor k\left(1 + \theta_1 \right) \right \rfloor + 1, \mbox{ for all } k \geq 1 \\
 \mathrm{CZ}(L_2^k) &= 2  \left \lfloor k\left(1 + 1 / \theta_2 \right) \right \rfloor + 1, \mbox{ for all } k \geq 1  \\
\end{align*}
\noindent  Then
\begin{enumerate}
\item If $\theta_1,\theta_2 >0$ and $\theta_1 \neq \theta_2$, then for each relatively prime pair $(p,q)$ such that
\[
 \frac{q}{p} \in (\theta_1,\theta_2) \sqcup (\theta_2,\theta_1)
\] 
\noindent there is a simple closed Reeb orbit $P_{(p,q)}$ such that $\ell(P_{(p,q)}, L_1)$ $= q$ and $\ell(P_{(p,q)}, L_2)$ $= p$.
\item If $\theta_1 < 0 < \theta_2$ (if $\theta_2 < 0 < \theta_1)$, relabel the orbits so that this is the case), then for each relatively prime pair $(p,q)$ such that
\[
p > 0, \mbox{ and } \frac{q}{p} \in (\theta_1,\theta_2)
\] 
\noindent there is a simple closed Reeb orbit $P_{(p,q)}$ such that $\ell(P_{(p,q)}, L_1)$ $= q$ and $\ell(P_{(p,q)}, L_2)$ $= p$.
\item   If $\theta_1,\theta_2 < 0$, then for each relatively prime pair $(p,q)$ such that
\[
p > 0 \mbox{ and } \frac{q}{p} \in (\theta_1,1]; \mbox{ or } q > 0 \mbox{ and } \frac{p}{q} \in (\frac{1}{\theta_2},1]
\]
\noindent there is a simple closed Reeb orbit $P_{(p,q)}$ such that $\ell(P_{(p,q)}, L_1)$ $= q$ and $\ell(P_{(p,q)}, L_2)$ $= p$.
\end{enumerate}

\end{corollary}

\begin{remark}
There are no hidden hypotheses on the contact form being either \emph{dynamically convex} or non-degenerate (except in the explicit hypotheses on the orbits $L_1, L_2$).  There are other approaches to prove such a result using the surface of section constructed by \cite{HWZ_convex}, \cite{2010arXiv1006.0049H} and applying the result of \cite{MR1161099}; however, the construction of these surfaces of section require additional hypotheses, so it is not clear that they always exist.  The result above applies even when there may be no such surface of section.
\end{remark}

\begin{proof}
By Theorem $4.1$ in \cite{Etnyre_Horn-Morris}, there is a transverse isotopy from $L_1 \sqcup L_2$ to the link $H_1 \sqcup H_2$ in Example \ref{example-2}, which extends to an ambient contact isotopy, and therefore a contactomorphism taking $L_1 \sqcup L_2$ to $H_1 \sqcup H_2$.  By applying the given contactomorphism we can assume $(\lambda,L_1 \sqcup L_2) \equiv (\lambda',H_1 \sqcup H_2)$ with $(\lambda',H_1 \sqcup H_2)$ from Example \ref{example-2}.  It follows immediately from Theorem 
\ref{thm-forcing2} 
(each component is elliptic by hypothesis) and the computation in Example \ref{example-2}.
\end{proof}

It is interesting to interpret this result in the case the contact form is obtained from a metric on $S^2$.  Angenent \cite{Angenent} proved that if one has a simple, closed geodesic for a $C^{2,\mu}$-Riemannian metric on $S^2$, then whenever Poincar\'{e}'s inverse rotation number $\rho \neq 1$, for every $p/q \in (\rho,1) \cup (1,\rho)$ there is a geodesic $\gamma_{p,q}$ with the flat-knot type of a $(p,q)$-satellite geodesic.  The following Corollary is a version of this result with stronger hypotheses on the rotation number and weaker conclusions about the resulting geodesics, but allowing the metric to be reversible Finsler.  It is a direct consequence of Corollary \ref{cor-forcing1} applied to the geodesic flow of a reversible Finsler metric on $S^2$ which will be proved in Section \ref{sec-FinslerMetrics}
\begin{corollary} \label{cor-Angenent}
Let $F$ be a reversible Finsler metric on $S^2$.  Suppose $\gamma$ is a simple, closed geodesic with irrational inverse rotation number $\rho \neq 1$.  Then for every pair of relatively prime integers $p,q$ such that
\[
\frac{p+q}{2q} \mbox{ or } \frac{p+q}{2p} \in (\rho, 1] \cup [1 , \rho)
\]
\noindent there is a geodesic $\gamma_{p,q}$ with the following topological property.  If $\gamma,\overline{\gamma}$ denote the lifted double covers of the geodesic to $S^3$ (traversed ``forwards'' and ``backwards''), then $\gamma_{p,q}$ has linking number $p$ with $\gamma$ and $q$ with $\overline{\gamma}$.  In particular, the geodesics are geometrically distinguished (up to a double-count by traversing $\gamma_{p,q}$ oppositely).
\end{corollary}

Corollaries \ref{cor-forcing1} and \ref{cor-Angenent} will be improved to cover the degenerate and hyperbolic cases as well in collaboration with Umberto Hryniewicz and Pedro Salom\~ao \cite{HMS}.

We give other examples by considering fibered knots or links in the $3$-sphere, in which conclusions can be drawn without non-degeneracy hypotheses\footnote{In Corollary \ref{cor-forcing2} we will also give another example, related to the example used in the previous Corollary, which does not assume non-degeneracy of the orbit set $L$.  We postpone the statement to section \ref{sec-examples} where it will be easier to describe the link $L$.}.  One class of examples is obtained as follows.  Let us say that a knot $B$ in the tight $3$-sphere is a \emph{tight fibered hyperbolic knot} if the following conditions hold:

\begin{enumerate}
 \item (fibered) The knot $B$ is the binding of an open book decomposition of $S^3$
 \item (tight) The contact structure supported by the open book is the tight contact structure on $S^3$
 \item (hyperbolic) The monodromy map $h$ of the associated open book decomposition is pseudo-Anosov
 \item The map $h^* - I:H^1(S;\R) \rightarrow H^1(S;\R)$ (where $S$ is a page of the open book) is invertible
\end{enumerate}

We remark that the fourth condition actually follows automatically from the first; it will be a convenient fact (see section \ref{sec-fiberedpseudoAnosov} where we also explain why this follows).  We make some further brief remarks about this class of knots in section \ref{sec-fiberedpseudoAnosov}.  For now we merely note that there are infinitely many examples, such as the Fintushel-Stern knot which is the Pretzel knot $(-2,3,7)$.  For \emph{tight fibered hyperbolic} knots in $S^3$ (see Section \ref{sec-fiberedpseudoAnosov}):

\begin{corollary} \label{cor-fiberedpseudoAnosov}
Suppose $\lambda$ is a tight contact form on $S^3$.  If its Reeb vector field has a closed orbit which is a tight fibered hyperbolic knot realizing the Thurston-Bennequin bound as a transverse knot, then there are infinitely many geometrically distinct closed Reeb orbits and the number of such orbits of period at most $T$ is bounded below by an exponential function of $T$.
\end{corollary}

The required calculation is essentially due to Colin-Honda \cite{2008arXiv0809.5088C}, but turns out simpler because one does not need to consider holomorphic curves that intersect the binding.

\begin{remark}
Again, there are no hidden non-degeneracy hypotheses, and it applies to all tight contact forms (not only dynamically convex ones).  Theorem \ref{thm-forcing2} allows one to draw conclusions about free homotopy classes which must contain closed Reeb orbits.
\end{remark}

Finally, cylindrical contact homology is usually thought to be only applicable to tight contact structures.  However, after removing a certain orbit set, it may be possible to apply cylindrical contact homology \emph{on the complement of the orbit set}.  For example, we can take the figure eight knot in the $3$-sphere, which satisfies properties $1,3,4$ in the definition of fibered hyperbolic supporting knots, but not property $2$ (that is, it is fibered and hyperbolic, but does not support the tight contact structure).  Hence it supports an over-twisted contact structure; Theorem \ref{thm-forcing2} still applies from which one can deduce the following:

\begin{corollary} \label{cor-forcing3} 
Let $\lambda$ be a contact form for the (over-twisted) contact structure on $S^3$ supported by the open book decomposition with binding the figure eight knot, page diffeomorphic to a once-punctured torus, and monodromy map given by the matrix transformation
\[
 \begin{bmatrix}  2 & 1 \\ 1 & 1  \end{bmatrix}
\]
\noindent Suppose $\lambda$ has a closed Reeb orbit transversely isotopic to the binding.  Then the number of geometrically distinct periodic orbits of action at most $N$ grows at least exponentially in $N$.
\end{corollary}

We note that in the above one can describe homotopy classes that must contain a closed orbit.

We will go over the details of this particular example in Section \ref{sec-examples}.  An analogous result will hold for any (non-tight) fibered hyperbolic knot in $S^3$, examples of which are plentiful \cite{knotinfo}.

\subsection{Further comments, outline, and acknowledgements}
\label{sec-comments}

\subsubsection{Comments}

It is possible to study the cylindrical contact homology of stable Hamiltonian structures on complements of elliptic Reeb orbits (or in proper link classes) in general.  For example, the mapping tori of Hamiltonian diffeomorphisms of surfaces fits in this category.  In this case there are clearly no contractible Reeb orbits, and bubbling of spheres is not difficult to rule out in most cases.  The portion of the loops space corresponding to the first return map are all simple homotopy classes and thus $CCH_*$ is well-defined and coincides with Floer homology.  This particular case has been carried out explicitly (at least in the case the surface is a disc) and is called ``braid Floer homology'' in \cite{GVBVVW}.

\subsubsection{Outline}  In section $2$ we recall the geometric and analytic set-up.  Section $3$ describes an intersection theory of pseudoholomorphic curves in symplectizations \cite{Siefring2} and derives some compactness results from this intersection theory.  In section $4$ we describe the chain complexes, maps and homotopies.  In section $5$, we describe how to draw conclusions when the contact form might be degenerate or there may be contractible (in $V \backslash L$) closed Reeb orbits.  In section $6$ some explicit examples on the $3$-sphere are worked out in detail.

\subsubsection{Acknowledgments}  This work is a continuation of the author's Ph.D. thesis; I thank my adviser Helmut Hofer for his invaluable guidance and support.  I thank Alberto Abbondandolo, Umberto Hryniewicz, Richard Siefring, Robert Vandervorst, and Chris Wendl for influential discussions, and Matt Hedden for answering some questions about fibered knots in $S^3$.  Finally, I thank Peter Albers and Pedro Salom\~{a}o for many helpful suggestions, and the anonymous referee for a careful reading and for generously offering numerous detailed improvements to the readability.

\section{Review of geometric and analytic set-up}
\label{sec-Preliminaries}

\subsection{Symplectizations}
\label{sec-symplectizations}
Let $(V,\xi)$ be a co-oriented contact manifold.  Then there is a symplectic manifold associated with it, $W_{\xi} \dot{=} \xi^{\perp} \backslash 0$ (the annihilator of $\xi$ in $T^*V$ minus the zero section) which is a sub-bundle and symplectic submanifold of $T^*V$ with respect to the exterior derivative of the Liouville form, $d\theta_{can}$.  The symplectic manifold $(W_{\xi},\omega_{\xi} = d\theta_{can}|_{W_{\xi}})$ is called the \emph{symplectization} of $(V,\xi)$.  Because $\xi$ is co-oriented there are two components, $W_{\xi,\pm}$.  We will consider only the component $W_{\xi,+}$ and when we write $W_{\xi}$ we mean this component only.  There is a natural $\R$ action $(a,\lambda) \mapsto a * \lambda = e^a \cdot \lambda$ given by scalar multiplication in the fibers, and a natural projection to $V$ by restricting the projection $T^*V \rightarrow V$ to $W_{\xi}$.

Any positively co-oriented contact form $\lambda$ for $\xi$ defines a global section of $W_{\xi}$ which gives an identification 
\[ 
\psi_{\lambda}: W_{\xi} \cong \R \times V, \qquad \lambda' \mapsto \left( a = \ln \frac{\lambda'}{\lambda},x = \pi(\lambda') \right) 
\]
(where $\pi$ denotes the restriction to $W_{\xi}$ of the projection $T^*V \rightarrow V$) in which the symplectic form for $W_{\xi}$ is $d(e^a \cdot \lambda)$.

Given such a $\lambda$ we denote by $X$ the Reeb vector field, defined by
\[
 X \neg \lambda \equiv 1, \qquad X \neg d\lambda \equiv 0
\]
\noindent  We will later assume that $\lambda$ is non-degenerate, a generic condition meaning that no closed periodic orbit has a Floquet multiplier equal to $1$; for the moment this is not necessary.  Given a contact form $\lambda$, we define a splitting of $TW_{\xi}$:
\begin{align*}
 T_{\lambda_0(x)} W_{\xi} & = \R \cdot \partial_a \oplus \R \cdot d(\psi^{-1}_{\lambda})_{\psi(\lambda_0)} (0,X) \oplus d(\psi_{\lambda}^{-1})_{\psi(\lambda_0)} (0 \times \xi) \\
& =: \R \cdot \partial_a \oplus \R \cdot \widehat{X} \oplus \widehat{\xi}
\end{align*}
\noindent  We might abuse notation and use $X$ to denote $\widehat{X}$ and $\xi$ for $\widehat{\xi}$\footnote{The subspaces $\R \cdot \widehat{X}$ and $\widehat{\xi}$ both depend on the choice of $\lambda$.}.

\subsubsection{Almost-complex structures}  Given a contact form $\lambda$ for $\xi$, consider the symplectic vector bundle $(\xi, d\lambda)$ over $V$ and let $\J(\xi)$ be the set of $d\lambda$-compatible almost-complex structures.  This defines a set of almost-complex structures on $\widehat{\xi}$ via
\[
 \widehat{J}(\lambda_0) = (d\psi_{\lambda}^{-1})_{\psi(\lambda_0)}|_{\xi} \circ J \circ (d\psi_{\lambda})_{\lambda_0}|_{\widehat{\xi}}
\]
Finally this extends to all of $T_{\lambda_0} W_{\xi}$ by
\[
  \partial_a \mapsto \widehat{X} \qquad \widehat{X} \mapsto -  \partial_a
\]

We denote the set of almost-complex structures on $TW_{\xi}$ that arise in this way $\J(\lambda)$; we can identify this with $\J(\xi)$.  The almost-complex structures in $\J(\lambda)$ are called \emph{cylindrical} almost-complex structures.

\subsubsection{Cylindrical ends}  To compare contact homology between different choices of $\lambda$, one studies the symplectization and almost-complex structure with \emph{cylindrical ends}.  We define an order relation on the fibers of $\pi: W_{\xi} \rightarrow V$ as follows: given $\lambda_0,\lambda_1 \in \pi^{-1}(x)$, we say $\lambda_0 \prec \lambda_1$ (resp. $\lambda_0 \preceq \lambda_1$) if $\lambda_1 / \lambda_0 > 1$ (resp. $\lambda_1 / \lambda_0 \geq 1$).  Given two contact forms $\lambda_0,\lambda_1$ for $\xi$, we write $\lambda_0 \prec \lambda_1$ if $\lambda_0(x) \prec \lambda_1(x)$ on each fiber, or equivalently when we write $\lambda_1 = r \lambda_0$ we have $r >1 $ pointwise.  If $\lambda_- \prec \lambda_+$, then we set
\[
 \overline{W}(\lambda_-,\lambda_+) = \left \{ \lambda \in W_{\xi} | \lambda_-(\pi(\lambda)) \preceq \lambda \preceq \lambda_+(\pi(\lambda)) \right\}
\]
This is an exact symplectic cobordism between $(V,\lambda_-),(V, \lambda_+)$.  Let
\[ \begin{split}
 W^-(\lambda_-) &= \left \{ \lambda \in W_{\xi} | \lambda \preceq \lambda_-(\pi(\lambda)) \right\} \\
 W^+(\lambda_+) &= \left \{ \lambda \in W_{\xi} | \lambda_+(\pi(\lambda)) \preceq \lambda \right\}
\end{split} \]
\noindent so\footnote{By $\partial^+$ we mean $\partial_a$ is outward pointing and by $\partial^-$ we mean $\partial_a$ is inward pointing.}
\[ 
 W_{\xi} = W^-(\lambda_-) 
\bigcup_{\substack{\partial^+ W^-(\lambda_-) = \\ \partial^- \overline{W}(\lambda_-,\lambda_+)}} \overline{W}(\lambda_-,\lambda_+) 
\bigcup_{\substack{\partial^+ \overline{W}(\lambda_-,\lambda_+) \\ =  \partial^- W^+(\lambda_+)}} 
W^+(\lambda_+) 
\]

An almost-complex structure with cylindrical ends is then an almost-complex structure $J$ such that
\begin{itemize}
 \item $J$ agrees with $\widehat{J}_{\pm} = \widehat{J_{\pm}}(\lambda_{\pm})$ on (a neighborhood of) $W^{\pm}$
 \item $J$ is $\omega_{\xi}$-compatible on all of $W$
\end{itemize}
We denote the set of such almost-complex structures by $\J(\widehat{J}_-,\widehat{J}_+)$.  A well-known argument shows that this is a non-empty contractible set.  For $J \in \J(\widehat{J}_+,\widehat{J}_-)$ the almost-complex manifold $(W,J)$ is said to have \emph{cylindrical ends} $W^{\pm}$.

It is also necessary to consider families of almost-complex structures; we will denote by $\J_{\tau}(\widehat{J}_-,\widehat{J}_+)$ the space of smooth paths $[0,1] \rightarrow \J(\widehat{J}_-,\widehat{J}_+)$, $\tau \mapsto J_{\tau}$.

\subsubsection{Splitting almost-complex structures}

Suppose we have $\lambda_- \prec \lambda_0 \prec \lambda_+$.  Consider cylindrical almost-complex structures $\widehat{J}_-, \widehat{J}_0, \widehat{J}_+$ and almost-complex structures $J_1 \in \J(\widehat{J}_-,\widehat{J}_0)$, $J_2 \in \J(\widehat{J}_0,\widehat{J}_+)$.  Then there is a smooth family of almost-complex structures $J'_R$ on $W_{\xi}$ for $R \geq 0$ defined by (using the coordinates $\psi_{\lambda_0}$):
\[
J'_R(a,x) = \left\{  \begin{matrix} 
J_1(a + R,x) & \mbox{if } a \leq R \\
J_2(a-R, x) & \mbox{if } a \geq -R 
\end{matrix} \right.
\]
\noindent which fits together smoothly because $\widehat{J_0}$ is $\R$-translation invariant.

Let $\delta > 0$ be an arbitrarily small but fixed number.  Choose diffeomorphisms $g_{\delta}^{-}: \R \rightarrow (-\infty,0)$, $g_{\delta}^{+}: \R \rightarrow (0,\infty)$ such that $g_{\delta}^-(a) = a-\delta$ if $a \leq 0$ and $g_{\delta}^+(a) = a+\delta$ if $a \geq 0$.  Choose a smooth family of diffeomorphisms $g^{(\delta,R)}:\R \rightarrow \R$ (for $R \geq \delta$) with the properties
\begin{itemize}
\item $g^{(\delta,\delta)}(a) = a$
\item $g^{(\delta,R)}(a) = \left\{  \begin{matrix} 
a - R + \delta & \mbox{if } a \geq R  \\
a + R -\delta  & \mbox{if } a \leq -R 
\end{matrix} \right.$
\item $g^{(\delta,R)}(\cdot+R), g^{(\delta,R)}(\cdot-R)$ converge to $g_{\delta}^+,g_{\delta}^-$ in $C^{\infty}_{\rm{loc}}(\R, \R)$ as $R \uparrow \infty$
\end{itemize}
\noindent Then we have diffeomorphisms $G^{(\delta,R)}(a,x) = (g^{(\delta,R)}(a),x)$ (using the coordinates $\psi_{\lambda_0}$ for $W_{\xi}$), and almost-complex structures $J_R = G^{(\delta,R)}_* J_R'$.  

We may concatenate the matching Hamiltonian structured ends of $(W_{\xi},J_1)$, $(W_{\xi},J_2)$ to get $(W_{\xi},J_1) \odot (W_{\xi},J_2)$ (see section \ref{sec-intersections}).  Then we have a diffeomorphism to the concatenation $G:W_{\xi} \rightarrow (W,J_1) \odot (W,J_2)$ defined by
\begin{itemize}
 \item $e^{g_{\delta}^-(a)} \cdot \lambda_0(x) \mapsto e^a \lambda_0(x) = \psi_{\lambda_0}^{-1}(a,x)$, on $W^-(\lambda_0) \backslash \lambda_0(V) \subset W_{\xi}$
 \item $e^{g_{\delta}^+(a)} \cdot \lambda_0(x) \mapsto e^a \lambda_0(x) = \psi_{\lambda_0}^{-1}(a,x)$, on $W^+(\lambda_0) \backslash \lambda_0(V) \subset W_{\xi}$
 \item $e^{0}      \cdot \lambda_0(x) \mapsto (\infty,x) \sim (-\infty,x) \in \{+\infty\} \times V \sim \{-\infty\} \times V$, on $\lambda_0(V) \subset W_{\xi}$
\end{itemize}
\noindent It follows from the construction of $J_R$ that $G_* J_R$ converges to the concatenated almost-complex structure $J_1 \odot J_2$ uniformly on compact subsets of $W_{\xi} \backslash \lambda_0(V)$.

In $(W_{\xi},J_R)$, denote the regions $W_{\pm} = W^{\pm}(\lambda_{\pm})$, and denote by $W_0$ the region $(-\delta,\delta) * \lambda_0(V)$.  Note that $J_R|_{W_{\pm}} = \widehat{J_{\pm}}$, $(W_0,J_R|_{W_0}) \cong ([-R , R] * \lambda_0(V), \widehat{J_0})$.  Let us denote the set of such almost-complex structures $J_R$ constructed in this way from $J_1, J_2$ by $\J(J_1,J_2) \cong [0,\infty)$.

\subsubsection{Holomorphic maps and finite energy}  A $J$-holomorphic map on a punctured Riemann surface $(\Sigma,j,\Gamma)$ is a map $U:\Sigma \backslash \Gamma \rightarrow W_{\xi}$ that satisfies the Cauchy-Riemann equation
\[
 DU + J \cdot DU \cdot j = 0
\]
We only consider solutions of this equation having finite energy; for the definition of finite energy $J$-holomorphic curves we refer to \cite{BEHWZ}.

A crucial example for cylindrical almost-complex structure $J$ are the trivial cylinders over a periodic Reeb orbit $x$ of some period $T$: it is a solution of the form (for $(s,t) \in \R \times S^1$) $\psi_{\lambda} \circ U(s,t) = (Ts + a_0, x(T \cdot t) )$.  If a contact form is (Morse-Bott) non-degenerate\footnote{We will always assume that all contact forms have non-degenerate asymptotic orbits, so this condition is always satisfied in everything that follows.}, then at any (non-removeable) puncture a finite energy $J$-holomorphic map is asymptotic to a half-trivial cylinder over a closed Reeb orbit in one of the ends $W^{\pm}$, see \cite{HWZ_propsII} or \cite{BEHWZ}.  %
A puncture is called positive (resp.~negative) if it is asymptotic to the positive (resp.~negative) half of a trivial cylinder in $W^{+}$ (resp. $W^-$).  We will say that a finite energy holomorphic map is positively (resp.~negatively) asymptotic to a Reeb orbit $x$ if there is a positive (resp.~negative) puncture $z$ (perhaps already specified) at which $U$ is asymptotic to a half-trivial cylinder over $x$ in $W^+$ (resp.~$W^-$).  

When the contact form is (Morse-Bott) non-degenerate, %
then the asymptotic convergence to the trivial half-cylinder mentioned above is exponential in nature and one has an asymptotic formula \cite{HWZ_propsI, Siefring1}.  If one looks at the space of maps asymptotic to fixed closed Reeb orbits with a given rate of exponential convergence, then there is a Fredholm theory for the Cauchy-Riemann operator \cite{Dragnev, Wendl2008}.

One more crucial property of finite-energy holomorphic curves is the compactness theory of \cite{BEHWZ}.  In section $8$ of \cite{BEHWZ}, a notion of a $k_-|1|k_+$ level holomorphic building and convergence of a sequence of holomorphic maps to that building is given.  We refer to that paper for the definitions.  Theorem $10.3$ of \cite{BEHWZ} extends to the case where $W$ has cylindrical ends as well, so that the space of $J_R$-holomorphic maps can be compactified for sequences with $R \uparrow \infty$.  The compactification consists of a holomorphic building in $(W^-(\lambda_0),J'_1) \sqcup (W^-(\lambda_0),J'_2)$ together with a $\widehat{J_0}(\lambda_0)$ holomorphic building, all of which glue together to $W^-(\lambda_0) \odot W^+(\lambda_0) = W_{\xi}$.  We can think of these as $k_-|1|k_0|2|k_+$ level holomorphic buildings, which we describe next.

Let $S$ be a decorated Riemann surface \cite{BEHWZ} $S$, with each smooth component assigned a level labeled $1,2$ or $(\lambda_i, j_i)$ where $i \in \{1,0,+\}$ and $1\leq j_i \leq k_i$.  Write $S = S_1 \cup S_2$, where $S_1$ consists of the levels labeled either $(\lambda_{-,0},\cdot)$ or $1$, while $S_2$ consists of the parts of the domain labeled $(\lambda_{0,+},\cdot)$ or $2$.  If $k_0 \neq 0$ then $S_1 \cap S_2 \neq \emptyset$.  Let $(U)$ be an 
assignment to each smooth component $S^a \subset S$ of a finite energy holomorphic map $U^a$ with domain $S^a$ such that
\begin{itemize}
 \item If one considers the subset of $(U)$ with domain in $S_1$, it is a holomorphic building in $W_1:= (W_{\xi},J_1)$
 \item If one considers the subset of $(U)$ with domain in $S_2$, it is a holomorphic building in $W_2 := (W_{\xi},J_2)$
 \item The maps glue together to give a continuous map $\overline{(U)}: \overline{S} \rightarrow \overline{W_1} \odot \overline{W_2}$
\end{itemize}
\noindent  This data will be called a $k_-|1|k_0|2|k_+$ holomorphic building.  It is called \emph{stable} if both $(U)_i = (U)|_{S_i}$ ($i = 1, 2$) are stable.

Let $U_k$ be a sequence of $J_{R_k}$ holomorphic maps with the same asymptotic orbits and genus; denote the domains by $\Sigma_k$ (each a Riemann surface).  Let $(U)$ be a building as above with domain $S$.  If $R_k$ is bounded then there is a limit which is a $J_{R_{\infty}}$-holomorphic building, where $R_{\infty}$ is an accumulation point of the sequence.  Else, suppose without loss of generality that $R_k \uparrow \infty$.  The sequence $U_k$ converges to the building $(U)$ if there exists a sequence of diffeomorphisms $\phi_k: S \rightarrow \Sigma_k$ converging as decorated Riemann surfaces \cite{BEHWZ} so that

\begin{itemize}
  \item[C1]  For each level $(\lambda_i,j_i)$ ($i \in \{-,0,+\}, 1 \leq j_i \leq k_i$), the maps $U_k \circ \phi_k|_{S_{(\lambda_i,j_i)}}$ lie in the $W_{i}$ part of $(W_{\xi},J_{R_k})$
  \item[C2]  There is a sequence of constants $c^{(\lambda_i,j_i)}_k$ so that $c^{(\lambda_i,j_i)}_k * U_k \circ \phi_k|_{S_{(\lambda_i,j_i)}}$ converges in $C^{\infty}_{\mathrm{loc}}$ to $U^{(\lambda_i,j_i)} = (U)|_{S_{(\lambda_i,j_i)}}$.
  \item[C3]  The maps $G \circ U_k \circ \phi_k$ converge in $C^{0}(\overline{S},\overline{W})$ to $\overline{(U)}$ (where $G:W_{\xi} \rightarrow W_1 \odot W_2$ is the identification map described above).
\end{itemize}

The asymptotic behavior of finite energy surfaces together with an application of Stokes' theorem shows - for finite energy solutions of the almost-complex structures that we consider - the energy $E(U)$ is bounded by the action of the positive asymptotic orbits.  Then by \cite{BEHWZ}
\begin{lemma}
\label{lem-finiteness}
  For $J \in \J(\widehat{J}_-,\widehat{J}_+)$, or $\J(\lambda)$, or $J = J_R \in \J(J_1,J_2)$, the space of $J$-holomorphic finite energy cylinders positively asymptotic to a closed orbit $x$ of $\lambda_+$ is precompact in $\overline{\M}_J$.\end{lemma}

\subsubsection{A restricted class of almost-complex structures}  Suppose $\lambda_{\pm}$ are two contact forms for $\xi$ satisfying $\lambda_- \prec \lambda_+$ and that $L \subset V$ is a link that is tangent to the Reeb vector fields for both.  With the projection $\pi:W_{\xi} \rightarrow V$ consider $Z_L = \pi^{-1}(L)$; which is a union of embedded cylinders.  We have
\begin{lemma}
 $Z_L$ is an embedded symplectic submanifold of $(W_{\xi},\omega_{\xi})$.
\end{lemma}
\begin{proof}
$Z_L$ is embedded, and each component of $Z_L$ is $\widehat{J}(\lambda)$-holomorphic for any $\lambda$ with kernel $\xi$ and Reeb vector field tangent to $L$, and any $J \in \J(\xi)$.  By hypothesis on $L$ there exist such $\lambda$.  Since each $\widehat{J}(\lambda)$ is $\omega_{\xi}$-compatible, we see $Z_L$ must be symplectic.
\end{proof}
\begin{lemma}
\label{lem-admcplxstr}
 The subset $\J(\widehat{J}_-,\widehat{J}_+:Z_L)$ of $\J(\widehat{J}_-,\widehat{J}_+)$ such that $TZ_L$ is $J$ invariant is non-empty and contractible.
\end{lemma}

The proof of this fact is standard after the observations that $(1)$ $Z_L$ is an embedded symplectic manifold, and $(2)$ that regardless of $\widehat{J}_{\pm}(\lambda_{\pm})$, $Z_L$ is $J$ holomorphic in $W^{\pm}$.  One considers the space of metrics for which, along $Z_L$, $TZ_L$ and $(TZ_L)^{\omega}$ (the symplectic complement) are orthogonal and then mimics the usual proof that the space of almost-complex structures is non-empty and contractible (e.g.~\cite{Hofer-Zehnder}).  We therefore omit the details.  Denote the space of smooth paths $[0,1] \rightarrow \J(\widehat{J}_-,\widehat{J}_+:Z_L)$ by $\J_{\tau}(\widehat{J}_-,\widehat{J}_+:Z_L)$.  Notice that if $J_1 \in \J(\widehat{J}_-,\widehat{J}_0:Z_L)$ and $J_2 \in \J(\widehat{J}_0,\widehat{J}_+:Z_L)$ then for each $J_R \in \J(J_1,J_2)$ the set $\pi^{-1}(L)$ is $J_R$-holomorphic as well.

\subsection{Review of some facts about Conley-Zehnder Indices} \label{sec-CZindices}

Let $x$ be a closed Reeb orbit, $\overline{x}$ be the simple Reeb orbit with the same image as $x$, and $\Phi$ be a trivialization  $\Phi: S^1 \times \C \rightarrow \overline{x}^* \xi$.  We denote by $\mathrm{CZ}_{\Phi}(x)$ the Conley-Zehnder index of $x$ with respect to the trivialization obtained from $\Phi$ by the obvious degree $m_x$ covering of $S^1 \times \C$ (where $m_x$ is the covering number of $x$ over $\overline{x}$), unless it is specified that $\Phi$ is a trivialization of $x^*\xi$ itself (in which case we will compute with respect to $\Phi$).  If the trivialization is clear in a given context we drop the subscript $\Phi$ from $\mathrm{CZ}_{\Phi}$.

One very useful fact about the Conley-Zehnder index is that it grows almost-linearly for Reeb orbits in dimension $3$ (see e.g. \cite{HWZ_FEF} Theorem $8.3$, or \cite{Hutchings-ECHindex}) in the following precise sense.

\begin{proposition}
\label{prop-Linear Growth}
Let $\gamma$ be a periodic orbit of minimal period $T$, and $\Psi$ a trivialization of the contact structure over that orbit.  The Conley-Zehnder index $\mathrm{CZ}_{\Psi}(\gamma^k)$ is monotone in $k$.  Moreover, one has the following characterization of the index of its coverings $\gamma^k$:
\begin{itemize}
 \item If $\gamma$ is elliptic, there exists a unique $\theta$ so that the complex eigenvalues of $\Phi_T$ are $e^{\pm 2 \pi i \theta}$, and for all covering numbers $k$, $\mathrm{CZ}_{\Psi}(\gamma^k) = 2 \lfloor k \cdot \theta \rfloor + 1$.
 \item If $\gamma$ is hyperbolic then there exists an integer $n$ so that $\mathrm{CZ}_{\Psi}(\gamma^k)$ is $k \cdot n$.
\end{itemize}
\end{proposition}

Another very useful fact is a relationship between the Conley-Zehnder index and the winding of eigensections of the asymptotic operator associated to the orbit $x$.

\begin{definition}  Let $P$ be a simple, closed Reeb orbit, and $k \geq 1$ be an integer.  We define $\alpha^{-}_{\Phi}(P,k)$ to be the winding number of the eigensection associated to the largest negative eigenvalue of the asymptotic operator of $(P,k)$ (the $k$-fold covering of $P$) with respect to $\Phi$, and $\alpha^+_{\Phi}(P,k)$ to be the winding associated to the least positive eigenvalue (see \cite{Siefring2}).
\end{definition}

It is proved in \cite{HWZ_propsII} that the the eigenvalues can be ordered (with multiplicity) so that the winding of the corresponding eigensections is non-decreasing and increases by $1$ every second eigenvalue, and that the Conley-Zehnder index is given by
\[ 
\mathrm{CZ}_{\Phi}(P,k) = \alpha^-_{\Phi}(P,k) + \alpha^+_{\Phi}(P,k) = 2\alpha^-_{\Phi}(P,k) + p(P,k) 
\]
\noindent where $p(P,k)$ is the parity of the Conley-Zehnder index and is equal to either $0$ or $1$ depending on whether the winding numbers of the `extremal' positive/negative eigenvalues agree or disagree:

\noindent In particular, the Conley-Zehnder index is odd if and only if the winding of the eigensections associated with the eigenvalues nearest zero differ.  This is important because it implies by the asymptotic formulas of e.g. \cite{HWZ_propsI, Siefring1} that at odd index orbits that the holomorphic curves necessarily wind around the orbit differently (depending on whether they approach positively or negatively).

Using Proposition \ref{prop-Linear Growth} we see easily that $\alpha^-_{\Phi}(P,k)$ is given by $\lfloor k \theta \rfloor$ if $P$ is elliptic, and $\lfloor \frac{k \cdot n}{2} \rfloor $ if it is hyperbolic (even or odd), and $\alpha^+_{\Phi}(P,k)$ is given by $\lceil k \theta \rceil$ if $P$ is elliptic, and $\lceil \frac{k \cdot n}{2} \rceil $ if it is hyperbolic.

\subsection{Transversality for cylinders}
\label{sec-transversality}

In this section we assume that $\lambda$ is a non-degenerate contact form on $V$.  We will consider $J$-holomorphic cylinders in the symplectization and transversality of the Cauchy-Riemann operator for these cylinders.  Here $J = \widehat{J}$ will be a cylindrical almost-complex structure.  The constructions of the chain complexes in Section \ref{sec-CCHint} requires the Cauchy-Riemann operator to be transverse on enough moduli spaces.  In particular we will prove the following transversality Theorem.  In order to state it, we recall first that a Reeb orbit is called \emph{SFT-good} if it is not an even cover of another orbit with odd Conley-Zehnder index.

\begin{theorem} \label{thm-transversality}
  By \cite{Dragnev} there is a residual subset $\mathcal{J}_{\mathrm{gen}}(\lambda)$ of $\J(\lambda)$ such that for all somewhere injective $J$-holomorphic curves the Cauchy-Riemann operator is transverse for $J \in \J_{\mathrm{gen}}(\lambda)$.  Let $J \in \J_{\mathrm{gen}}(\lambda)$:
\begin{enumerate}
  \item  At a $J$-holomorphic cylinder $U$ of index $\mathrm{Ind}(U) \leq 1$, the Cauchy-Riemann operator is transverse at $U$.  In particular the only index $\leq 0$ cylinders are trivial cylinders.
  \item  At a $J$-holomorphic cylinder $U$ of index $\mathrm{Ind}(U) = 2$, and such that both asymptotic orbits are SFT-good, the Cauchy-Riemann operator is transverse at $U$.
\end{enumerate}
Therefore the corresponding moduli spaces (modulo the free smooth $\R$-action) consist of isolated points (index $0$ or $1$ case) or of isolated intervals or circles (index $2$ case).
\end{theorem}

\subsubsection{Some Inequalities}

Let $\pi$ denote the projection of $TV$ to the contact structure $\rm{ker}(\lambda)$.  Given a finite energy surface $U = (a,u)$\footnote{We use the identification $\psi_{\lambda}: W_{\xi} \rightarrow \R \times V$ and let $a = \pi_{\R} \circ U$, $u = \pi_{V} \circ U$.}, the section $\pi \circ Tu$ of $\Omega^{0,1}(T\Sigma,u^*\xi)$ is quite useful.  It is proved in \cite{HWZ_propsII} that either it vanishes identically or it has isolated zeros.  We let $\rm{wind}_{\pi}(U)$ denote the oriented count of zeros of this section, which is proved in \cite{HWZ_propsII} to be non-negative.  Also, we will use the following topological quantities:

\begin{itemize}
 \item Denote by $\Gamma$ the set of punctures of the map $U$ and by $\Gamma^{\pm}$ the subsets of positive and negative punctures.  Also, $\# \Gamma_{\mathrm{even}}$ denotes the number of punctures asymptotic to orbits of even Conley-Zehnder index, and $\# \Gamma_{\mathrm{odd}}$ denotes the number of punctures asymptotic to orbits of odd Conley-Zehnder index.
 \item  
\begin{equation} \label{eq-CZ(U)}
 \mathrm{CZ}(U) = \sum_{z^+ \in \Gamma^+} {\mathrm{CZ}}(z^+) -\sum_{z^- \in \Gamma^-} {\mathrm{CZ}}(z^-)
\end{equation}
\noindent (this requires a choice of trivialization of $U^*\xi$, but is independent of that choice).
 \item  $\mathrm{Ind}(U) = \mathrm{CZ}(U) - \chi(\Sigma) + \# \Gamma$, and $2 c_N(U) = \mathrm{Ind}(U) - \chi(\Sigma) + \# \Gamma_{\mathrm{even}}$
\end{itemize}

\noindent
We will use the following theorem, proved in \cite{HWZ_propsII}, to show that certain curves are immersed in order to apply the automatic transversality Theorem $2.8$.

\begin{theorem} (\cite{HWZ_propsII}, Theorem $5.8$)
\label{thm-pitu}
Suppose $\pi \circ T u$ does not vanish identically.  Then it vanishes only a finite number of times, and $0 \leq \mathrm{wind}_{\pi}(U) \leq c_N(U)$.
\end{theorem}


\subsubsection{Review of some results about transversality}
\label{sec-knowntransresults}

The Fredholm theory here was developed in \cite{Dragnev}.  There the author defines a Banach manifold 
\[ \B_g^{1,p;d}(x_1,\dots,x_k;y_1,\dots,y_l),\] 
\noindent where $p > 2$, of $W^{1,p}_{\rm{loc}}$ maps from a surface of genus $g$ with $k$ `positive' punctures and $l$ `negative' punctures into $\R \times V$ which converge exponentially with weight $d$ in $W^{1,p}$ to the periodic orbits $x_1,\dots,x_k$ at positive punctures and to $y_1,\dots, y_l$ at negative punctures.  We will really only need the cases $g=0, k = 1, l = 1$.  Given this structure we have a section

\begin{equation*}
 \overline{\partial}_J(U) = d U + J \circ d U \circ j
\end{equation*}

\noindent
defined over $\B$ into a smooth bundle $\E$ over $\B$ with fibers

\begin{equation*}
 \E = \bigcup_{U \in \B} \{ U \} \times L^{p,\delta}(\Omega^{0,1}(T \dot{\Sigma} \otimes_{\C} U^*TW))
\end{equation*}

At a zero of this section, $U$, we define $F_{U}$ to be the projection of $D \overline{\partial}_J (U)$ to the vertical part $T_{U} \E \cong T_{U}\B \oplus \E_{U}$.  It is then proved in \cite{Dragnev} that $F_{U}$ is a Fredholm operator.  We say the Cauchy-Riemann operator is \emph{transverse} or \emph{regular} at $U$ (or $U$ is a transverse solution, etc.) if $F_U$ is a surjective operator, and that a subset of solutions is \emph{transverse} if each element is a transverse solution.  Its index (computed in \cite{Dragnev}) gives the dimension of parameterized $(j,J)$-holomorphic maps near $U$ when $U$ is transverse.  The dimension of the moduli space (allowing $j$ to vary and dividing by automorphisms of the domain) is then given by the topological quantity $\mathrm{Ind}(U) = \mathrm{CZ}(U) + \left( \frac{\mathrm{dim}(\R \times V)}{2} -3 \right) \left(\chi_{\Sigma} - \# \Gamma \right)$ which coincides with the definition of $\rm{Ind}(U)$ given earlier since $\frac{\rm{dim}(\R \times V)}{2} = 2$.

We cite a special case of Theorem $0.1$ in \cite{Wendl2008} (valid for $\rm{dim}(\R \times V) = 4$):

\begin{theorem}
\label{thm-autotrans}
\cite{Wendl2008}
 Suppose that $U$ is an immersed finite energy surface.  Letting $g$ denote the genus of the domain of $U$, the linearized operator is surjective if
\begin{equation*}
  \rm{Ind}(U) \geq 2g + \# \Gamma_{\mathrm{even}} - 1.
\end{equation*}
\end{theorem}

We also need results for somewhere-injective curves.  Following the notation of \cite{Dragnev}, define the set $\M = \mathcal{M}(x_1, \dots, x_n;y_1\dots,y_n)$ the set $(C,J)$ of pairs consisting of non-parametrized curves $C$ and compatible almost-complex structures $J \in \J$ (where $\J = \J(\lambda)$ or $\J(\widehat{J}_-,\widehat{J}_+)$) for which $C$ is $J$-holomorphic (for some parametrization) positively asymptotic to the Reeb orbits $x_i$ and negatively to the $y_j$.  In case $\J = \J(\lambda)$ one quotients by the free $\R$-action on solutions as well.  We have the following from \cite{Dragnev} (Theorem 1.8 and its Corollary) which holds in any dimension:

\begin{theorem}
\label{thm-genericJ}
 The set $\mathcal{M}$ carries a structure of a separable Banach manifold.  The projection map $\eta: \mathcal{M} \rightarrow \J$, $\eta(C,J) = J$, is a Fredholm map with Fredholm index $\mathrm{Ind} (C) = \mathrm{Ind}(U)$, where $U:(S,j) \backslash \Gamma \rightarrow \R \times V$ parametrizes $C$.

For regular values $J$ of $\eta$, $\M_J = \eta^{-1}(J)$ is a smooth, finite dimensional manifold whose dimension agrees with the above index $\mathrm{Ind}(C)$.  By the Sard-Smale theorem, the set of regular values is a residual set.  Consequently there is a residual set $\J_{\mathrm{gen}} \subset \J$ of compatible almost-complex structures such that for every $J \in \J_{\mathrm{gen}}$ if $U : S \backslash \Gamma \rightarrow \R \times V$ is a somewhere injective finite energy surface for $J$, then $\rm{Ind}(U) \geq 1$ provided $\pi \circ Tu$ does not vanish identically.
\end{theorem}

\noindent Theorem \ref{thm-genericJ} can be proved if $\J = \J(\widehat{J}_-(\lambda_-),\widehat{J}_+(\lambda_+):Z)$ as well because any somewhere injective curve for such a $J$ will have a point of injectivity in $W \backslash (W^+ \cup W^- \cup Z)$ with the exceptions of $(1)$ components of $Z$, and $(2)$ holomorphic curves for $\widehat{J}(\lambda_+)$ contained entirely in a cylindrical end $W^{\pm}$.  We will not be concerned about transversality of such curves anyway and therefore they can be safely ignored.  This observation will be needed in section \ref{sec-CCHint}.

\subsubsection{Proof of Theorem \ref{thm-transversality}}

Theorem \ref{thm-transversality} is a direct consequence of

\begin{proposition}
  \label{prop-transversality}
  There is a residual subset $\mathcal{J}_{\rm{gen}}$ of $d\lambda$-compatible almost-complex structures on $\xi$ such that:
\begin{enumerate}
 \item  At any somewhere injective finite energy solution, the Cauchy-Riemann operator is transverse, by \cite{Dragnev}.
 \item  For any compatible $J$ (not necessarily in $\J_{\rm{gen}}$), if $U \in \M_{J}(x;y)$ is of index $1$, then the linearization of the Cauchy-Riemann operator is surjective at $U$.
 \item  Let $J$ be any compatible almost-complex structure, and suppose $U$ is a $J$-holomorphic finite-energy cylinder.  Then its index is at least $0$.
 \item   For $J \in \mathcal{J}_{\rm{gen}}$, if $U \in \M_{J}(x;y)$ is an index $2$ cylinder, and both asymptotic orbits are SFT-good, then the linearization of the Cauchy-Riemann operator is surjective at $U$.
 \item  For $J \in \mathcal{J}_{\rm{gen}}$, all index zero holomorphic cylinders are trivial cylinders.  The linearized Cauchy-Riemann operator is surjective at any trivial cylinder.
\end{enumerate}
\end{proposition}

This is in fact what we will prove.

\begin{proof}

To begin the proof, take $\J_{\rm{gen}}$ to be as in Theorem \ref{thm-genericJ}.  

\noindent \textit{$(1)$}  The first item is an assertion in Theorem \ref{thm-genericJ}.  

\noindent \textit{$(2)$}  The map $U$ is an immersion if $\rm{wind}_{\pi}(U) = 0$.  For finite energy cylinders of index $1$, Theorem \ref{thm-pitu} implies $0 \leq \rm{wind}_{\pi}(U) \leq 1 - 2(2) + 1 + 2(1) = 0$ (since there must be one odd puncture and one even puncture).  Thus, finite energy cylinders of index $1$ are always immersed.  We have $g = 0$ and $\# \Gamma_{\mathrm{even}} = 1$ (since the index difference is $1$, there is one even and one odd puncture), so by Theorem \ref{thm-autotrans} the linearized Cauchy-Riemann operator is surjective.

\noindent \textit{$(3)$}  This was already observed in \cite{HWZ_propsII}; we include the short proof.  $\pi \circ T u$ vanishes identically on a cylinder if and only if the cylinder is a trivial cylinder.  If a non-trivial finite energy cylinder has index less than $0$, then by Theorem \ref{thm-pitu} $0 \leq 2\rm{wind}_{\pi}(U) < 0 - 4 + \# \Gamma_{\mathrm{odd}} + $ $2 \# \Gamma_{\mathrm{even}} \leq 0$, which is impossible.  

To prove $(4)$ we will need the following lemma.

\begin{lemma}
\label{lemma-even covers}
 Suppose $U$ is a cylinder of index $2$ with asymptotic limits $x,y$.  If $x$ and $y$ have even Conley-Zehnder index and are SFT-good, then $U$ is simply covered.
\end{lemma}

\begin{proof}
Let $k$ be the covering number of $U$, so that $U = U' \circ \tau_k$ where $U$ is a simple curve and $\tau_k$ is the degree $k$ holomorphic cover of $\R \times S^1$.  Choose a trivialization $\Phi'$ for $U'^* \xi$, which induces trivializations $\Phi$ of $\xi$ over $U$, $x$ and $y$.  We then have $\mathrm{CZ}_{\Phi}(x) - \mathrm{CZ}_{\Phi}(y) = 2$ by the index formula.  Let $\mathrm{CZ}_{\Phi}(x) = 2 l + 2$, so $\mathrm{CZ}_{\Phi}(y) = 2 l$.  Letting $x'$ and $y'$ denote the asymptotic limits of $U'$, so that $x,y$ are $k$-covers of $x', y'$, by Proposition \ref{prop-Linear Growth} we have $k | 2l + 2$, and $k | 2l$, from which it follows that $k = 2$ or $1$.  Suppose $k = 2$.  Then $\mathrm{CZ}_{\Phi'}(x') = l + 1, \mathrm{CZ}_{\Phi'}(y') = l$.  Therefore, one of these has odd Conley-Zehnder index, so either $x$ or $y$ is a `bad' orbit.  Since both orbits are assumed SFT-good, $k=1$.
\end{proof}

\noindent \textit{$(4)$}  The parity of the Conley-Zehnder indices are either both even or both odd.  If they are both even, then by Lemma \ref{lemma-even covers} all curves in the index $2$ component of the moduli space are simple.  Transversality of the index $2$ component of the moduli space for $J \in \mathcal{J}_{\rm{gen}}$ follows from Theorem \ref{thm-genericJ}.

In the case of an index-$2$ cylinder with odd asymptotic orbits, we observe

\begin{align*}
 0 \leq \rm{wind}_{\pi}(U) \leq 2 - 2(2) + 2(1) + 2(0) = 0\\
\rm{Ind}(U) = 2 \geq 2g + \# \Gamma_{\mathrm{even}} - 1 = 2(0) + 0 - 1
\end{align*}

\noindent
so automatic transversality applies by Theorem \ref{thm-pitu} and Theorem \ref{thm-autotrans}, in particular for any $J$ chosen for the case of even, SFT-good orbits.

\noindent \textit{$(5)$}  Choose any $J \in \mathcal{J}_{\rm{gen}}$.  For simply covered index $0$ cylinders, the inequality of Corollary \ref{thm-genericJ} implies $\pi \circ Tu$ vanishes identically, which implies the cylinder is trivial.  If a cylinder is multiply covered of index $0$, the underlying simple cylinder has index $0$ - by the monotone growth of the index of a cylinder under covering.  Hence, it is a cover of a trivial cylinder so is trivial itself.  Transversality is proved  as in \cite{Schwarz}.
\end{proof}

\section{Intersections and compactness}

\subsection{Intersections}
\label{sec-intersections}

We recall the intersection theory described in \cite{Siefring2}.

\begin{definition}
 A \emph{stable Hamiltonian structure} \cite{Siefring2, MR2284048} on a $3$-dimensional manifold $V$ is a pair $\mathcal{H} = (\lambda,\omega)$ such that
\begin{enumerate}
 \item $\lambda \wedge \omega$ is a volume form for $V$
 \item $\omega$ is closed
 \item $d\lambda|_{\rm{ker}\omega} \equiv 0$
\end{enumerate}
\noindent  It defines a \emph{$\mathcal{H}$-Reeb vector field} by $X \neg \lambda \equiv 1, X \neg \omega \equiv 0$.
\end{definition}

For example, if $V$ is a $3$-manifold with a contact form $\lambda$, the associated stable Hamiltonian structure is $\mathcal{H}_{\lambda} = (\lambda,d\lambda)$.

Let $W$ be a $4$-manifold.  A cylindrical Hamiltonian structure on $W$ is the data $(\Phi,V,\mathcal{H})$ where $\Phi: W \rightarrow \R \times V$ is a diffeomorphism and $\mathcal{H}$ is a stable Hamiltonian structure on $V$.  A positive (resp. negative) Hamiltonian structured end is the data $(W^{\pm},\Phi^{\pm},V^{\pm},\mathcal{H}^{\pm})$ where $W^{\pm}$ is a subset, $\Phi^{\pm}:W^{\pm} \rightarrow \pm[0,\infty) \times V^{\pm}$ is a diffeomorphism and $\mathcal{H}^{\pm}$ is a stable Hamiltonian structure on $V^{\pm}$.  We will usually refer only to $W^{\pm}$ with the rest of the data implicit.  A cobordism is a manifold $W$ such that $W \backslash W^{\pm}$ is compact oriented with boundary.

\begin{example}
 Suppose $\lambda_{\pm}$ are contact forms with kernel $\xi$ on $V$, and $\lambda_- \prec \lambda_+$ in $W = W_{\xi}$.  We can equip $W_{\xi}$ with the Hamiltonian structured ends $W^{\pm} = W^{\pm}(\lambda_{\pm}), V^{\pm} = V$, $\mathcal{H}^{\pm} = (\lambda_{\pm},d\lambda_{\pm})$
\[
 \Phi^{\pm} (\lambda) = \left(\log \frac{\lambda}{\lambda_{\pm}(\pi(\lambda))}, \pi(\lambda) \right)
\]
\end{example}

Suppose $W_i$ are two such cobordisms.  We say that we can concatenate $W_1$ and $W_2$ if $(V^+_1,\mathcal{H}_1^+) = (V^-_2,\mathcal{H}_2^-)$.  Then we can define the manifold $W_1 \odot W_2$ by compactifying the positive end of $W_1$ with $\{\infty\} \times V_1^+$, the negative end of $W_2$ with $\{ - \infty \} \times V_2^-$ (using the diffeomorphisms $\Phi^{\pm}$ to do so), and identifying them.  This can be generalized to $n$-fold concatenations $W_1 \odot \dots \odot W_n$ for any $n \geq 2$ as long as one can concatenate $W_i$ and $W_{i+1}$ for each $i$.

Suppose $W_1$ is a cylindrical Hamiltonian structured manifold i.e.~$W_1 \cong \R \times V$, and that we can form $W_1 \odot W_2$.  Then obviously $W_1 \odot W_2$ is homeomorphic to $W_2$ and is equipped with the same Hamiltonian ends.  This identification in convenient and we will use it often in the following without comment.

\subsubsection{A homotopy invariant intersection number}  

Let $\mathcal{H}$ be a Hamiltonian structure on $V$, and let $W \cong \R \times V$.  Let $(\Sigma,j,\Gamma,U, W)$ be a $C^1$ map $U: \Sigma \backslash \Gamma \rightarrow W$.  For a closed $T$-periodic Reeb orbit $\delta$, denote by $Z_{\delta}$ the trivial cylinder over $\delta$:
\[
 Z_{\delta}(s,t)= (Ts, \delta(Tt))
\]

\begin{definition}  \cite{Siefring2} Suppose $\gamma$ is a simple $T$-periodic orbit for $X_{\mathcal{H}}$.  We say $U$ is asymptotically cylindrical over $\gamma^m$ at $z \in \Gamma$ if there is a holomorphic embedding $\phi:[0,\infty) \times S^1 \rightarrow \Sigma \backslash \{ z \}$ with $\phi(s,t) \rightarrow z$ (as $s \rightarrow \infty$) so that 
  \[
    \lim_{c \rightarrow \infty} (-m T c) * U \circ \phi (s + c,t)|_{[0,\infty) \times S^1} = Z_{\gamma^m}|_{[0,\infty) \times S^1}
  \]
\noindent  (with convergence in $C^1([0,\infty)\times S^1, W)$, and where $c * U$ denotes the action by translation of the $\R$ co-ordinate by $c$)).  We say $U$ is asymptotically cylindrical if at each puncture $z \in \Gamma$ $U$ is asymptotically cylindrical to some $\mathcal{H}$-Reeb orbit $\gamma^m$.

If $U$ is a map into a $4$-manifold $W$ with Hamiltonian structured ends $\mathcal{H}^{\pm}$, then we say $(\Sigma,j,\Gamma,U,W)$ is asymptotically cylindrical if for each puncture $z \in \Gamma$ there is a neighborhood $O_z$ of $z$ such that $U|_{O_z}$ is contained in one of the ends and is asymptotically cylindrical as defined above.
\end{definition}

We denote by $C^{\infty}_{g,p_+,p_-}(V,\mathcal{H})$ the set of smooth asymptotically cylindrical maps of genus $g$ with $p_+$ positive punctures and $p_-$ negative punctures, and set $C^{\infty}(V,\mathcal{H})$ to be the union over all $g,p_-,p_+ \geq 0$.  Similarly in \cite{Siefring2} the author defines asymptotically cylindrical maps in a $4$-manifold $W$ with
Hamiltonian structured ends $W^{\pm}$ and denotes these by $C^{\infty}(W,\mathcal{H}^+,\mathcal{H}^-)$.  We will extend this notation to include continuous, piecewise smooth maps which are smooth on neighborhoods of the punctures and satisfy the same asymptotic conditions.

\begin{definition}  Suppose $U,V$ are asymptotically cylindrical maps, and $z$ is a puncture for $V$ asymptotic to $\gamma^m$.  Choose a trivialization $\Phi:S^1 \times \C \rightarrow \gamma^* \xi$.  There is a neighborhood $O_z$ of $z$ that is mapped entirely into a cylindrical end and so that
\[
 V(\phi(s,t)) = (ms, \exp_{\gamma^m(t)} h(s,t))
\]
\noindent Let $\beta$ be a cut-off function supported in $O_z$ which is identically $1$ on a smaller neighborhood of $z$.  Consider the perturbation
\[
 V'(\phi(s,t)) = \left( ms,\exp_{\gamma^m(t)} \left[ h(s,t) + \beta(\phi(s,t)) \cdot \Phi(mt) \cdot \epsilon \right] \right)
\]
\noindent which is well-defined for all $\epsilon > 0$ sufficiently small.  Let $V_{\Phi,\beta,\epsilon}$ be the map that results from making such a perturbation at each puncture of $V$.  Then set $\iota_{\Phi}(U,V) = \rm{int}(U, V_{\Phi,\beta,\epsilon})$.  It depends only on the homotopy classes of the maps and the trivializations $\Phi$, and is symmetric in the arguments $U,V$ (see \cite{Siefring2}).
\end{definition}

\begin{definition}
Let $P$ be a closed, simple Reeb orbit, and let $k_1,k_2$ be two positive integers.
\begin{align} 
\label{eq-alpha+}  \Omega^{+}_{\Phi}(P,k_1,k_2) &=  k_1 k_2 \max \left \{ \frac{ \alpha_{\Phi}^-(P,k_1)}{k_1} , \frac{ \alpha_{\Phi}^-(P,k_2)}{k_2}\right \} \\
\label{eq-alpha-} \Omega^{-}_{\Phi}(P,k_1,k_2) &=  k_1 k_2 \min \left \{ \frac{ \alpha_{\Phi}^+(P,k_1)}{k_1} , \frac{ \alpha_{\Phi}^+(P,k_2)}{k_2}\right \}
\end{align}

Let $z$ and $z'$ be punctures of $U$ and $V$ respectively.  If $U$ and $V$ are asymptotic to covers of different Reeb orbits, then we set $\Omega_{\Phi}(z,z') = 0$.  If $U$ and $V$ are positively asymptotic to covers $P^{k_1}, P^{k_2}$ of the same Reeb orbit $P$ at $z$ and $z'$ then we set $\Omega_{\Phi}(z,z') = \Omega_{\Phi}^{+}(P,k_1,k_2)$; if they are negatively asymptotic at $z$ and $z'$ to covers $P^{k_1}, P^{k_2}$ of the same Reeb orbit $P$ then we set $\Omega_{\Phi}(z,z') = \Omega_{\Phi}^{-}(P,k_1,k_2)$.  %
With this notation, set
\[
 \Omega_{\Phi}(U,V) = \sum_{\substack{z \in \Gamma_{U}^+ \\ z' \in \Gamma_{V}^+}} \Omega_{\Phi}(z,z') - \sum_{\substack{z \in \Gamma_{U}^- \\ z' \in \Gamma_{V}^-}} \Omega_{\Phi}(z,z')
\]
\end{definition}

It follows from Proposition $4.1$ and Lemma $3.4$ of \cite{Siefring2} that $\iota_{\Phi}(U,V) + \Omega_{\Phi}(U,V)$ depends only on the homotopy classes of $U,V$.

\begin{definition}  For a pair of punctures $z,z'$ of $U,V$, if $U,V$ are asymptotic to covers of different orbits at $z,z'$ then set $\Delta(z,z') = 0$; else, let $P$ denote the underlying simple orbit, $k_1(z),k_2(z')$ the respective covering numbers of the orbit, and set

\[
 \Delta(z,z') = \Omega_{\Phi}^-(P,k_1(z),k_2(z')) - \Omega_{\Phi}^+(P,k_1(z),k_2(z')) \geq 0
\]

\noindent  It follows from equations \eqref{eq-alpha+}, \eqref{eq-alpha-} that this difference is indeed trivialization independent.  Sum this quantity over all pairs of punctures to get

\[
 \Delta(U,V) = \sum_{\substack{z \in \Gamma^+_U \\ z' \in \Gamma^+_V}} \Delta(z,z') + \sum_{\substack{z \in \Gamma^-_U \\ z' \in \Gamma^-_V}} \Delta(z,z') \geq 0
\]
\end{definition}

\begin{definition}  Let $[U], [V]$ be homotopy classes of asymptotically cylindrical maps with representatives $U,V$ (resp.).  Set
 \[
   [U] * [V] = \iota_{\Phi}(U,V) + \Omega_{\Phi}(U,V) + \frac{1}{2} \Delta(U,V)
 \]
 \noindent  By the above comments this only depends on the homotopy classes $[U],[V]$ of the maps $U,V$ in $C^{\infty}(V,\mathcal{H})$ (or $C^{\infty}(W,\mathcal{H}^+,\mathcal{H}^-)$) and is therefore well-defined.
\end{definition}

Asymptotically cylindrical maps $U_i$ in $W_i$ can be concatenated\footnote{The concatenation is well-defined only up to a Dehn twist of the domain $\Sigma_1 \circ \Sigma_2$, since a gluing of these domains may require a choice of asymptotic markers.  This data is determined when the concatenation occurs as a holomorphic building, which is the only case that interests us.  Regardless, different choices of asymptotic markers yield the same intersection number (see the proof of Proposition $4.3$ of \cite{Siefring2}).} if the asymptotics match in the obvious way so that the maps can be glued together (after choosing asymptotic markers) to form a (piecewise smooth) $C^0$ map $U_1 \odot U_2$ in $W_1 \odot W_2$.  One can also form $n$-fold concatenations $U_1 \odot \dots \odot U_n$ in the obvious way to make a map into $W_1 \odot \dots \odot W_n$; every holomorphic building defines such a concatenation.  Proposition $4.3$ of \cite{Siefring2} states (except for the last item which differs slightly; see the remark following the statement):

\begin{proposition}
\label{prop-intersectionproperties}
  Let $W,W_1,\dots,W_n$ be $4$-manifolds with Hamiltonian structured cylindrical ends.  Then:
  \begin{enumerate}
    \item  If $(\Sigma,j,\Gamma,W,U)$ and $(\Sigma',j',\Gamma',W,V)$ $\in C^{\infty}(W,\mathcal{H}^+,\mathcal{H}^-)$ then $[U]*[V]$ depends only on the homotopy classes $[U],
    [V]$.
     \item $[U]*[V] = [V] * [U]$
     \item Using ``$+$'' to denote disjoint union of maps, we have $[U_1 + U_2] * [V] = [U_1] * [V] + [U_2] * [V]$.
     \item If $U_1 \odot \dots U_n$ and $V_1 \odot \dots V_n$ are concatenations of asymptotically cylindrical maps in $W_1 \odot \dots W_n$ then
     \[
        [U_1 \odot \dots U_n] * [V_1 \odot \dots V_n] = \sum_{i = 1}^n [U_i] * [V_i]
     \]
  \end{enumerate}
\end{proposition}

\begin{remark}
  The last assertion differs from the one made in \cite{Siefring2}; however, it is shown in the proof that the difference
  \[
    [U_1 \circ U_2] \tilde{*} [V_1 \circ V_2] - [U_1] \tilde{*} [V_1] - [U_2] \tilde{*} [V_2]
  \]
  \noindent (where $\tilde{*}$ denotes the intersection number minus the term $\frac{1}{2} \Delta(U,V)$, which is what is used in \cite{Siefring2}) is equal to
\[
 \Delta = \Delta^+((U)_1,(V)_1) = \Delta^-((U)_2,(V)_2)
\]
\noindent where we use $\Delta^{\pm}$ to denote the sum of the terms in the definition of $\Delta$ corresponding to the positive (resp. negative) punctures only (in particular $\Delta = \Delta^+ + \Delta^-$).  It is easy to see that we get the above identity from our definition of ``$[U]*[V]$''.

The statement in \cite{Siefring2} is extended to $n$-fold concatenations above.
\end{remark}

Suppose $U: \Sigma \backslash \Gamma \rightarrow W$ is in $C^{\infty}(W,\mathcal{H}^+,\mathcal{H}^-)$.  Considering the oriented blow-up  $\overline{\Sigma}$ \cite{BEHWZ}, because $U$ is asymptotically cylindrical it extends over $\overline{\Sigma}$ to define a map $\overline{U}: \overline{\Sigma} \rightarrow \overline{W}$, with the property that if $B$ is a boundary component of $\overline{\Sigma}$, then $\overline{U}|_B$  is a closed Reeb orbit in $\partial \overline{W}$.  Let us consider the subclass $C^0(\overline{W},\mathcal{H}^+,\mathcal{H}^-)$ of those $\overline{U} \in C^0(\overline{\Sigma}, \overline{W})$ such that for each boundary component $B$ of $\overline{\Sigma}$, $\overline{U}|_B$ is a closed Reeb orbit in $(\partial^{\pm} \overline{W} \cong V^{\pm},\mathcal{H}^{\pm})$.  By standard arguments the intersection number extends to the class $C^0(\overline{W},\mathcal{H}^+,\mathcal{H}^-)$ and is homotopy invariant within that class.

In particular, the above discussion extends the intersection number to holomorphic buildings: a building $(U)$ defines a map $\overline{(U)} \in C^0(\overline{W},\mathcal{H}^+,\mathcal{H}^-)$, which we use to set
\[
[(U)] * [(V)] \quad \dot{=} \quad [\overline{(U)}] * [\overline{(V)}]
\]
Once the intersection number is extended this way, it is clear that
\begin{lemma} \label{lem-continuityofint}
  Suppose $U_k \rightarrow (U)$ and $V_k \rightarrow (V)$ in the SFT-sense.  Then $\lim [U_k] * [V_k] = [(U)] * [(V)]$.
\end{lemma}
\begin{proof}
  By property \texttt{C3} of SFT-convergence, $G \circ \overline{U_k} \circ \phi_k \rightarrow \overline{(U)}$ uniformly, and similarly $G \circ \overline{V_k} \circ \phi'_k \rightarrow \overline{(V)}$ uniformly.  Therefore they represent the same homotopy classes in $C^0(\overline{W},\mathcal{H}^+,\mathcal{H}^-)$ for $k$ sufficiently large and therefore 
\[
\lim [U_k] * [V_k] = \lim [\overline{U_k}] * [\overline{V_k}] = [\overline{(U)}] * [\overline{(V)}] = [(U)] * [(V)]
\]
\end{proof}

In a concatenation $W_1 \odot \dots \odot W_n$, if both ends of each $W_i$ correspond to contact forms for which $L$ consists of closed orbits, then the cylinder $\pi^{-1}(L) = Z_L$ is asymptotically cylindrical for each $W_i$, and the concatenated map $Z_L \odot \dots Z_L$ identifies with $Z_L$ under the homeomorphism $W_1 \odot \dots W_k \cong W_{\xi}$.  So it makes sense to speak about the intersection of a holomorphic building $(U)$ and the cylinder $Z_L$ if all the ends of all pieces in the concatenation in which $(U)$ lives are such that $L$ is a closed orbit for the form defining the Hamiltonian structure for that end.  Moreover, if we index the levels of $(U)$ by $U^i$ i.e.~ $(U) = U^1 \odot \dots \odot U^k$ then
\[
 [(U)] * [Z_L] = \sum_{i = 1}^k [U^i] * [Z_L]
\]
\noindent by the level-wise additivity property (Property $(4)$ of Proposition \ref{prop-intersectionproperties}) of the intersection number.

\subsubsection{Positivity and local computations}

It is shown in \cite{Siefring2} that if $U,V$ are both pseudoholomorphic asymptotically cylindrical with no components sharing identical images then $[U]*[V] \geq \frac{1}{2}\Delta(U,V)$, with strict inequality if there is an interior intersection.  Moreover we see that if $U,V$ are both asymptotic to covers of an orbit $\gamma$ either of which has odd Conley-Zehnder index then $[U]*[V] > 0$.

For the following, suppose $U,V$ are asymptotically cylindrical holomorphic maps.  The asymptotics formula and the similarity principle implies that if $U,V$ have no components on which they are covers of the same holomorphic curve then the number of intersection points is finite.  Let $\rm{int}(U,V)$ be the sum of all local intersections, which is non-negative by e.g.~\cite{MR1114456, MR1314031}.

\begin{theorem}[\cite{Siefring2} Theorem $4.4$]
\label{thm-posintersections}
Suppose $U,V$ are asymptotically cylindrical pseudoholomorphic maps in an almost-complex cobordism with cylindrical ends, with no common components.  Then $[U]*[V] = \rm{int}(U,V) + \delta_{\infty}(U,V) + \frac{1}{2}\Delta(U,V)$ and $\delta_{\infty}(U,V) \geq 0$.
\end{theorem}

The number $\delta_{\infty}(U,V)$ is described in \cite{Siefring2} in terms of local quantities for holomorphic maps.  We refer to \cite{Siefring2} for the definition.  For later use we now recall what the quantity $\delta_{\infty}(U,V)$ is in the particular case $V = Z_{L_i}$ where $L_i$ is a closed Reeb orbit for both $\lambda_{\pm}$ in a cobordism of the form $\overline{W}(\lambda_-,\lambda_+) \subset W_{\xi}$ and $Z_{L_i} = \pi^{-1}(L_i)$.
In the following formula, let $L_i \sim U(w)$ mean that $U$ is asymptotic to a cover of $L_i$ at $w$, let $k_w$ be the covering number of the $L_i$ to which $U$ is asymptotic at $w$, let $w_{\lambda,\Phi}(U,w)$ be the winding number of the loop $(\Phi^{k_w})^{-1}e_U$, where $e_U$ is the eigensection of the asymptotic operator appearing in the asymptotic formula of \cite{Siefring1} for $U$ at the puncture $w$; let $\alpha_{\lambda,J,\Phi}^{\pm}(L_i,k)$ be the extremal winding number of the asymptotic operator associated with $(L_i^k,\lambda,J)$ and trivialization $\Phi^k$, and finally let $p_{\lambda}(x) \in \{0, 1 \}$ be the parity of the orbit $x$ with respect to the form $\lambda$.  Then

\begin{align}
\label{eq-deltainfty} \begin{split}
\delta_{\infty}(U,Z_{L_i}) &= \sum_{\substack{w \in \Gamma^+_U \\ L_i \sim U(w)}} - w_{\lambda_+,\Phi}(U,w) + \alpha_{\lambda_+,J_+,\Phi}^-(L_i,|k_w|)\\
			   &+ \sum_{\substack{w \in \Gamma^-_U \\ L_i \sim U(w)}} + w_{\lambda_-,\Phi}(U,w) - \alpha_{\lambda_-,J_-,\Phi}^+(L_i,|k_w|) 
\end{split} 
\end{align}


\subsubsection{Moduli spaces}

Let $(W,J)$ be a symplectic $4$ manifold with Hamiltonian structured ends and compatible $J$ adjusted to the ends.  We use $\M_J(x;y_1,\dots,y_k)$ to denote the moduli space of finite energy $J$-holomorphic spheres with one positive asymptotic orbit $x$ and perhaps several negative asymptotic orbits $y_1,\dots y_k$ (modulo reparametrizations of the domain).  Usually we are interested in moduli space of cylinders with one positive and one negative end, $\M_J(x;y)$.  If it is clear we may drop the subscript.  If $J$ is a cylindrical almost-complex structure $J =\widehat{J}(\lambda)$ then we will actually mean the moduli space modulo the further $\R$-action by translations.  We assume that all Hamiltonian structures considered are non-degenerate, so we have the asymptotic formula \cite{HWZ_propsI, BEHWZ} which implies all solutions are asymptotically cylindrical and the intersection theory described above applies.

Given a holomorphic curve $Z$, we denote
\[
\M_J(x;y_1,\dots,y_k: Z) = \left\{ [U]  | [U] * [Z] = 0 \right\} \subset \M_J(x;y_1, \dots y_k)
\]
\noindent For homotopies $l \mapsto J_l$, we write 
\[
  \calN_{\{ J_l \}}(x;y_1,\dots,y_k) = \bigcup_{l \in [0,1]} \{l\} \times \M_{J_l}(x;y_1,\dots,y_k) \}
\] 
\noindent and
\[ 
 \calN_{ \{ J_l \} }(x;y_1,\dots,y_k:Z) = \bigcup_{l \in [0,1]} \{l\} \times \M_{J_l}(x;y_1,\dots,y_k :Z) \}
\]

In the following, $Z$ will almost always be $Z_L = \pi_V^{-1}(L)$ (recall $\pi_V:W_{\xi} \rightarrow V$ is the projection $T^*V \rightarrow V$ restricted to $V$), though we do not specialize yet.  

\begin{lemma} \label{lem-zerointstructure}
 Suppose $Z$ is a finite energy surface with all asymptotic orbits elliptic (e.g. $Z_L = \pi^{-1}(L)$ under condition $(E)$), and $[U] \in \M_J(x;y_1,\dots,y_n)$ is such that no component has image contained in a component of $Z$.  Then $[U] \in \M_J(x;y,\dots:Z)$ if and only if $U$ and $Z$ never intersect and are never asymptotic to covers of the same asymptotic orbit with the same sign (in the case $Z = Z_L$ this is the same as saying $x,y \notin D_{\pm}'$ and $U$ never intersects $Z$).
\end{lemma}

\begin{proof}
 If $U$ intersects $Z$, by positivity of intersections $U \cdot Z > 0$, and since all other terms are non-negative $[U]*[Z] > 0$.  Since the asymptotic orbits of $Z$ are elliptic, if $U,Z$ share a common asymptotic orbit we would have $\Delta(U,Z) > 0$. Therefore in either case $[U]*[Z] > 0$ by Theorem \ref{thm-posintersections}.  If $[U]*[Z] = 0$ both of these terms must be zero so the converse holds as well.
\end{proof}

\subsection{Compactness}

Let us call a subset $C$ of ${\M}$ \emph{closed under compactification} if for every $[(U)] \in \partial C \subset \overline{\M}$, each component $U^i$ of $(U)$ represents an element $[U^i] \in C$.  Denoting by $\M_{J}(1;k)$ the union of all $\M_J(x;y_1,\dots,y_k)$, then $\M_J(1;*) = \cup_{k \geq 0} \M_J(1;k)$ is closed under compactification.  Denoting by $\M_J^{\leq 2}(1;1)$ the subset of cylinders of
index at most $2$, the main obstacle to defining cylindrical contact homology is to show that $\M_J^{\leq 2}(1;1)$ is closed under compactification.  Similarly set $\M(1;k:Z)$ to be the union of all $\M_J(x;y_1,\dots,y_k:Z)$.  Our goal is to show the subsets
$\M^{[a]}_J(1;1:Z_L)$ of cylinders in $\M_J(1;k:Z_L)$ connecting asymptotic orbits in the homotopy class of loops $[a]$ are closed under compactification when
$(\lambda_+,L) \geq (\lambda_-,L)$ satisfy $(E)$, or when $(\lambda_{\pm},L,[a])$ satisfy $(PLC)$.

\subsubsection{Condition $(E)$}

We investigate the compactification of finite energy holomorphic curves in $W_{\xi}$ with intersection number $0$ with $[Z = Z_L = \pi_V^{-1}(L)]$ (in the sense defined above).  We assume $(\lambda_+,L) \geq (\lambda_-,L)$ throughout, and later that both satisfy condition $(E)$.  We suppose $J \in \J(\widehat{J}_-,\widehat{J}_+:Z)$ (in the cylindrical case $\lambda_+ = \lambda_-$ and $J \in \J(\lambda_{\pm})$; then $Z$ is automatically $J$-holomorphic).

\begin{lemma}
\label{lem-compactness1}
 Suppose $(\lambda_+,L) \geq (\lambda_-,L)$ (with $L_{\pm}$ all elliptic), $Z = \pi_V^{-1}(L)$ and $J \in \J(\widehat{J}_-(\lambda_-),\widehat{J}(\lambda_+):Z)$ (allow the possibility $J$ is cylindrical i.e.~$J \in J(\lambda_-) = \J(\lambda_+)$).  Let $L_i$ be a connected component of $L$, and $C = \pi^{-1}(L_i)$.  Let $U$ be a possibly branched cover of $C$ with $1$ positive puncture and genus $0$.  Then $[U]*[C] \geq \frac{1}{2}(1 - \# \Gamma_-)$, where $\# \Gamma_-$ is the number of negative punctures of $U$.
\end{lemma}

\begin{proof}  Let $S^2 \backslash \Gamma(U)$ be the domain of $U$, so we have a lift $\hat{U}:S^2 \backslash \Gamma \rightarrow \R \times S^1$ i.e. $U = C \circ \hat{U}$.  Using the metric $g_J$, since $J \in \J(\widehat{J}_-,\widehat{J}_+:Z)$ we can identify the normal bundle to $C$ with $\xi$, and with the exponential map and a unitary trivialization $\Phi: C^*\xi \rightarrow \R \times S^1 \times \C$ we can identify a tubular neighborhood of $C$ with $(\R \times S^1) \times D_{\epsilon}$, where $D_{\epsilon}$ is a disc containing the origin in $\C$.  Let us denote this map $\hat{f} = \exp \circ \Phi^{-1} : (\R \times S^1) \times D_{\epsilon} \rightarrow W$.  Consider the pull-back bundle $N = U^*\xi$ over $S^2\backslash \Gamma(U)$.  Using the trivialization and lift, we have a map $f = \hat{f} \circ (\hat{U} \times Id): S^2 \backslash \Gamma(U) \times D_{\epsilon} \rightarrow W$, with the property that $f(z,c) \in C$ if and only if $c = 0$ (possibly taking $D_{\epsilon}$ smaller if needed).  It is a local-diffeomorphism away from the branch points.

Sufficiently small sections of $N$ can be identified with functions $\sigma:S^2 \backslash \Gamma(U) \rightarrow D_{\epsilon}$ by this trivialization, and it defines a map $S_{\sigma}: S^2 \backslash \Gamma(U) \rightarrow W_{\xi}$ by $z \mapsto f(z,\sigma(z))$.  Choose any section $\hat{\sigma}$ of $N$ such that 

\begin{itemize}
 \item near the punctures it has the form\footnote{In cylindrical coordinates at the puncture: $(s,t) \in [0,\infty) \times \R /\Z \cong D_{z^+} \backslash \{z^+\}$ at a positive puncture $z^+$, and $(s,t) \in (-\infty,0] \times \R /\Z \cong D_{z^-} \backslash \{z^-\}$ at a positive puncture $z^-$.}
\[
 \hat{\sigma}(s,t) = e^{\lambda s}e_{\lambda}(t)
\]
\noindent
where $\lambda$ is the extremal eigenvalue of the asymptotic operator of the orbit to which $U$ is asymptotic at that puncture, and $e_{\lambda}(t)$ represents an eigensection associated to $\lambda$,
 \item it has only finitely many transverse zeros none of which occur at branch point,
 \item it is small enough to be represented by a function $\sigma:U \rightarrow D_{\epsilon}$.
\end{itemize}

We compute $[C] * [U]$ by computing $[C] * [S_{\sigma}]$ using the formula of Theorem \ref{thm-posintersections} 
\[
 [C] * [S_{\sigma}] = \mathrm{int}(C,[S_{\sigma}]) + \delta_{\infty}(C,S_{\sigma}) + \frac12 \Delta([C],[S_{\sigma}]),
\]
since $S_{\sigma}$ is homotopic to $U$ via $\tau \mapsto S_{\tau \cdot \sigma}$. In view of the representation of $\hat{\sigma}$ near the punctures we have
\begin{equation} \label{eq-ws}
\begin{split}
- w_{\lambda_+,\Phi}(S_{\sigma},w) + \alpha_{\lambda_+,\Phi}^-(L_i,k_w) & = 0, \mbox{ if } w \in \Gamma^+(S_{\sigma})\\
+ w_{\lambda_-,\Phi}(S_{\sigma},w) - \alpha_{\lambda_-,\Phi}^+(L_i,k_w) & = 0, \mbox{ if } w \in \Gamma^-(S_{\sigma})
\end{split}
\end{equation}
\noindent The formula \eqref{eq-deltainfty} for $\delta_{\infty}$ applies here, so the equations \eqref{eq-ws} imply $\delta_{\infty}(C,S_{\sigma}) = 0$.  Since all asymptotic limits are elliptic, it follows that each puncture $w \in \Gamma(S_{\sigma})$ contributes $\Delta(\infty_w, w) \geq 1 $ ($\infty_w$ denotes the puncture of $C$ of the same sign as $w$; this inequality is demonstrated in the proof of Proposition $4.3$ item $(4)$ in \cite{Siefring2}).  Therefore $\frac{1}{2} \Delta(C,S_{\sigma}) \geq \frac{1}{2} \# \Gamma(U)$.

It remains to compute $\mathrm{int}(C, S_{\sigma})$.  Since the map $\hat{f}: \R \times S^1 \times D_{\epsilon} \rightarrow W$ is an orientation preserving local-diffeomorphism away from branch points, and all zeros occur away from branch points by construction, the computation reduces to a signed count of intersections of this map $\sigma$ with $0 \in \C$ (the mapping degree $deg(\sigma;0)$).  This degree is computed by the homology class of $\sigma|_{\partial{S}} \in H_1(\C \backslash 0)$ (where $S$ is $S^2$ minus the small disc-like neighborhoods $D_{z^{\pm}}$ of the punctures $\Gamma(U)$ on which the asymptotic expression holds, i.e. $\sigma(s,t) = e^{\lambda s}e_{\lambda}(t)$).  The extremal winding numbers are given by $\alpha^{\mp}_{\lambda_{\pm},\Phi}(L_i;k_w)$ where $k_w$ is the covering number of $L_i$ at each puncture $w$ of $U$, using $\alpha^-$ with respect to $\lambda_+$ at positive punctures, and $\alpha^+$ with respect to $\lambda_-$ at negative punctures.  We recall from section \ref{sec-CZindices} that there are irrational numbers $\theta^i_{\pm}$ so that
\begin{align*}
 \alpha^-_{\lambda_+, \Phi}(L_i;k_w) &= \lfloor k_w \theta^i_+ \rfloor \\
 \alpha^+_{\lambda_-, \Phi}(L_i;k_w) &= \lceil k_w \theta^i_- \rceil = \lfloor k_w \theta^i_- \rfloor + 1
\end{align*}
\noindent
Therefore, the mapping degree is
\[
 \lfloor k^+ \theta^i_+ \rfloor - \sum_{z_-} (\lfloor k^-_{z_-} \theta^i_- \rfloor + 1)
\]
\noindent where $k^+$ is the covering number of the positive puncture, and the $k^-_{z_-}$ are the covering numbers at the negative punctures $z_-$, so $k^+ = \sum_{z_- \in \Gamma_-} k^-_{z_-}$.  By the hypothesis $(\lambda_+,L) \geq (\lambda_-,L)$ we must have that $\theta^i_+ \geq \theta^i_-$ which implies
\[
 \lfloor k^+ \theta^i_+ \rfloor - \sum_{z_-} \lfloor k^-_{z_-} \theta^i_- \rfloor \geq 0,
\]
\noindent so the mapping degree is at least $-\# \Gamma_-$.  Summing everything together,
\begin{align*}
  [C]*[U] &= C \cdot S_{\sigma} + \delta_{\infty}(C,S_{\sigma}) + \frac{1}{2} \Delta(C,U) \\
& \geq - \# \Gamma_-(U) + 0 + \frac{1}{2} \# \Gamma(U) \\
& = \frac{1}{2}(1 - \# \Gamma_-(U))
\end{align*}
\end{proof}

\begin{remark}
If we drop the hypothesis $(\lambda_+,L) \geq (\lambda_-,L)$, then we lose this control on this intersection number since $ \lfloor k^+ \theta^i_+ \rfloor - \sum_{z_-} \lfloor k^-_{z_-} \theta^i_- \rfloor $ could be negative.
\end{remark}

\begin{lemma} \label{lem-planepositivity}
 Suppose $(\lambda_{+},L) \geq (\lambda_{-},L)$ satisfy condition $(E)$ (actually we only need to assume $(\lambda_+,L)$ satisfies $(E)$) and $J \in \J(\widehat{J}_-(\lambda_-),\widehat{J}_+(\lambda_+):Z)$.  Then for every finite energy $J$-plane $P$ in $W_{\xi}$, $[P] * [Z] \geq 1/2$.  If it is not asymptotic to an orbit in $L$ then $[P] * [Z] \geq 1$.

 If $(\lambda_{+},L) \sim (\lambda_{-},L)$ satisfy condition $(PLC)$ instead, then for every finite energy $J$-plane $P$ in $W_{\xi}$ $[P] * [Z] \geq 1$.
\end{lemma}

\begin{proof}
If $P$ is asymptotic to $d \in L$ then the first assertion is immediate by the formula of Theorem \ref{thm-posintersections} since the parity term will contribute $1/2$ and all terms are non-negative.  Suppose instead that $P$ is asymptotic to an orbit $x \notin L _+$.  Then $\pi_V \circ P: \C \rightarrow V$ must intersect $L$ by hypothesis, and therefore $P$ must intersect $Z_L$.  Since both are $J$-holomorphic the intersection contributes $1$ to the intersection, which bounds below $[P] * [Z] \geq 1$ by positivity of intersections.

If $(\lambda_+,L)$ satisfies $(PLC)$ then $\pi_V \circ P$ intersects $L$ by hypothesis, so $[P] * [Z] \geq 1$ as explained above.
\end{proof}

We will call a puncture of a holomorphic building $(U)$ a \emph{free puncture} if there is a path from the puncture to an end of the tree such that all nodes along this path represent trivial cylinders.

\begin{lemma}
\label{lem-compactness2}
  Suppose that  $(\lambda_+,L) \geq (\lambda_-,L)$, and $J \in \J(\widehat{J}_-(\lambda_-),\widehat{J}_+(\lambda_+):Z)$.  Let $(U)$ be any genus $0$ holomorphic building with only $1$ positive end.  Then
\begin{itemize}
 \item If no component of $(U)$ is a branched cover of a component of $Z$, then $[(U)]*[Z] \geq 0$.  It is strictly greater than zero if either 
\begin{itemize}
\item  there is an edge in the bubble tree corresponding to an orbit in $L$, or 
\item  the projection of some component intersects $L$
\end{itemize}
\noindent In particular, if both $(\lambda_+,L), (\lambda_-,L)$ satisfy $(E)$ and some component of $(U)$ is a plane then one of these holds necessarily, so $[(U)]*[Z_L] > 0$.
\item If some components of $(U)$ are branched covers of components of $Z$, then $[(U)]*[Z] \geq \frac{1}{2}(1 - \# \Gamma^-_{free}((U),L)$), where $\Gamma^-_{free}((U),L)$ denotes the set of free negative punctures belonging to components which are branched covers of components of $Z$.
\end{itemize}
\end{lemma}

\begin{proof}
  Let $U_i$ be the vertices of $(U)$ which represent curves that are not branched covers of components of $Z$, and let $V_j$ be the vertices of $(U)$ which represent curves that are branched covers of components of $Z$.  The level-wise additivity and disjoint union properties of the intersection number (Proposition \ref{prop-intersectionproperties}) imply

\begin{align*}
 [(U)]*[Z] &= [\sum U_i + \sum V_j]*[Z] \\
	&= \sum [U_i]*[Z] + \sum [V_j]*[Z] \\
\end{align*}
\noindent
In the first case, the second sum is empty.  Since $[U_i]*[Z] \geq 0$, and is strictly positive if it is asymptotic to some $x \in L$ or if $\pi_V \circ U^i$ intersects $L$ in its interior (Lemma \ref{lem-zerointstructure}), the first statement is clear.

For the second case, using the estimate from Lemma \ref{lem-compactness1} we have

\begin{align*}
 \sum_{V_j} [V_j]*[Z] &\geq \sum_{V_j} \frac{1}{2}(1 - \# \Gamma_-(V_j)) \\
	&= \frac{1}{2} \# V_j  + (\sum_{V_j} \sum_{\Gamma_-(V_j)} -\frac{1}{2} )\\
\end{align*}
\noindent
Consider a negative puncture $z \in \cup \Gamma_-(V_j)$ which is not a free puncture for $(U)$.  There is a first vertex on the path below this puncture which does not represent a trivial cylinder, which is either a $U_i$ or a $V_{j'}$ (and which conversely determines $(V_j,z)$).  If this vertex is is not a $V_{j'}$, then it is a $U_i$ asymptotic to an orbit in $L$ and therefore $[U_i]*[Z] \geq 1/2$.  Else, it is a $V_{j'}$.  Therefore, given a term in the sum $\sum_{V_j} \sum_{\Gamma_-(V_j)} -\frac{1}{2}$ which does not represent a free negative puncture, we find a uniquely determined term in the sum $\sum [U_i]*[Z] + \sum_{V_j} \frac{1}{2}$ that combines with it to make the total non-negative.  By this grouping of terms we determine
\begin{align*}
 [(U)]*[Z] &= \sum [U_i]*[Z] + \sum [V_j]*[Z] \\
	&= \sum [U_i]*[Z] + \frac{1}{2} \# V_j + (\sum_{V_j} \sum_{\Gamma_-(V_j)} -\frac{1}{2} )\\
	& \geq \frac{1}{2} (1 - \# \Gamma^-_{free}(U,L) )
\end{align*}
\noindent
the extra $1$ coming from the fact that there is at least one $V_j$ that is not used to cancel one of the terms in $\sum_{V_j} \sum_{\Gamma_-(V_j)} -\frac{1}{2}$ (there must be one $V_j$ not on a lower level than some other $V_{j}$).
\end{proof}

Now the following proposition is a consequence of Lemma \ref{lem-compactness2} and Lemma \ref{lem-zerointstructure}.

\begin{proposition}
\label{prop-compactness}
 Suppose $(\lambda_+,L) \geq (\lambda_-,L)$ satisfy $(E)$, $x \in \mathcal{P}(\lambda_+)$, $y \in \mathcal{P}(\lambda_-)$ are not in $L$, and $(U) \in \overline{\M}_J(x;y)$ has genus $0$.  Then $[(U)] * [Z] \geq 0$.  

If $[(U)] * [Z] = 0$, then each level $U^i$ of $(U)$ is a cylinder with $[U^i] \in \M_{J'}(a;b:Z)$ (where $J' \in \{\widehat{J}_+(\lambda_+), J, \widehat{J}_-(\lambda_-)$ determined by the level $i$).  In particular, this holds whenever $(U)$ is a SFT-limit of a sequence $U_k$ in $\M_J(x;y:Z_L)$, since Lemma \ref{lem-continuityofint} implies $[(U)]*[Z] = 0$.

If instead we only have $(\lambda_+,L) \geq (\lambda_-,L)$ and each component of $L$ is elliptic for both $\lambda_{\pm}$, then if $(U)$ is a holomorphic building with one positive puncture, no asymptotic limits in $L$ (so $\Gamma^-_{\mathrm{free}}((U),L) = 0$), and $[(U)] * [Z_L] = 0$, then no component $U^i$ of $(U)$ is asymptotic to $L$ or such that $\pi_V \circ U^i$ intersects $L$ in the interior.  In particular, the compactified projected image $\pi_V \circ \overline{(U)} \subset V \backslash L$.
\end{proposition}

\subsubsection{Compactness under splitting}
\label{sec-compactnessundersplitting}
In the following, suppose $\lambda_+ \succ \lambda_0 \succ \lambda_-$, $(\lambda_+,L) \geq (\lambda_0,L) \geq (\lambda_-,L)$, and $L$ is elliptic for each.  Consider almost-complex structures $J_1 \in \J(\widehat{J_-}(\lambda_-),\widehat{J_0}(\lambda_0):Z_L), J_2 \in \J(\widehat{J_-}(\lambda_-),\widehat{J_0}(\lambda_0):Z_L)$, and the almost-complex structures $J_R \in \J(J_1,J_2)$ described in section \ref{sec-symplectizations} obtained from $J_1,J_2$.  

\begin{proposition}
 \label{prop-splitcompactness}  
Suppose further that $(\lambda_+,L) \geq (\lambda_0,L) \geq (\lambda_-,L)$ all satisfy condition $(E)$.  Then each level $U^i$ of $(U)$ is a cylinder $[U^i] \in \M_J(a;b:Z)$ where $J \in \{ \widehat{J}_+(\lambda_+), J_2, \widehat{J}_0(\lambda_0), J_1, \widehat{J}_-(\lambda_-) \}$.\end{proposition}

\begin{proof}
  Let $(U)$ be a SFT-limit for the sequence $U_k$; by continuity of the intersection number (Lemma \ref{lem-continuityofint}) we must have $[(U)]*[Z_L] = 0$.  Let $(U)_2$ be the building obtained by looking only at the top $|1|k_+$-levels of $(U)$, and $(U)_1$ be the building obtained by looking only at the bottom $k_-|1|k_0$ levels of $(U)$.  Then $(U)_1$ represents the homotopy class $[(U)_1]$ of asymptotically cylindrical maps in $W_1 = (W_{\xi},J_1)$, $(U)_2$ represents the homotopy class $[(U)_2]$ of asymptotically cylindrical maps in $W_2 = (W_{\xi},J_2)$, and $[(U)] = [(U)_2 \odot (U)_1]$.  The map $(U)_2$ has a connected domain, while the map $(U)_1$ may have a disconnected domain: let us denote by $(U_1^i)$ the connected components of $(U)_1$.  Then
  \[
     [(U)] = [(U)_2 \odot (U)_1] = \left( \sum [(U_1^i)] \right) \odot [(U)_2]
  \]
  The level-wise additivity and disjoint union properties of Proposition \ref{prop-intersectionproperties} imply
 \[
   [(U)]*[Z_L] = [(U)_2]*[Z_L] +  \sum [(U_1^i)] *[Z_L] 
 \]
  Let $k$ be the number of negative punctures of $(U)_2$.  Then there are $k$-components $(U_1^i)$: label them so that for $1 \leq i \leq k-1$, the component $(U_1^i)$ is planar (has no negative punctures), while $(U_1^k)$ is cylindrical (has one negative puncture asymptotic to $y$).  

Suppose the second alternative of Lemma \ref{lem-compactness2} holds for $(U)_2$.  Then
  \[
    [(U)_2]*[Z_L] \geq \frac{1}{2} - \frac{1}{2} \cdot k
  \]
\noindent If for $(U_1^k)$ the positive asymptotic orbit is not in $L$, the stronger inequality
  \[
     [(U)_2]*[Z_L] \geq \frac{1}{2} - \frac{1}{2} \cdot (k-1)
  \]
\noindent holds.  By Lemma \ref{lem-compactness2}, for the $k-1$ planar components $(U_1^i)$:
  \[
    [(U_1^i)] *[Z_L] \geq \frac{1}{2}
  \]
For $(U_1^{k})$ (by Lemma \ref{lem-compactness2} as well):
  \[
    [(U_1^{k})] *[Z_L] \geq 0
  \]
Moreover, if $(U_1^{k})$ has a positive asymptotic limit in $L$, then the last inequality is strengthened to $\geq 1/2$.  This is because if the second alternative to Lemma \ref{lem-compactness2} holds for $(U_1^{k})$, the inequality implies $[(U_1^k]*[Z] \geq 1/2$.  If the first holds, then some level of $(U_1^k)$ is asymptotic to $L$: Lemma \ref{lem-compactness2} asserts both that level will have strictly positive intersection, and that each level has non-negative intersection, so by level-wise additivity $[(U_1^k)]*[Z] \geq 1/2$ again.  It follows that whether or not $(U_1^k)$ has positive asymptotic orbit in $L$, the inequality $[(U)]*[Z] \geq 1/2$ holds.  Therefore the first alternative to Lemma \ref{lem-compactness2} must apply to $(U)_2$.

Since the first alternative to Lemma \ref{lem-compactness2} holds for $(U)_2$, $[(U)_2]*[Z] \geq 0$.  Lemma \ref{lem-compactness2} implies that for $1 \leq i \leq k-1$, $[(U_1^i)] * [Z] \geq 1/2$ and that $[(U_1^k)] * [Z] \geq 0$.  Therefore $0 = [(U)]*[Z] \geq \frac{k-1}{2}$, and therefore $k = 1$ i.e~ $(U)_1 = (U_1^k)$.  Lemma \ref{lem-compactness2} implies that $[(U)_1] * [Z] \geq 0$.  Since $[(U)_2]*[Z] \geq 0$, it must be that $[(U)_2] *[Z] = [(U)_1]*[Z] = 0$.  Lemma \ref{lem-compactness2} therefore asserts that $(U)_2$ has no planar components and is never asymptotic to an orbit in $L$ - because $k = 1$ it must be that every level of $(U)_2$ is a single cylinder.  Since $[(U)_1] *[Z] = 0$ and $\# \Gamma^-_{\mathrm{free}}((U)_1,L) = 0$ (this is defined in Lemma \ref{lem-compactness2}), the second alternative of Lemma \ref{lem-compactness2} cannot hold, so the first alternative to Lemma \ref{lem-compactness2} must apply.  It implies that $(U)_1$ also has no planar components and is never asymptotic to an orbit in $L$.  Therefore, every level $U^i$ for $(U)$ is a cylinder and $[U^i]*[Z] \geq 0$ - it follows by level-wise additivity $[U^i] * [Z] = 0$ for all levels $i$, in particular $U^i \in \M_{J'}(a,b:Z)$.
\end{proof}

In case $(\lambda_+,L) \geq (\lambda_0,L) \geq (\lambda_-,L)$ do not necessarily satisfy condition $(E)$ (but are each such that $L$ is elliptic), it will be useful in the proof of Theorem \ref{thm-forcing2} to note the following.  Let $J_n$ be a sequence of almost-complex structures splitting into $J_1 \odot J_2$.

\begin{proposition} \label{prop-splitcompactnessNOT-E}
Suppose $x,y \notin L$ and that the sequence $U_n \in \M_{J_n}(x;y:Z)$ has the SFT-limit $(U) = (U)_2 \odot (U)_1$.  Then $[(U)_1]*[Z_L] = 0, [(U)_2]*[Z_L] = 0$, and therefore by Lemma \ref{lem-compactness2} each component $U^i$ of the limit satisfies the following properties:
\begin{itemize}
 \item No asymptotic limit of $U^i$ is an orbit in $L$
 \item $\pi_V \circ U^i$ has no interior intersections with $L$
\end{itemize}
\noindent  In particular, the compactified image $\pi_V \circ \overline{(U)_i} \subset V \backslash L$ (in particular it defines a homotopy from $x$ to $y$ since its domain is a compactified cylinder).
\end{proposition}

\begin{proof}
The proof is similar to the proof of Proposition \ref{prop-splitcompactness}.  First, by continuity of the intersection number $[(U)]*[Z_l] = 0$.  Let $(U_1^i)$ be the connected components of $(U)_1$.  Then
\[
0 =  [(U)]*[Z_L] = [(U)_2]*[Z_L] +  \sum [(U_1^i)] *[Z_L] 
\]
Let $k$ be the number of negative punctures of $(U)_2$, and $l = \# \Gamma^-_{\mathrm{free}}((U)_2,L) \leq k-1$ (recall from earlier that this is the number of free negative punctures of $(U)_2$ asymptotic to a cover of $L$).  Then by Lemma \ref{lem-compactness2}
  \[
    [(U)_2]*[Z_L] \geq \frac{1}{2} - \frac{1}{2} \cdot l
  \]
\noindent and for each puncture of $\Gamma^-_{\mathrm{free}}((U)_2,L)$ there is a planar component of $(U)_1$, say $(U)_1^1,\dots,(U)_1^l$ such that $[(U)_1^i]*[Z_L] \geq \frac{1}{2}$ (again by Lemma \ref{lem-compactness2}).  Taken together this implies $[(U)]*[Z_L] > 0$, which is impossible so in fact $l = 0$.  Thus every component $(U)_2, (U)_1^1, \dots (U)_1^k$ satisfies $\Gamma^-_{\mathrm{free}}((U)_2,L) = \Gamma^-_{\mathrm{free}}((U)_1^j,L) = \emptyset$.  Applying Lemma \ref{lem-compactness2} it follows that each term in the above sum is non-negative, and therefore each term is zero.  Reading Lemma \ref{lem-compactness2} once again, it now asserts that the components have the properties stated above.
\end{proof}

\subsubsection{Proper link classes $(PLC)$}

If $(\lambda_{\pm},L,[a])$ satisfy $(PLC)$, we can prove a compactness theorem for holomorphic cylinders within the class of broken trajectories such that all asymptotic limits lie within the class $[a]$:

\begin{proposition}
\label{prop-PBCcompactness}
 Suppose $(\lambda_+,L) \sim (\lambda_-,L)$, and both $(\lambda_{\pm},L,[a])$ satisfy $(PLC)$.  Let $J \in \J(\widehat{J}_-(\lambda_-),\widehat{J}_+(\lambda_+):Z)$ and $x_{\pm}$ be $\lambda_{\pm}$-Reeb orbits in the class $[a]$.  If $[U_k] \in \M_J(x_+;x_-:Z)$ converges to a building $[(U)]$ in $\overline{\M_J}$, then each level $U^i$ of $(U)$ represents an element in $\M_{J'}(b;c:Z)$ ($J' \in \{ \widehat{J}_+(\lambda_+), J, \widehat{J}_-(\lambda_-) \}$) and $b,c \in [a]$.
\end{proposition}

\begin{proof}
Condition $(PLC)$ prevents bubbling off as follows.  Let $U_k$ be a convergent subsequence with limit building $(U)$.  Denote the connected components by $U^i$ (vertices of the bubble tree) indexed by $i \in I$.  Let $S$ be the underlying nodal domain of the building with components $S^i$ corresponding to the $U^i$.  By properties \texttt{CHCE2, CHCE3} of SFT-convergence\footnote{\label{footnote-CHCE} These properties are analogs of properties \texttt{C2}, \texttt{C3} introduced earlier when considering the compactness theorem for cylindrical ends (\emph{without splitting} along a contact-type hypersurface) - we refer to Section $8$ of \cite{BEHWZ} for the precise statements.} there are maps $\phi^i_k : S^i \rightarrow \R \times S^1$ from $S^i$ to the domain of $U_k$ such that $U_k \circ \phi^i_k|_{S^i}$ converges to $U^i$ in $C^{\infty}_{\rm{loc}}$ (up to a translation in cylindrical levels).  Suppose $U^i$ is a plane for some $i$.  Then by Lemma \ref{lem-planepositivity} and condition $(PLC)$ $[U^i] * [Z] \geq 1$.  It follows that $U_k \circ \phi^i_k|_{S^i} \cdot Z \neq 0$, because $U_k \circ \phi^i_k|_{S^i}$ converges to $U^i$ in $C^{\infty}_{\rm{loc}}$ and positivity of intersections implies that all maps $C^0$-near $U^i$ intersect $Z$ as well.  Therefore $\mathrm{int}(U_k, Z) \neq 0$ and so $[U_k] * [Z] > 0$ contrary to the hypothesis $[U_k] \in \M(x_+;x_-:Z)$.  This shows that the bubble tree is linear i.e.~all components are cylinders.

Since $x \notin L$, it is easy to see as follows that no asymptotic orbit between levels is in $L$ or not in $[a]$.  If not, there would be a first level $U^i$ with positive asymptotic orbit $x' \in [a]$ and negative asymptotic orbit $y'$ either in $L$ or not in $[a]$.  But $\pi_V \circ U^i$ provides a homotopy between $x'$ and $y'$, so by condition $(PLC)$ there must be an interior intersection with $L$ in either case.  But then the $C^{\infty}_{\mathrm{loc}}$-convergence property of SFT convergence cited in the previous paragraph implies by the same argument used above that for large $k$, $\mathrm{int}(U_k,Z) > 0$, which violates Theorem \ref{thm-posintersections} since $U_k, Z_L$ have non-identical images and $[U_k] * [Z_L] = 0$.  After establishing this fact about the asymptotic limits, no $U^i$ can have identical image with $Z$ and the same $C^{\infty}_{\mathrm{loc}}$-convergence argument shows that $\mathrm{int}(U^i,Z) = 0$ for each level.  The formula of Theorem \ref{thm-posintersections} shows $[U^i]*[Z] = 0$.  So $[U^i] \in \M_{J'}(x';y':Z)$ and $x', y' \in [a]$.
\end{proof}

\begin{remark}
 The same proof goes for the limit of a sequence of such $J_k$-holo\-morphic cylinders where the $J_k$ converge in $C^{\infty}$ to a $J \in \J(\widehat{J}_+,\widehat{J}_-:Z)$; in particular we get a similar statement when considering $\calN_{\{ J_l \} }(1;1:Z)$.
\end{remark}

A similar result holds for a sequence of $J_R \in \J(J_1,J_2)$.  In the following, suppose $\lambda_+ \succ \lambda_0 \succ \lambda_-$, that $(\lambda_{+,0,-},L,[a])$ each satisfy $(PLC)$, and that $J_{R_k} \in \J(J_1,J_2)$ with $R_k \uparrow \infty$ (where $J_1 \in \J(\widehat{J}_-,\widehat{J}_0:Z), J_2 \in \J(\widehat{J}_0,\widehat{J}_+:Z)$).

\begin{proposition}
 \label{prop-PBCsplitcompactness}  If $(U)$ is a SFT-limit of a sequence $U_k \in \M_{J_{R_k}}(x;y:Z)$ with $R_k \uparrow \infty$ and $[x] = [y] = [a]$, then each level $U^i$ represents a cylinder $[U^i] \in \M_J(b;c:Z)$ where $[b] = [c] = [a]$ and $J \in \{ \widehat{J}_+(\lambda_+), J_2, \widehat{J}_0(\lambda_0), J_1, \widehat{J}_-(\lambda_-) \}$.
\end{proposition}

\begin{proof}
The proof is similar to the previous proposition.  Properties \texttt{C2, C3} of SFT-convergence and condition $(PLC)$ prohibit the bubbling of planes as in the above proof, so each level $U^i$ is connected and cylindrical.  Similarly, the arguments in the second paragraph of the proof of Proposition \ref{prop-PBCcompactness} prevent any level from intersecting $Z_L$, prevent any level from being asymptotic to any component of $L$, and force any asymptotic orbit $x$ between levels to be in the homotopy class $[a]$.  So the SFT-limit must have the form stated.
\end{proof}

In the proof of Theorem \ref{thm-forcing2}, we will require a more general fact.  Suppose that $L$ satisfies the linking condition of Theorem \ref{thm-forcing2}, namely:
\begin{quote}
 $L$ is such that every disc $F$ with boundary $\partial F  \subset L$ and $[\partial F] \neq 0 \in H_1(L)$ has an interior intersection with $L$
\end{quote}
\noindent and that $[a]$ is a proper link class for $L$ (Definition \ref{def-PLC}).  Suppose only that $\lambda_+,\lambda_0,\lambda_-$ are  each non-degenerate and such that $L$ is closed under the Reeb flows.  Then

\begin{proposition}
 \label{prop-PBCsplitcompactnessNOT-PLC}  Consider a sequence $U_k$ of cylinders in $\M_{J_{R_k}}(x;y:Z)$, with $R_k \uparrow \infty$ and $[x] = [y] = [a]$.  Then any SFT-limit $(U) = (U)_1 \odot (U)_2$ is such that $[(U)_1]*[Z_L] = 0, [(U)_2]*[Z_L] = 0$, and each component $U^i$ of the limit satisfies the following properties:
\begin{itemize}
 \item No asymptotic limit of $U^i$ is an orbit in $L$
 \item $\pi_V \circ U^i$ has no interior intersections with $L$
\end{itemize}
\noindent  In particular, the compactified images $\pi_V \circ \overline{(U)} \subset V \backslash L$.
\end{proposition}

\begin{proof}
The assertion that for each component $U^i$ that $\pi_V \circ U^i$ has no interior intersections with $L$ is argued using the $C^{\infty}_{\mathrm{loc}}$ convergence properties of SFT-convergence in a familiar way.  By properties \texttt{CHCE2, CHCE3}\footnote{See footnote \ref{footnote-CHCE}.} of SFT-convergence there are maps $\phi^i_k : S^i \rightarrow \R \times S^1$ from $S^i=$ the domain of $U^i$ to the domain of $U_k$ such that $U_k \circ \phi^i_k|_{S^i}$ converges to $U^i$ in $C^{\infty}_{\rm{loc}}$ (up to a translation in cylindrical levels).  Since $U^i$ intersects $Z_L$ positively, it follows that $U_k$ must intersect $Z_L$ for all $k$ large enough, but this contradicts $[U_k] * [Z_L] = 0$.  So we conclude $U^i$ has no interior intersections with $Z_L$.

The assertion that no asymptotic orbit of any component $U^i$ is in $L$ (equivalently, no edge in the bubble tree represents an orbit in $L$) is argued as follows.  Suppose some edge $E$ in the bubble tree represents an orbit in $L$.  Suppose first that the portion of the holomorphic building corresponding to the subtree below this edge is planar.  Choose an edge $E'$ in this subtree such that no edge lower on the bubble tree lies in $L$ (this is possible because the tree is planar so there are no free negative asymptotic orbits, thus every edge has a vertex below it).  Let $(P)$ be the holomorphic building corresponding to the subtree below this edge $E'$.  Then $\pi_V \circ \overline{(P)}$ has an interior intersection with $L$, which must correspond to an interior intersection of one of its components (since no asymptotic orbit is in $L$), contradicting the first assertion which we have already proved.

Otherwise, the portion below $E$ has one free negative puncture.  $E$ is therefore on the unique path from the free positive puncture to the free negative puncture.  Without loss of generality, suppose $E$ is the highest edge along this path (i.e.~closest to the positive free puncture).  Let $(X)$ be the tree obtained by removing from the original bubble tree the portion below this edge $E$ (by hypothesis $E$ is not the positive free puncture so this is non-empty).  By the previous step, any free negative puncture for $(X)$ bounds a planar building $(P)$ such that each component satisfies the conclusions of this proposition.  Hence, its compactified image satisfies $\pi_V \circ \overline{(P)} \subset V \backslash L$.  It will now follow that the interior of $\pi_V \circ \overline{(X)} \subset V \backslash L$, since no edge of $(X)$ represents an orbit in $L$ (except $E$) and we have already shown that no component of $(U)$ has an interior intersection with $Z_L$.  However, the negative asymptotic orbit of $(X)$ is in $L$ (since by hypothesis $E$ represents an orbit in $L$), and $\pi_V \circ \overline{(X)}$ realizes a homotopy between this orbit in $L$ and the positive asymptotic orbit $x \in [a]$ through $V \backslash L$.  But this contradicts the hypothesis that $[a]$ is a proper link class, by definition.  Therefore no such edge $E$ can exist, which finally proves that no asymptotic orbit of any component $U^i$ is in $L$ as claimed.
\end{proof}

\section{Cylindrical Contact Homology of Complements}
\label{sec-CCHint}

The two crucial requirements for the definition of contact homology are compactness and regularity of the moduli spaces.  Proposition \ref{prop-compactness}, \ref{prop-PBCcompactness}, \ref{prop-splitcompactness}, \ref{prop-PBCsplitcompactness} supply the required compactness for the spaces $\M_J(x;y:Z)$.  In the cylindrical case, Theorem \ref{thm-transversality} provides the required transversality for the moduli spaces (to prove $\partial^2 = 0$), but in the non-cylindrical case these arguments do not apply and we need to make some hypotheses.  When considering non-cylindrical almost-complex structures, we will consider only simple homotopy classes of loops in $V \backslash L$.  Then following well-known arguments as in e.g. \cite{MR1360618,Dragnev, MR2200047,MR2284048}, for each fixed choice of cylindrical ends $\widehat{J}_{\pm}$ there is a residual subset\footnote{One should specify a Banach space of perturbations; the procedure is now standard so we refer to \cite{MR1360618, Dragnev, MR2284048}.  The key point is that each holomorphic cylinder with a positive and negative puncture which is not a cover of a cylinder in $Z$ has a point of injectivity outside $W_+ \cup W_- \cup Z$ where we are free to infinitesimally perturb the almost-complex structure.} $\J_{\rm{gen}} = \J_{\rm{gen}}(\widehat{J}_-,\widehat{J}_+:Z)$ of admissible almost-complex structures for which the Cauchy-Riemann operator will be regular for all $J$-holomorphic cylinders connecting loops in simple classes $[a]$ if $J \in \J_{\rm{gen}}$.  Similarly, there are generic paths of almost-complex structures $\J_{\tau, \rm{gen}} = \J_{\tau,\rm{gen}}(\widehat{J}_-,\widehat{J}_+:Z)$.  

\subsection{Boundary Maps; Chain maps and homotopies}
\label{sec-cchint}
With the properties of intersections from Section \ref{sec-intersections}, we can describe our cylindrical contact chain complex on $V \backslash L$ in detail and show that it has the required properties.  The proofs of Propositions \ref{prop-dsquared}, \ref{prop-chainmap}, \ref{prop-chainhomotopy} stated below are postponed to the end of the section.

Suppose $(\lambda,L)$ satisfies $(E)$, or that $(\lambda,L,[a])$ satisfies $(PLC)$, and choose $J \in \J_{\rm{gen}}(\lambda)$.  Let $Z = Z_L$ be the cylinders over $L$ ($Z = \pi_V^{-1}(L)$).  Let $\mathcal{G}$ denote the set of SFT-good Reeb orbits, and $\mathcal{G}'$ denote $\mathcal{G} \backslash L$.  Let $I$ be an ideal in $\rm{ker}(c_1(\xi))$ where $c_1(\xi):H_2(V;\Z) \rightarrow \Z$ is the first Chern class of $\xi$, and let $R$ be the ring $R = \Q[H_2(V;\Z) / I]$.  We will denote elements of $R$ as sums $\sum q_i e^{[A_i]}$, where $q_i \in \Q$, $ \{A_i\} \subset H_2(V;\Z)$ is a finite set and $[A_i]$ denotes its image in the quotient $H_2(V;\Z) / I$.  Consider the $R$-module $CC_*(\lambda \mbox{ rel } L)$ freely generated by elements $q_x$ (one for each $x \in \mathcal{G}'$).  When considering $(PLC)$ we only consider orbits in the homotopy class $[a]$ and denote it by $CC^{[a]}_*(\lambda \mbox{ rel } L)$ instead.  Grade the module as in \cite{EGH}.  Recall from \cite{EGH} that this also determines a homology class $A \in H_2(V;\Z)$ for each asymptotically cylindrical curve $U$, and we denote by $\M^{A}_J$ the moduli space of solutions representing the class $A$.

Define a degree $-1$ map $\partial_L = \partial_L(J)$ on the generators of $CC_*(\lambda \mbox{ rel } L)$ (resp. $CC^{[a]}_*(\lambda \mbox{ rel } L)$) by summing over index $1$ solutions (and extend linearly):

\begin{equation} \label{eq-boundary}
\begin{split}
	n^{A}_{x y} &= \sum_{[U] \in \M^{A}_J(x;y:Z)} m_{x} \cdot \frac{\epsilon([U])}{\rm{cov}([U])}\\
	n_{x y} &= \sum_{A \in H_2(V;\Z)} n^A_{x y} e^{[A]} \in R\\
	\partial_{L} q_x &= \sum_{y \in \mathcal{G}'} n_{x y} q_y
\end{split}
\end{equation}

\noindent
Here $m_x$ denotes the covering number of the orbit $x$, and $\rm{cov}(U)$ denotes the covering number of the map $U$.  The sign $\epsilon[U] = \pm 1$ is determined by comparing a choice of coherent orientation \cite{MR2092725} with the canonical orientation given by the free $\R$-action (it is well-defined since both asymptotic orbits are SFT-good).  Since the moduli spaces we consider are open/closed subsets of the moduli space for $(V,J)$, a coherent orientation chosen in the usual way will restrict to one for our moduli spaces and it does not require further comment (see \cite{MR2092725} or \cite{EGH} for details).  Though we suppressed it in the notation, it is an essential ingredient in the above definition.  This sum will always be finite by compactness (Proposition \ref{prop-compactness}, or \ref{prop-PBCcompactness}) and transversality (Theorem \ref{thm-transversality}).

\begin{proposition} \label{prop-dsquared}
Under the above hypotheses $\partial_L^2 = 0$.  Denote the chain complex $\left(CC_*(\lambda \mbox{ rel } L), \partial_L(J) \right)$ by $CC_*(\lambda, J \mbox{ rel } L)$ (or $CC^{[a]}_*(\lambda, J \mbox{ rel } L)$ for $(PLC)$).

Under hypothesis $(E)$, the complex splits according to homotopy classes of loops in $V\backslash L$:
\[
 CC_*(\lambda,J \mbox{ rel } L) = \bigoplus_{[a] \in \pi_0(\Omega V)} CC^{[a]}_*(\lambda,J \mbox{ rel } L)
\]
\end{proposition}

We remark that this is the usual definition of the cylindrical contact homology chain complex for the pair $(\lambda|_{V\backslash L}, J|_{V\backslash L})$ on $V \backslash L$.

\subsubsection{Chain maps}

Suppose that $(\lambda_+,L) \geq (\lambda_-,L)$ satisfy $(E)$ and $[a]$ is a homotopy class of simple loops, or that $(\lambda_+,L,[a]) \sim (\lambda_-,L,[a])$ satisfy $(PLC)$ and $[a]$ is simple.  Choosing ${J}_{\pm} \in \J_{\rm{gen}}(\lambda_{\pm})$, we consider the chain complexes 
\[
  CC^{[a]}_*(\lambda_+,{J}_+ \mbox{ rel } L), \quad CC^{[a]}_*(\lambda_-,{J}_- \mbox{ rel } L)
\] 
\noindent defined in section \ref{sec-cchint}.  We will denote their boundary maps $\partial_L({J}_\pm)$ when it is needed to clarify.  Also, we use $\mathcal{G}'(\lambda_{\pm})$ to denote the set of SFT-good Reeb orbits not in $L$.

Choose and fix a $J \in \J_{\rm{gen}}(\widehat{J}_-,\widehat{J}_+:Z)$.  We will define a map
\[
  \Phi^{[a]}_{-+}(J): CC^{[a]}_*(\lambda_+, {J}_+ \mbox{ rel } L) \rightarrow CC^{[a]}_{*}(\lambda_-, {J}_- \mbox{ rel } L)
\]
For $x \in \mathcal{G}'(\lambda_+)$, we sum over index $0$ components of the moduli space of cylinders to define as in \cite{EGH}\footnote{The sign $\epsilon(U)$ is again determined by the choices of coherent orientations assigned \cite{EGH, MR2092725}.}:
\begin{equation}
\begin{split}\label{eq-chainmap}
	\Phi^{[a]}_{-+}(J) q_x &:= \sum_{y \in \mathcal{G}'(\lambda_-)} m_{xy} q_y
\end{split}
\end{equation}
\noindent where $m_{xy}$ and $m^A_{xy}$ are defined as
\begin{equation}
\begin{split}\label{eq-chainmapaux}
	m^{A}_{xy} &= \sum_{U \in \M^{A}_J(x;y:Z)} m_{x} \cdot \frac{\epsilon(U)}{\rm{cov}(U)} \\
	m_{xy} &= \sum_{A \in H_2(V;\Z)} m^{A}_{xy} e^{\pi(A)} \in R\\
\end{split}
\end{equation}

This sum is finite by compactness (Proposition \ref{prop-compactness}, or \ref{prop-PBCcompactness}) and transversality ($[a]$ simple plus Theorem \ref{thm-genericJ}).  It decreases the action by Stokes' theorem (see the comment after Lemma \ref{lem-finiteness}) and preserves the class $[a]$ (else the cylinders would have non-trivial intersection with $Z$).  \emph{Since we only consider homotopy classes of loops with only simple Reeb orbits}, the numbers $m_x$ and $\rm{cov}(U)$ are equal to $1$, so in fact the first equation in \eqref{eq-chainmapaux} simplifies to
\[
 	m^{A}_{xy} = \sum_{U \in \M^{A}_J(x;y:Z)} \epsilon(U) \\
\]

\begin{proposition}  \label{prop-chainmap}
Under the hypotheses used to define the operators above (including the regularity hypothesis and that $[a]$ is simple):
 \begin{align*}
  \partial_L(J_-) \circ \Phi^{[a]}_{-+}(J) &- \Phi^{[a]}_{-+}(J) \circ \partial_L(J_+) = 0\\
 \end{align*}
\end{proposition}

\subsubsection{Homotopies between chain maps}

Again, let $(\lambda_+,L) \sim (\lambda_-,L)$ and suppose either they both satisfy $(E)$ and $(\lambda_+,L) \geq (\lambda_-,L)$ or that $[a]$ is a proper link class and both satisfy $(PLC)$.  Let $\widehat{J}_{\pm} \in \J_{\rm{gen}}(\lambda_{\pm})$ so the chain complexes
\[
 CC_*(\lambda_{\pm},J_{\pm} \mbox{ rel } L)
\]
are defined as above.

Consider two choices $J_0,J_1 \in \J_{\rm{gen}}(\widehat{J}_-,\widehat{J}_+:Z)$, which defines chain maps $\Phi^{[a]}(J_0), \Phi^{[a]}(J_1)$ as above in the simple homotopy class $[a]$.  Choose a homotopy $J_{l} \in \J_{\tau, \rm{gen}}(\widehat{J}_-,\widehat{J}_+:Z)$ between the two.  Define for $x \in \mathcal{G}'(\lambda_+)$ (summing only over the index $0$ part of the moduli space $\calN_{\{ J_l \} }(x;y:Z)$)\footnote{Again, everything must be oriented appropriately, the details of which we omit.  Also, by ``index $0$'' we refer to the total Fredholm index including the parameter $\mu$: this means that $\rm{Ind}(\mu,U) = \rm{Ind}(U) + 1$, so we mean that $\rm{Ind}(U) = - 1$.}

\begin{equation}\label{eq-prismop}
\begin{split}
	k^{A}_{xy} &= \sum_{(\mu,U) \in \calN_{\{ J_l \} }^{A}(x;y:Z)} m_{x} \frac{\epsilon(\mu,U)}{\mathrm{cov}(U)} \\
	k_{xy} &= \sum_{A \in H_2(V;\Z)} k^{A}_{xy} e^{[A]} \in R\\
	K^{[a]}_{-+}(J_{\tau}) q_x &:= \sum_{y \in \mathcal{G}'(\lambda_-)} k_{x y} q_y
\end{split}
\end{equation}

\noindent (again, $m_x, \mathrm{cov}(U)$ are $1$ because $[a]$ is simple, and this sum is finite by compactness (Proposition \ref{prop-compactness}, or \ref{prop-PBCcompactness}) and transversality (Theorem \ref{thm-genericJ})).  Then
\begin{proposition} \label{prop-chainhomotopy}  Under the above hypotheses
 \begin{align*}
  & K^{[a]}_{-+}(J_{\tau}) \circ \partial_L(J_+) + \partial_L(J_-) \circ K^{[a]}_{-+}(J_{\tau}) \\ 
  &= \Phi^{[a]}_{-+}(J_1) - \Phi^{[a]}_{-+}(J_0) 
 \end{align*}
\end{proposition}

\begin{proof}  The proofs of Proposition \ref{prop-dsquared}, \ref{prop-chainmap}, \ref{prop-chainhomotopy} are standard in contact homology (see \cite{EGH} sections $1.9.1, 1.9.2$); the only parts that need to be verified are that the moduli spaces have the correct boundary structure, which is done as follows:

\begin{itemize}
 \item Each proof involves the count of two level holomorphic buildings of cylinders, each in $\M(1;1:Z_L)$.  For example, in the proof $\partial^2 = 0$, one counts pairs of index $1$ cylinders in $\M_J(1;1:Z_L)$ which can be concatenated, $U_1 \odot U_2$.  By level-wise additivity (Proposition \ref{prop-intersectionproperties}), this configuration also has intersection number $0$ with $Z_L$.
 \item After choosing asymptotic markers, the pair can be glued, and by SFT-continuity of the intersection number (Lemma \ref{lem-continuityofint}), the glued cylinders must be in $\M(1;1:Z_L)$.
 \item By index additivity and transversality, the glued solutions belong to a connected component of the moduli space diffeomorphic to an interval (after quotienting the free $\R$-action in the $\R$-invariant case).  By homotopy invariance of the intersection number it is a subset of $\M(1;1:Z_L)$.
 \item At the other end of the interval, there is an SFT-limit.  The compactness Propositions \ref{prop-compactness} and \ref{prop-PBCcompactness} show that the limit is necessarily a broken trajectory with each cylinder in $\M(1;1:Z_L)$ again; transversality implies there are only two levels and determines the index of both cylinders, so it represents an object corresponding to another term in the algebraic identity.
 \item Comparison of the orientations shows that the terms representing opposite ends of the moduli space cancel, proving the required algebraic identity (Proposition \ref{prop-dsquared}, \ref{prop-chainmap}, or \ref{prop-chainhomotopy}).
\end{itemize}

In the cylindrical case the required transversality for non-simple curves is guaranteed by Theorem \ref{thm-transversality}, so $\partial_L^2 = 0$ even in non-simple homotopy classes $[a]$.

It is clear that all maps preserve the class $[a]$; else, some curve in the count (say $[U]$) is such that $\pi_V \circ U$ intersects $L$, but then $[U] *[Z] > 0$ because $\rm{int}(U,Z) \neq 0$ (see Theorem \ref{thm-posintersections}), contradicting $[U] \in \M_J(1;1:Z)$.
\end{proof}

\subsection{Consequences of chain maps/homotopies}

The following proposition implies Theorem \ref{thm-forcing1} as an obvious corollary; most of it follows already from Propositions \ref{prop-dsquared}, \ref{prop-chainmap}, \ref{prop-chainhomotopy}:

\begin{proposition}
\label{prop-functoriality}
 Given $(\lambda_+,L) \geq (\lambda_-,L)$ both satisfying $(E)$, and a simple class $[a]$ of loops in $V \backslash L$: for each generic $J \in \J_{\rm{gen}}(\lambda_-,\lambda_+:Z_L)$ there is a chain map
\[
\Phi^{[a]}_{-+}(J): CC^{[a]}_*(\lambda_+,J_+  \mbox{ rel } L) \rightarrow CC^{[a]}_*(\lambda_-,J_-  \mbox{ rel } L)
\]
\noindent Moreover, given two such chain maps $\Phi^{[a]}_{-+}(J_1),\Phi^{[a]}_{-+}(J_2)$ and a generic homotopy $J_{\tau} \in \J_{\tau, \rm{gen}}(\lambda_-,\lambda_+:Z_L)$, there is a chain homotopy operator $K(J_{\tau})$ i.e.
\[
\Phi^{[a]}_{-+}(J_1) - \Phi^{[a]}_{-+}(J_2) = K(J_{\tau}) \circ \partial_+ + K(J_{\tau}) \circ \partial_-
\]
\noindent Therefore, there are natural maps in homology 
\[
\left(\Phi^{[a]}_{-+} \right)_*: CCH^{[a]}_*(\lambda_+,J_+  \mbox{ rel } L) \rightarrow CCH^{[a]}_*(\lambda_-,J_-  \mbox{ rel } L)
\]
\noindent which furthermore satisfy
\[
 (\Phi^{[a]}_{-0})_* \circ (\Phi^{[a]}_{0+})_* = (\Phi^{[a]}_{-+})_*
\]
\noindent  If $(\lambda_-,L) \equiv (\lambda_+,L)$ then this map is an isomorphism.

Similarly, if $(\lambda_+,L,[a]) \sim (\lambda_-,L,[a])$ both satisfy $(PLC)$ and $[a]$ is simple then for each generic $J$ as above there is a chain map $\Phi^{[a]}_{-+}(J)$, and given two such maps and a generic homotopy between them $J_{\tau}$ there is a chain homotopy $K(J_{\tau})$ satisfying the above identity, so there are natural maps 
\[
\left(\Phi^{[a]}_{-+} \right)_*: CCH^{[a]}_*(\lambda_+,J_+  \mbox{ rel } L) \rightarrow CCH^{[a]}_*(\lambda_-,J_-  \mbox{ rel } L)
\]
\noindent which obey the above composition law and are also isomorphisms.
\end{proposition}

\begin{proof} The proof of Proposition \ref{prop-functoriality} is complete once the composition law is verified and we prove that the chain maps give isomorphisms.  The composition law is verified by a gluing/compactness argument, and we only need to check that we have compactness under stretching and enough transversality; the required compactness is guaranteed by Propositions \ref{prop-splitcompactness}, \ref{prop-PBCsplitcompactness}, and the required transversality is guaranteed by choosing $J_1,J_2 \in \J_{\rm{gen}}$ and considering only simple homotopy classes of loops.  Once the composition law is verified, one considers the compositions $\left( \Phi_{-+} (J_1) \circ \Phi_{+-}(J_2) \right)_* \cong \left( \Phi_{--}\right)_*$, and $\left( \Phi_{+-}(J_2) \circ \Phi_{-+}(J_1) \right)_* \cong \left( \Phi_{++}\right)_*$ (under condition $(E)$, both compositions can only be made if $(\lambda_+,L) \equiv (\lambda_-,L)$, hence the hypothesis in Proposition \ref{prop-functoriality}).  Using a cylindrical almost-complex structure we find that $\left( \Phi_{\pm \pm}\right)_*$ is the identity map.  This proves the maps $(\Phi_{\pm \mp})_*$ are both injective and surjective and therefore isomorphisms.
\end{proof}

\begin{remark} \label{remark-simple}
We would like to point out that the condition ``$[a]$ is simple'' can be relaxed to the condition ``$[a]$ is such that, for one of $\lambda_{\pm}$, all closed Reeb orbits in $[a]$ are simple''.  First, the chain complexes are defined even for $[a]$ not simple.  Then the condition that one of $\lambda_{\pm}$ has only simple closed orbits in $[a]$ guarantees that holomorphic cylinders in a cobordism used to define the chain maps above are somewhere injective \cite{HWZ_propsI, Siefring1} (hence also have an open set of somewhere injective points), and therefore chain maps/homotopies can be constructed for almost-complex structures in $\J_{\mathrm{gen}}$ just as above.  The conclusion of Proposition \ref{prop-functoriality} holds even if the contact form at the opposite end of the cobordism has multiply covered Reeb orbits in $[a]$.  We will use this observation in Section \ref{sec-examples}.
\end{remark}

\subsection{Action filtration}

As in Morse theory, the chain complexes can be filtered by action.  We consider the subspace $CC^{[a],\leq N}_*(\lambda, J \mbox{ rel } L)$ of $CC^{[a]}_*(\lambda, J \mbox{ rel } L)$ generated by orbits with action $\mathcal{A}(x) = \int_x \lambda \leq N$.  By Stokes' theorem, the differential decreases the action.  For $A \leq B$ (one may take $B = \infty$ as well here) there are natural chain maps
\[
 \iota_{B,A}: CC^{[a], \leq A}_*(\lambda,J \mbox{ rel } L) \rightarrow CC^{[a], \leq B}_*(\lambda,J \mbox{ rel } L), \qquad q_x \mapsto q_x 
\]
\noindent  and clearly $\iota_{C,B} \circ \iota_{B,A} = \iota_{C,A}$.  The morphisms $\Phi$ and homotopies $K$ respect the filtration: specifically, given $\Phi_{10}: CC^{[a]}_*(\lambda_0,J_0 \mbox{ rel } L) \rightarrow CC^{[a]}_*(\lambda_1,J_1 \mbox{ rel } L)$
\[
 \Phi_{10} \left(CC^{[a], \leq A}_*(\lambda_0,J_0 \mbox{ rel } L) \right) \subset CC^{[a], \leq A}_*(\lambda_1,J_1 \mbox{ rel } L)
\]
\noindent again by Stokes' theorem, and any two such maps (defined by different choices of $J$) are chain homotopic.

In fact, in order to construct $CC^{[a],\leq N}_*(\lambda, J \mbox{ rel }L)$ with $N < \infty$ one only requires that $\lambda$ satisfies $(E)$ up to action $N$.  This means that the assertion $(E)$ (resp.~$(PLC)$) about contractible in $V \backslash L$ Reeb orbits holds for orbits with action less than or equal to $N$.  Then even though $CC^{[a]}_*(\lambda \mbox{ rel } L)$ may not be defined, for every $A < N$ one can define the chain complexes
\[
 CC^{[a], \leq A}_*(\lambda, J \mbox{ rel } L)
\]
\noindent as above, and they are filtered by the action in the same way.  This follows exactly as the construction of the above chain complexes.  The proofs of compactness are carried out as straightforward extensions of Propositions \ref{prop-compactness}, \ref{prop-splitcompactness}, \ref{prop-PBCcompactness}, \ref{prop-PBCsplitcompactness}, using the fact that any contractible (in $V \backslash L$) Reeb orbit has action greater than $N$ to rule out bubbling for moduli spaces of holomorphic cylinders.  The Fredholm theory remains the same.  We will use these filtered chain complexes later.

\section{Forms with degeneracies}
\label{sec-degeneracies}
It is possible to obtain existence results which include cases where the form $\lambda$ may be degenerate or have closed orbits which are contractible in $V \backslash L$, namely Theorem \ref{thm-forcing2}.  The strategy is to take a sequence of small perturbations of a degenerate form to sufficiently non-degenerate approximating forms and deduce the existence of closed Reeb orbit for the approximating forms by a stretching argument.  Then, using the action filtration to bound the action of the orbits found this the way (which provides $C^1$-bounds), the Arzela-Ascoli theorem is used to find the desired Reeb orbit for the original degenerate form.  

We will also briefly justify the use of the Morse-Bott complex (when $\lambda$ satisfies $(PLC)$ or $(E)$ but is Morse-Bott non-degenerate) to compute $CCH^{[a]}_*([\lambda] \mbox{ rel } L)$ for $L \neq \emptyset$

\subsection{Perturbing degenerate contact forms}  The following Lemma is based on the proof in \cite{2008arXiv0809.5088C} Lemma $7.1$ (pages 36-37) of the fact that non-degeneracy is dense.

\begin{lemma}
 Suppose $(\lambda,L)$ is such that $L$ consists of non-degenerate closed Reeb orbits for $\lambda$ (including multiple covers).  Given any $N > 0$, there is a sequence $\lambda_n = f_n \cdot \lambda \rightarrow \lambda$ (i.e.~$f_n \rightarrow 1$ in $C^{\infty}(V;\R)$) such that $(\lambda_n,L) \equiv (\lambda,L)$ and such that all periodic orbits of the Reeb vector field for $\lambda_n$ of action at most $N$ are non-degenerate.
\end{lemma}

\begin{proof}
Around each Reeb orbit $x$ of action at most $N$ choose neighborhoods $x \subset U_x \subset V_x \subset V$ satisfying 
\begin{enumerate}
 \item For all such $x$, $V_x \cong S^1 \times D^2(2)$, and in these coordinates $U_x \cong S^1 \times D^2(1)$
 \item The Reeb vector field $X_{\lambda}$ is transverse to the pages $\{t \} \times D^2(2)$ of $V_x$
 \item For all $t \in S^1$, any orbit starting at a point $\{t\} \times D^2(1)$ (i.e.~starting at a point in $U_x$) stays inside $V_x$ for at least time $N+1$, and thus at least until the first return to $\{ t \} \times D^2(2)$ which occurs before $N + \epsilon$
\end{enumerate}
\noindent  If $x \notin L$ is an orbit of action at most $N$, then we choose $V_x$ so that, in addition, $V_x \cap L = \emptyset$.  
%
Moreover, select the neighborhoods $V_{L_i}$ around the components $L_i$ of $L$ such that the $V_{L_i}$ are mutually disjoint and contain no other orbits of period at most $N$ in $V_{L_i}$ (other than iterates of $L_i$).  This is possible because the $L_i$ and covers are assumed non-degenerate.

By the Arzel\`a-Ascoli theorem, the set $\Gamma_N$ of periodic orbits of action at most $N$ is compact.  By the compactness of the orbit set, it is not difficult to see that the image of these orbits in $V$ (also denoted $\Gamma_N$) may be covered by finitely many of the above sets $U_i$, i.e.~ $\Gamma_N \subset U_1 \cup \dots \cup U_{k}$.  Order them so that $U_{1},\dots, U_{m}$ are the neighborhoods of the form $U_{L_i} \subset V_{L_i}$ for components $L_i$ of $L$ among $\{U_1,\dots,U_k\}$.

We modify $\lambda$ on $V_1,\dots, V_k$ so that there are no degenerate orbits of period at most $N$ in these $V_i$ by the method of \cite{2008arXiv0809.5088C} Lemma $7.1$.  Choose any of the neighborhood pairs $(U_i,V_i)$.  Let $\mathcal{U}$ be any neighborhood of $\lambda$.  The proof in \cite{2008arXiv0809.5088C} shows how one can find a form $\lambda' \in \mathcal{U}$ such that 
\begin{enumerate}
 \item $\lambda'$ differs from $\lambda$ only on a compact subset of $V_i$,
 \item any $\lambda'$ orbit that originates in $U_i$ of action at most $N$ remains in $V_i$ for time $N + 1 - \epsilon$,
 \item all closed orbits of period at most $N$ contained in $U_i$ are non-degenerate,
 \item there are no closed orbits of action at most $N$ originating in the complement $V \backslash \bigcup_{j=1}^{k} U_j$ (for $\lambda_n'$ sufficiently near $\lambda$); therefore all closed orbits are still contained in $\bigcup_{j=1}^k U_j$.
\end{enumerate}

Proceeding one neighborhood pair at a time, one finds a form $\lambda^i \in \mathcal{U}$ such that $\bigcup_{j=1}^k U_j$ still covers all the closed Reeb orbits, but all closed Reeb orbits in $\bigcup_{j=1}^i U_j$ are non-degenerate (the above procedure shows how to make all Reeb orbits in $U_i$ non-degenerate without introducing new ones in $V \backslash \bigcup_{j=1}^{k} U_j$, and by choosing the perturbation small enough all Reeb orbits contained in $U_j$ for $j \leq i$ remain non-degenerate).  The form $\lambda^k$ will be the desired form.

However, we have to take care about the perturbations so that they satisfy the conclusions stated above.  By hypothesis all periodic orbits originating in $U_1,\dots,U_{m}$ of period at most $N$ are non-degenerate, by the choice of the $V_{L_i}$ and the choice of the ordering.  In this case, in the above process we do not need to perturb the form for the first $m$ steps, i.e.~ we can take $\lambda^m = \lambda$.  In subsequent steps, the new perturbations $\lambda^i, i > m$ differ from $\lambda^m$ only on $\bigcup_{j=m+1}^i V_j$, which is disjoint from $L$ by choice; in particular there is a neighborhood $\mathcal{N}$ containing $L$ on which $\lambda^i|_{\mathcal{N}} = \lambda|_{\mathcal{N}}$.  Therefore in the end $(\lambda^k, L) \equiv (\lambda,L)$.

To summarize, for any neighborhood $\mathcal{U}$ of $\lambda$, there is a $\lambda'$ such that
\begin{enumerate}
 \item $\lambda'$ has only non-degenerate periodic orbits of period at most $N$ contained in $U_1 \cup \dots \cup U_k$.
 \item $\lambda'$ has no periodic orbits of period at most $N$ which enter $V \backslash \bigcup_{i=1}^{k} U_i$.
 \item There is a neighborhood $\mathcal{N}$ of $L$ such that $\lambda'|_{\mathcal{N}} = \lambda|_{\mathcal{N}}$, thus $(\lambda',L) \equiv (\lambda,L)$
\end{enumerate}

The desired sequence is obtained by choosing a countable neighborhood basis $\mathcal{U}_n$ at $\lambda$ and choosing $\lambda_n \in \mathcal{U}_n$ as above.
\end{proof}

Let $[(\lambda,L)]$ denote the subset of the set of (positive) contact forms $\mu$ for $\xi$ (which we will denote $\Lambda(\xi)$) such that $(\mu,L) \equiv (\lambda,L)$, equipped with the $C^{\infty}$ topology as a subspace of $\Lambda(\xi)$.  It is a closed subset of $\Lambda(\xi)$ so its topology can be obtained from a complete metric space structure on it.  Let us denote by $\Lambda_{\mathrm{gen}}(\xi)$ the set of non-degenerate contact forms - as is well-known this is the intersection of a countable set of open dense sets (a $G_{\delta}$), and therefore dense in $\Lambda_{\mathrm{gen}}(\xi)$.  Then in fact

\begin{lemma} \label{lem-Gdelta}
 Suppose $\lambda$ is such that all orbits in $L$ are non-degenerate.  The set
\[
 [(\lambda,L)]_{\mathrm{gen}} :=  \left\{ \mu | \mu \in [(\lambda,L)], \mu \in \Lambda_{\mathrm{gen}}(\xi) \right\}
\]
\noindent is the intersection of a countable collection of open, dense sets (in $[(\lambda,L)]$) and therefore also dense.
\end{lemma}

\begin{proof}
Let $\widetilde{G_N}$ denote the subset of $\Lambda(\xi)$ for which all periodic orbits of action $\leq N$ are non-degenerate, and $G_N = \widetilde{G_N} \cap [(\lambda,L)]$.  The argument of \cite{2008arXiv0809.5088C} Lemma $7.1$ shows that $\widetilde{G_N}$ is open in $\Lambda(\xi)$, and therefore $G_N$ is also open in the subspace $[(\lambda,L)]$.  In the previous Lemma it was shown that $G_N$ is dense in $[(\lambda,L)]$.  The conclusion follows from the observation that 
\[
 [(\lambda,L)]_{\mathrm{gen}} =  \bigcap_{N = 1}^{\infty} G_N
\]
\noindent is a $G_{\delta}$.
\end{proof}

Finally, we note the following

\begin{lemma} \label{lem-2jetpert}
 If $\lambda$ is a contact form and $L$ is a closed link for the Reeb vector field, then in any $C^{\infty}$-neighborhood $\mathcal{U}$ of $\lambda$ there is a $\lambda'$ such that $(\lambda',L) \sim (\lambda,L)$ and each component of $L$ (including multiple covers) is non-degenerate for $\lambda'$.
\end{lemma}

This can be achieved by an arbitrarily small perturbation to the $2$-jet of $\lambda$ along $L$ (leaving the $1$-jet unperturbed so $L$ is still tangent to the Reeb vector field).

\subsection{The proof of the implied existence Theorem \ref{thm-forcing2}} \label{sec-forcing}

We prove Theorem \ref{thm-forcing2}:

\begin{proof}
Let $\lambda'$ be any form such that $L$ is closed for the Reeb vector field and elliptic non-degenerate.  Suppose $(\lambda,L) \equiv (\lambda',L)$ is non-degenerate and satisfies $(E)$, and by rescaling if necessary that $\lambda' \prec \lambda$.  Choose a constant $c$ such that also $c \lambda \prec \lambda'$.  Choose $J \in \J_{\mathrm{gen}}(\lambda), J' \in \J_{\mathrm{gen}}(\lambda')$, and $J_1 \in \J_{\mathrm{gen}}(J,J':Z_L)$, $J_0 \in \J_{\mathrm{gen}}(J',J:Z_L)$ (to be specific $J_0$ is equal to $J$ on $W^-(c \lambda)$).  Then we have the path $J_R, R\geq 0$ of almost-complex structures in the symplectization splitting along $\lambda'$ to $(W_{\xi},J_0) \odot (W_{\xi},J_1)$.

\emph{Claim: for each $R$ sufficiently large there is a $J_R$-holomorphic cylinder $U$ such that $[U] * [Z_L] = 0$ with positive asymptotic orbit having $\lambda$-action at most $N$ and asymptotic limits in $[a]$.}  To see why this is so, we may find a sequence $J_n \rightarrow J_R$ with each $J_n \in \J_{\mathrm{gen}}$ (since $\J_{\mathrm{gen}}$ is Baire), and each such defines a morphism 
\[
 \Phi(J_n): CC^{[a],\leq N}(\lambda,J \mbox{ rel } L) \rightarrow CC^{[a], \leq N}(c \cdot \lambda,J \mbox{ rel } L)
\]
\noindent Such a chain map is chain homotopic to the morphism obtained from a cylindrical almost-complex structure $\widehat{J}$, which can be identified with inclusion
\begin{align*}
 \iota_{N/c,N}: CC^{[a],\leq N}(\lambda,J \mbox{ rel } L) & \rightarrow CC^{[a], \leq N}(c \cdot \lambda,J \mbox{ rel } L) \cong CC^{[a], \leq N/c}(\lambda,J \mbox{ rel } L) \\
 q_{x} &\mapsto q_{x}
\end{align*}
\noindent so $\Phi(J_n)$ is non-zero if the homology is non-zero.  Let $q$ be a $\partial_J$-closed element of $CC^{[a]}_*(\lambda,J \mbox{ rel } L)$ generated by orbits of action $\leq N$ such that $[q] \neq 0$ when considered as an element of $CCH^{[a]}_*(\lambda,J \mbox{ rel } L)$.  Since $\Phi(J_n)$ is chain homotopic to the identity and $q$ is closed but not exact, it is easy to deduce that $\Phi(J_n)(q) \neq 0$, so there must be a $J_n$ finite energy holomorphic cylinder $U_n$ with $[U_n] *[Z_L] = 0$ and positive asymptotic orbit of action at most $N$.  Taking a SFT-limit as $n\rightarrow \infty$, by Proposition \ref{prop-compactness} (and using the fact that $(\lambda,L), (c\lambda,L)$ satisfy $(E)$) we obtain a $J_R$-holomorphic cylindrical building, which has a component $U_R$ which is a $J_R$-holomorphic cylinder with $[U_R] * [Z_L] = 0$, positive asymptotic orbit of action at most $N$, and asymptotic limits in $[a]$ as claimed.

Since $R$ was arbitrary, we make take a sequence of $J_{R_n}$ finite energy cylinders, $U_n$,  with $R_n \uparrow \infty$, each with positive asymptotic orbit of action at most $N$, $[U_n] * [Z_L] = 0$, and both asymptotic limits in $[a]$.  To continue, first suppose that $\lambda'$ is non-degenerate up to action $N$.  Then by SFT-compactness we obtain a limit holomorphic building $(U) = (U)_0 \odot (U)_1$ in $(W_{\xi},J_0)\odot (W_{\xi},J_1)$.  Proposition \ref{prop-splitcompactnessNOT-E}  asserts that the compactified image $\pi_V \circ \overline{(U)} \subset V \backslash L$.  The image is a piecewise-smooth cylinder with one boundary component equal to $x$ and the other boundary component equal to $y$.  Suppose $(U)_1$ has $k$ free negative punctures.  Let $(U)_0^1,\dots,(U)_0^{k-1}$ be the planar components of $(U)_0$.  Consider now the connected sub-building
\[ 
(X) = \left( (U)_0^1 + \dots (U)_0^{k-1} \right) \odot (U)_1
\]
\noindent with one free negative puncture which is a negative asymptotic orbit for $(U)_1$, and therefore a closed Reeb orbit for $\lambda'$, say $z$.  The compactified image $\pi_V \circ \overline{(X)}$ is a compact cylinder with boundary components $x$ and $z$ and image $\subset V \backslash L$, and therefore $[z] = [x] = [a]$.  Thus $z$ is the desired closed $\lambda'$-Reeb orbit in $[a]$ with action at most $N$ (since action is non-increasing as one descends levels).

To conclude in the case $\lambda'$ is degenerate, select $\lambda_n' \rightarrow \lambda'$ in $C^{\infty}$ with $(\lambda_n',L) \equiv (\lambda',L)$ such that $\lambda_n'$ is non-degenerate up to action $N$ by Lemma \ref{lem-Gdelta}.  For each $n$ sufficiently large we obtain a closed $\lambda_n'$-Reeb orbit $x_n$ as above (since eventually $\lambda_n' \prec \lambda$).  Applying the Arzela-Ascoli theorem, we obtain a closed $\lambda'$-Reeb orbit $x$ as a $C^{\infty}$ limit of the $x_n$.  %
$x$ cannot be a $l$-fold cover of a component of $L$ for any $l \geq 1$, because if it were then one could show that $1$ is an eigenvalue for the corresponding linearized $l$th-return map of the Reeb flow, contradicting the hypothesis that each cover of $L$ is elliptic non-degenerate.  %
It follows easily that $[x] = [a]$, since $x_n \rightarrow x$ in $C^{\infty}$, any loop $C^0$-close enough to $x$ is in the same homotopy class as $x$, and every $x_n \in [a]$.  The orbit $x$ also has action at most $N$.

If instead $L$ satisfies the linking condition of the Theorem and $[a]$ is a proper link class for $L$, select any $(\lambda',L,[a])$ satisfying $(PLC)$.  Assuming first $\lambda$ is non-degenerate, repeat the above argument using Proposition \ref{prop-PBCsplitcompactnessNOT-PLC} instead.  To pass to the degenerate case, use Lemmas \ref{lem-2jetpert}, \ref{lem-Gdelta} to similarly find approximating non-degenerate Reeb vector fields and a sequence of orbits $x_n \in [a]$ with $x_n \rightarrow x$ with $x$ a closed $\lambda$ Reeb orbit.  It is clear that $x$ cannot be a cover of a component of $L$ since no sequence of loops in $[a]$ can converge to a cover of a component of $L$: this is because $[a]$ is a proper link class relative to $L$.
\end{proof}

\subsection{A brief justification of Morse-Bott computations}

Suppose that $\lambda$ satisfies $(E)$ and is Morse-Bott non-degenerate (a similar argument will work if $\lambda$ satisfies $(PLC)$ instead).  Choose a Morse function on the Reeb orbit sets of $\lambda$.  The Morse-Bott chain complex $CC^{[a]}_*(\lambda,J \mbox{ rel } L)$ is generated by critical points of the Morse function chosen on the orbit sets, removing the generators corresponding to Reeb orbits with image in $L$.  The differential counts generalized holomorphic cylinders (holomorphic cylinders between orbit sets with cascades along the gradient flow lines of the Morse function, described in detail in \cite{Bourgeois_thesis}) of index $1$ which neither intersect nor are asymptotic to $L$.

Using the perturbation of Lemma $2.3$ of \cite{Bourgeois_thesis} one finds, given any $N>0$ and $C^{\infty}$-neighborhood of $\lambda$, a form $\lambda_N$ (which in fact only differs from $\lambda$ on an arbitrarily small tubular neighborhood the orbits sets of action at most $N + 1$) such that

\begin{itemize}
\item $\lambda_N$ is in the $C^{\infty}$-neighborhood chosen
\item All closed orbits that have period at most $N+1$ correspond to critical points of a Morse function on the orbit sets of $\lambda$ and are non-degenerate.
\item There are no contractible (in $V \backslash L$) Reeb orbits of period at most $N$ (because a sequence of such orbits for a sequence of forms converging to $\lambda$ would produce a contractible in $V \backslash L$ Reeb orbit for $\lambda$ of action at most $N$ by Arzela-Ascoli, contradicting hypotheses)
\end{itemize}

Such forms can be given explicitly in terms of the Morse functions on the orbit set in \cite{Bourgeois_thesis}.  Given such a $\lambda_N$, one may construct the cylindrical chain complex (for generic $J_N$ and Morse functions)
\[
 \left( CC^{[a], \leq N}_*(\lambda_N), \partial(\lambda_N,J_N) \right)
\]
\noindent  For any fixed $N$, choosing a sequence of such $\lambda_n \rightarrow \lambda$ (with corresponding $J_n \rightarrow J$), the work of \cite{Bourgeois_thesis} guarantees that for large $n$ the cylindrical contact chain complex generated by orbits of action at most $N$ is canonically identified with the Morse-Bott complex for $(\lambda,J)$ generated by the orbit sets of action at most $N$, so this portion of the homology can be canonically identified.  It is not difficult to see that a sequence $U_n$ of $J_n$-holomorphic cylinders with $U_n * Z_L = 0$ will converge to a generalized holomorphic cylinder of the type used to construct the differential above.  Similarly, if any sequence converges to such a generalized holomorphic cylinder, it is easy to see that $\lim U_n * Z_L = 0$.  The results of \cite{Bourgeois_thesis} imply that for large $n$
\[
 \left( CC^{[a], \leq N}_*(\lambda_n \mbox{ rel } L), \partial_L(\lambda_n,J_n) \right) \cong \left( CC^{[a], \leq N}_*(\lambda \mbox{ rel } L), \partial_L(\lambda,J) \right)
\]
\noindent as chain complexes.

Now given any non-degenerate form $\lambda'$ satisfying $(E)$ and $(\lambda',L) \equiv (\lambda,L)$, and $J' \in \J_{\mathrm{gen}}(\lambda')$, we have chain maps (choosing $C > 1$ so $C \lambda \succ \lambda'$, $c < 1$ so $c \lambda \prec \lambda'$, and $n$ large enough) 
\[
 CC^{\leq N/C}_*(\lambda_n,J_n \mbox{ rel } L) \rightarrow CC^{\leq N}_*(\lambda',J' \mbox{ rel } L) \rightarrow CC^{\leq N/c}_* (\lambda_n,J_n \mbox{ rel } L)
\]
\noindent with the composition chain homotopic to the filtration inclusion $\iota_{N/c,N/C}$.  Given any non-zero class in $CCH^{[a]}_*(\lambda \mbox{ rel }L)$, it is represented in $CC^{[a], \leq N/C}_*(\lambda_n)$ and $CC^{[a], \leq N/c}_*(\lambda_n)$ by a closed, non-exact element (if $N$ was chosen large enough and $\lambda_n$ is chosen near enough $\lambda$).  Since the composition above is chain homotopic to $\iota$, we get a closed, non-exact element in $CC^{[a], \leq N}_*(\lambda' \mbox{ rel }L)$.  For any two such choices of representatives, we can find $N$ large enough so the elements represent the same (non-zero) class in $CCH^{[a], \leq N}_*(\lambda' \mbox{ rel }L)$.  If this class were zero in $CCH^{[a]}_*(\lambda',J' \mbox{ rel }L)$, then it would be exact in $CC^{[a], \leq N}_*$ for some $N$.  Thus we have an inclusion in homology $CCH^{[a]}_*(\lambda,J \mbox{ rel } L) \hookrightarrow CCH^{[a]}_*(\lambda',J' \mbox{ rel } L)$.

Similar considerations using instead the composition of morphisms
\[ 
  CC^{\leq N/C'}_*(\lambda',J' \mbox{ rel }L) \rightarrow CC^{\leq N}_*(\lambda_n,J_n \mbox{ rel }L) \rightarrow CC^{\leq N / c'}_*(\lambda',J' \mbox{ rel }L) 
\] 
\noindent can be used to find a surjective map on homology 
\[ 
  CCH^{[a]}_*(\lambda,J \mbox{ rel }L) \twoheadrightarrow CCH^{[a]}_*(\lambda',J' \mbox{ rel }L)
\]

\section{Examples in $S^3$} \label{sec-examples}

We will use two approaches to computing contact homology on some orbit complements in $S^3$.  The first is to use an integrable model for which the dynamics are known precisely, usually because of some symmetry, and compute using the symmetry to make deductions about holomorphic curves as well.  The Morse-Bott technique introduced in \cite{Bourgeois_thesis} is a particularly useful tool when computing using such an approach.  The second technique is to use open book decompositions (an approach used in \cite{2008arXiv0809.5088C} to compute contact homology for a large class of examples).  In this section we will give sample computations using both approaches.

Since $H_2(S^3;\Z) = 0$ we can use $\Q$ as the coefficient ring for all chain complexes, and a global trivialization of the contact structure is used to compute all Conley-Zehnder indices (and thus to grade the chain complexes) for the tight contact structure.

\subsection{Computations with Morse-Bott contact forms}

First, consider the ``irrational ellipsoids'' (Example \ref{example-1}).  Since there is a unique orbit in each homotopy class (labeled by linking number $\ell$) of loops in $S^3 \backslash P'$ it is trivial to compute
\[
 CCH_*^{\ell}(\lambda',J \mbox{ rel } P') = \Q \cdot q_{Q'^{\ell}}, \qquad  CCH_*^{\ell}(\lambda',J \mbox{ rel } P',Q') = 0
\]
\noindent In the following examples we will consider Morse-Bott non-degenerate forms, for which we apply the techniques of \cite{Bourgeois_thesis} to compute the contact homology.

We extend Example \ref{example-1} as follows.  Let $\theta_1, \theta_2$ be any irrational (possibly negative) real numbers.  Let $\gamma(t) = (x(t), y(t))$ for $t \in [0,1]$ be a smooth embedded parameterized curve in the first quadrant of $\R^2$.  Suppose that $\gamma$ has the following properties:
\begin{itemize}
 \item $x(0) > 0, y(0) = 0$, and $y'(0) > 0$;
 \item $x(1) = 0, y(1) > 0$, and $x'(1) < 0$;
 \item The ratio $y'(t)/x'(t)$ is strictly monotone (has non-zero first derivative) on the subdomains of $[0,1]$ for which it is defined
 \item $x \cdot y' - x' \cdot y > 0$ for all $t \in [0,1]$.  Equivalently, $\gamma$ and $\gamma'$ are never co-linear.
 \item $- \frac{x'(0)}{y'(0)} = \theta_1$;
 \item $- \frac{x'(1)}{y'(1)} = \theta_2$.
\end{itemize}

Given such a curve, we can construct a star-shaped hypersurface in $\C^2 = \R^4$ (which is thus of contact type) as follows:

\begin{example} \label{example-2}
Given such a $\gamma$, we define the surface $S = S_{\gamma}$ as (where $r_i,\theta_i$, $i = 1,2$, are polar coordinates in each $\R^2$ factor of $\R^4 = \R^2 \times \R^2$):
\[
 S = S_{\gamma} = \{(r_1,\theta_1, r_2, \theta_2) \in \R^4 | (r_1^2,r_2^2) \in \gamma \}
\]
\noindent with the contact form $\lambda' = \lambda_0 |_S$.  Alternately, it can be viewed as the contact form $f^2 \cdot \lambda_0$ on $S^3 = \{r_1^2 + r_2^2 = 1\}$, where $f:S^3 \rightarrow (0,\infty)$ satisfies $f(z) \cdot z \in S_{\gamma}$.  There are closed orbits $H_1 = S^3 \cap \C \times \{0\}$ and $H_2 = S^3 \cap \{0\} \times \C$.
\end{example}

The closed orbits $H_1$ and $H_2$ have Conley-Zehnder indices
\[
 CZ(H_1^k) = 2 \lfloor k(1 + \theta_1) \rfloor + 1, \qquad CZ(H_2^k) = 2 \lfloor k(1 + 1/\theta_2) \rfloor + 1
\]

It is not difficult to convince oneself that, once given real numbers $\theta_1, \theta_2$, such a curve $\gamma$ can be constructed.  If $\theta_1 = \theta_2 > 0$ are irrational, if one took $\gamma$ to be a straight line then $S$ would be the irrational ellipsoid (though, strictly speaking, it violates the hypothesis that $y'/x'$ be strictly monotone).

\begin{figure}
\begin{center}
\includegraphics[width=85mm, angle=270]{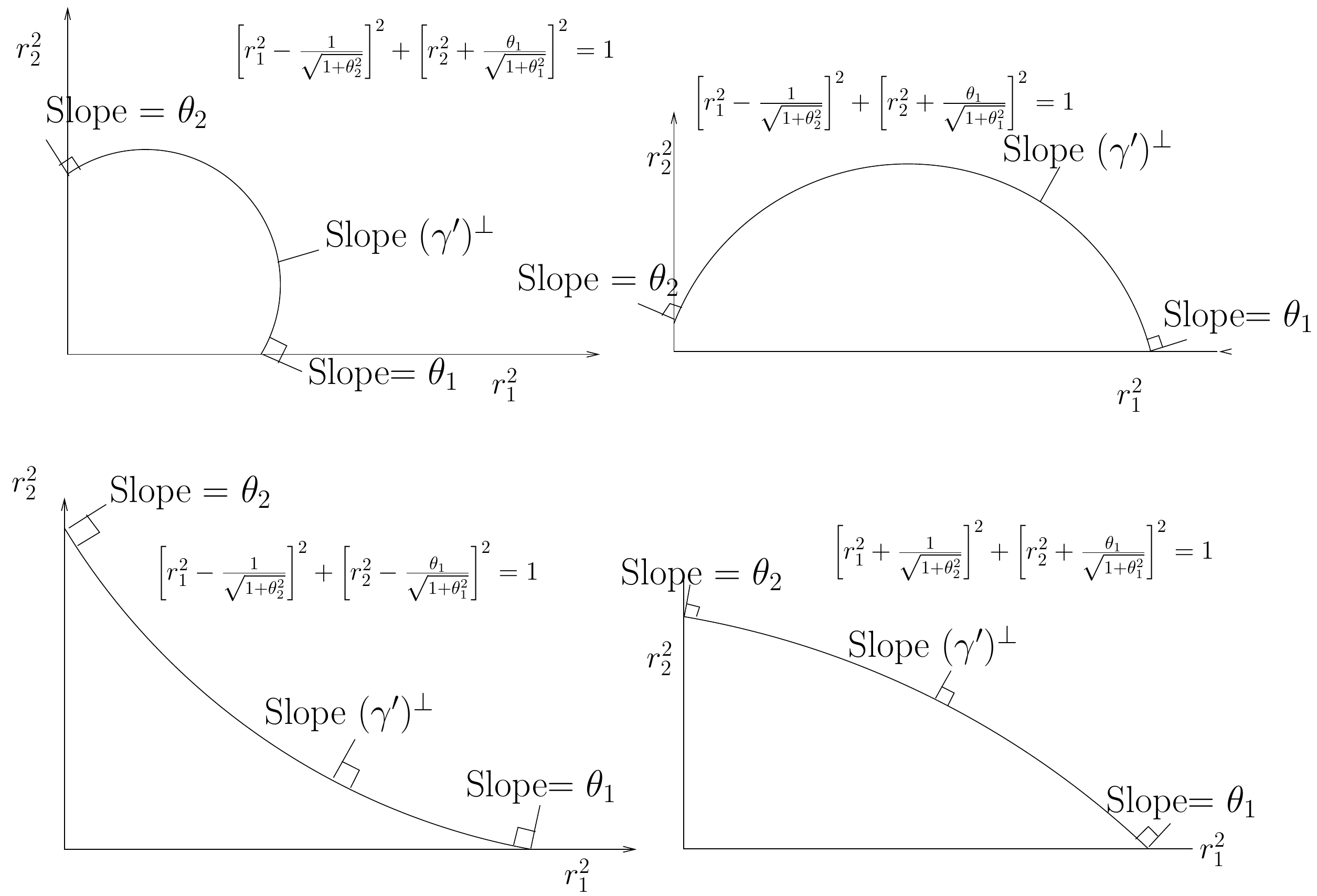} 
\end{center}
\caption{Curves $\gamma$ (given as level sets) with various endpoint normal slopes $\theta_1, \theta_2$.}
\end{figure}

First we will show that $S$ is a smooth, star-shaped hypersurface with respect to the origin in $\R^4$.  Let $H(x,y)$ be any smooth function such that $H \circ \gamma \equiv 1$ and $\nabla H \neq 0$ along $\gamma$.  Then $S$ is realized as the set $H(r_1^2, r_2^2) = 1$ (at least in some neighborhood of $S$).  We compute
\begin{align*}
r_1 \partial_{r_1} +  r_2 \partial_{r_2} &\neg d(H(r_1^2,r_2^2)) \\
&= r_1 \partial_{r_1} +  r_2 \partial_{r_2} \neg 2r_1 (\partial_x H)(r_1^2,r_2^2) dr_1 + 2r_2 (\partial_y H)(r_1^2,r_2^2) dr_2 \\
 &= 2 \left( r_1^2 (\partial_x H) + r_2^2 (\partial_y H) \right)
\end{align*}
\noindent  However, since $dH \circ \gamma' \equiv 0$ we mush have that $\nabla H (\gamma(t))$ is proportional to $j \cdot \gamma'(t) = (y'(t),-x'(t))$ (since this vector is perpendicular to $(x',y')$) and therefore, using $x(t) = r_1^2, y(t) = r_2^2$, the previous line is proportional to
\[
 2 (x (y') + y (-x')) = 2 (x \cdot y' - y \cdot x')  > 0
\]
\noindent  The proportionality constant is nowhere zero, so we see that the original quantity is nowhere vanishing and we can conclude that $dH \neq 0$ and hence $S$ is a smooth hypersurface in $\R^4$.  The radial vector field $r_1\partial_{r_1} + r_2 \partial_{r_2}$ is transverse to $S$ and therefore $S$ is star-shaped with respect to the origin.  Since the radial vector field is Liouville, the hypersurface is of contact type.  

The Hamiltonian vector field of $H$ is
\begin{align*}
X_H &= 2(\partial_x H) \partial_{\theta_1} + 2(\partial_y H) \partial_{\theta_2}  \\
& \propto \left( y'(t) \partial_{\theta_1} - x'(t) \partial_{\theta_2} \right)
\end{align*}
\noindent  So $y'(t) \partial_{\theta_1} -x'(t) \partial_{\theta_2}$ positively generates the characteristic foliation.

It is clear that $r_1, r_2$ are integrals for the flow, so $S$ foliates (on the complement of the Hopf link $H_1 \sqcup H_2 = (r_1 = 0 \sqcup r_2 =0)$) by tori, parameterized by points on the curve $(x(t),y(t)) = \gamma(t)$ on which the flow has slope $(y'(t),-x'(t))$.  In particular, we find for each $t$ such that
\[
( y'(t),-x'(t)) = C \cdot (p,q), \quad C > 0
\]
\noindent is rational a torus foliated by closed Reeb orbits.  The signs of $p,q$ will be determined by the direction of the vector $\gamma'(t)$.  The condition that $x'(t)/y'(t)$ is strictly monotone implies that this torus is Morse-Bott, so we have a Morse-Bott flow on $S$.  We shall see that closed orbits in different tori represent different homotopy classes of loops in $S^3 \backslash (H_1 \sqcup H_2)$ (which is homotopy equivalent to $T^2$).  In the following, let us denote the homotopy class of loops $[a]$ such that for $l \in [a]$ $\ell(l,H_2) = p$ and $\ell(l,H_1) = q$ by $[a] = (p,q) \in \Z^2$: this classifies the set of homotopy classes of loops in $S \backslash (H_1 \sqcup H_2)$.

Let us more explicitly describe the set of closed characteristics in $S$.  First, the knots $H_1$, $H_2$ corresponding respectively to the sets $r_2 = 0$, $r_1 = 0$ are closed Reeb orbits on $S$.  Using that the characteristic foliation is positively generated by $y'(t) \partial_{\theta_1} - x'(t) \partial_{\theta_2}$, one computes the Conley-Zehnder indices:
\[
 CZ(H_1^k) = 2 \lfloor k(1 - \frac{x'(0)}{y'(0)})\rfloor + 1 = 2 \lfloor k(1 + \theta_1 ) \rfloor + 1
\]
\[
 CZ(H_1^k) = 2 \lfloor k(1 - \frac{y'(1)}{x'(1)})\rfloor + 1 = 2 \lfloor k(1 + 1/\theta_2 ) \rfloor + 1
\]

The characteristic foliation on the torus at $(r_1^2,r_2^2) = \gamma(t)$ is (positively) generated by
\[
 X_H = y'(t) \partial_{\theta_1} - x'(t) \partial_{\theta_2}
\]
\noindent  The curve $\gamma$ divides into three distinct arcs: on the first arc $t \in [0,t_1)$ (for some $0 \leq t_1 < 1$), $x'(t) > 0$ and $y'(t) > 0$; on the second arc $t \in (t_1,t_2)$ (where $t_1 < t_2 \leq 1$), $x'(t) < 0$ and $y'(t) > 0$; on the third arc, $t \in (t_2,t_3]$ (where $t_2 \leq t_3 \leq 1$), $x'(t) < 0$ and $y'(t) < 0$.  Note that the first and third arcs may be empty.  We analyze the Reeb orbits belonging to points on each arc separately.

\textbf{On the first arc:}  $y'(t) > 0, x'(t) > 0$, so the ratio $-x'(t)/y'(t)$ is negative and increases until $t_1$ where $x'(t_1) = 0$.  

Then for those $t$ for which $-x'/y'$ is rational and equal to $(-q)/p$ in least terms (with $p,q > 0$), the leaves of the foliation form closed Reeb orbits which link $p$ times with $H_2$ and $-q$ times with $H_1$, i.e. it is in the homotopy class $(p,-q)$.  

This arc is non-empty if and only if $\theta_1 < 0$.  In this case, since $ -x'(t)/y'(t)$ is initially $\theta_1$, and strictly increases.  Thus, we find, for each relatively prime positive pair of integers $(p,q)$ such that 
\[ \theta_1 < (-q) / p < 0 \] 
\noindent a unique torus (corresponding to that $t_{(p,-q)}$ for which $-x'(t_{(p,-q)})/y'(t_{(p,-q)}) = (-q)/p$) foliated by Reeb orbits in the homotopy class $(p,-q)$.

\textbf{On the second arc:} $y'(t) > 0, x'(t) < 0$, so the ratio $y'(t)/x'(t)$ is negative and strictly increasing until $t_2$.  Either $t_2 = 1$, or, $t_2 < 1$ and $y'(t_2) = 0$.

Then, for those $t$ for which $-x'/y'$ is rational and equal to $q/p$ in least terms (with $p,q > 0$, the leaves of the foliation form closed Reeb orbits which link $p$ times with $H_2$ and $q$ times with $H_1$, i.e. it is in the homotopy class $(p,q)$.  

This arc is always non-empty.  Let us consider the case $-x'(t)/y'(t)$ is increasing, first.  If $\theta_1 > 0$, then $ -x'(t)/y'(t)$ is initially $\theta_1$; else, it is initially zero.  If $\theta_2 > 0$, then $t_2 = 1$ and $-x'(t)/y'(t)$ tends towards $\theta_2$ as $t \rightarrow 1^-$.  Else, $t_2 < 1$ and $-x'(t)/y'(t)$ is unbounded.

Thus, we find, for each relatively prime positive pair of integers $(p,q)$ such that 
\[
 \max \{ 0, \theta_1 \} < q / p <  \left \{ \begin{matrix} \theta_2, & \mbox{if } \theta_2 > 0 \\ \infty, & \mbox{if } \theta_2 < 0 \end{matrix} \right.
\]
\noindent a unique torus (corresponding to the point $\gamma(t_{(p,q)})$ where $-x'(t_{(p,q)})/y'(t_{(p,q)}) = q/p$) foliated by Reeb orbits in the homotopy class $(p,q)$.

There is also the possibility that $-x'(t)/y'(t)$ is \emph{decreasing}, which happens if and only if $0 < \theta_2 < \theta_1$.  In this case, we find for each
\[
 \theta_2 = \frac{-x'(1)}{y'(1)} < q/p = -x'(t_{(p,q}) / y'(t_{(p,q}) < \frac{-x'(0)}{y'(0)} = \theta_1
\]
\noindent the torus $(r_1^2,r_2^2) = \gamma(t_{(p,q)})$ foliated by orbits in the homotopy class $(p,q)$.  


\textbf{On the third arc:} both $x'(t)$ and $y'(t)$ are negative, and $y'(t)/x'(t)$ increases strictly until $t = 1$.

For those $t$ for which $-x'/y'$ is rational and equal to $q/(-p)$ in least terms (with $p,q > 0$, the leaves of the foliation form closed Reeb orbits which link $-p$ times with $H_2$ and $q$ times with $H_1$, i.e. it is in the homotopy class $(-p,q)$.  

This arc is non-empty if and only if $\theta_2 < 0$.  In this case, $ -x'(t)/y'(t)$ tends to $-\infty$ as $t \rightarrow t_2^+$, and strictly increases to the limiting value $\theta_2$ as $t \rightarrow 1$.  Thus, we find, for each relatively prime positive pair of integers $(p,q)$ such that 
\[ q / (-p) < \theta_2 \] 
\noindent a unique torus (corresponding to that $t_{(p,-q)}$ for which $-x'(t_{(-p,q)})/y'(t_{(-p,q)}) = q/(-p)$) foliated by Reeb orbits in the homotopy class $(-p,q)$.

Finally, there are the points $t_1, t_2$.  If $t_1 \neq 0$, then $x'(t_1) = 0$ and $y'(t_1) > 0$ so there will be a torus of closed orbits at $\gamma(t_1)$ representing the homotopy class $(1,0)$ (i.e.~linking once with $H_2$ and zero times with $H_1$).  If $t_2 \neq 1$, then $x'(t_2) < 0$ and $y'(t_2) = 0$ so there will be a torus of closed orbits at $\gamma(t_2)$ representing the homotopy class $(0,1)$ (i.e.~linking $0$ times with $H_2$ and once with $H_1$).

Thus, we can classify the closed characteristics on $S$ as follows:
\begin{enumerate}
\item If $0< \theta_1,\theta_2$ then for
\[
 \frac{q}{p} \in (\theta_1,\theta_2) \sqcup (\theta_2, \theta_1)
\]
\noindent there is a torus of closed orbits with each closed orbit representing the homotopy class $(p,q)$;
 \item If $\theta_1 < 0 < \theta_2$, then for each fraction (written in least terms)
\[
 \frac{q}{p} \in (\theta_1,\theta_2)
\]
\noindent (taking $q < 0$ when this fraction is negative) there is a torus of simple closed orbits with each closed orbit representing the homotopy class $(p,q)$;
 \item If $\theta_2 < 0 < \theta_1$, then for each fraction (written in least terms)
\[
 \frac{p}{q} \in (\frac{1}{\theta_2},\frac{1}{\theta_1})
\]
\noindent (taking $p < 0$ when this fraction is negative) there is a torus of simple closed orbits with each closed orbit representing the homotopy class $(p,q)$;
\item If $\theta_1, \theta_2 < 0$, then for $p,q$ non-negative and relatively prime such that
\[
 \frac{-q}{p} \in (\theta_1,0), \qquad \frac{q}{p} \in [0,+\infty], \qquad \frac{q}{-p} \in (-\infty,\theta_2)
\]
\noindent (where we consider $1/0 = +\infty$), there is a torus of simple closed orbits with each closed orbit representing the homotopy class $(p,-q)$, $(p,q)$ or $(-p,q)$ (respectively).
\item  In each case above, these are all the closed orbits i.e.~there are no other closed orbits (besides the closed orbits $H_1$, $H_2$).
\end{enumerate}

\noindent (The reader may have noticed that case $(3)$ above is superfluous, since by interchanging the roles of $H_1, H_2$ it can be realized by case $(2)$.)  In particular, there are no closed Reeb orbits which are contractible in $S \backslash H_1 \sqcup H_2$, and $H_1,H_2$ are elliptic, so $CCH_*^{[a]}(S \mbox{ rel } H_1 \sqcup H_2)$ exists.  By the Morse-Bott calculation we find that for $[a] = (p,q)$ satisfying the above conditions, $CCH_*^{(p,q)}$ is isomorphic to the $H_1(S^1)$ up to a grade shift.

Thus, given a transverse link Hopf link $H_1 \sqcup H_2$ which realizes the maximal Thurston-Bennequin bound, if the Conley-Zehnder indices of $H_1$ and $H_2$ satisfy
\[
 CZ(H_1^k) = 2 \lfloor k(1 + \theta_1) \rfloor + 1, \qquad CZ(H_2^k) = 2 \lfloor k(1 + 1/\theta_2) \rfloor + 1
\]
\noindent then we have computed
\begin{enumerate}
 \item If $\theta_1, \theta_2$ are both positive and $p,q$ are relatively prime, then
\[
 CCH^{(p,q)}_*(S \mbox{ rel } H_1 \sqcup H_2) = \left \{ \begin{matrix} \Q^2, & \mbox{if } p,q > 0 \mbox{ and } \frac{q}{p} \in (\theta_1,\theta_2) \cup (\theta_2,\theta_1) \\ 0, & \mbox{otherwise} \end{matrix} \right.
\]
\noindent  If $\theta_1 < \theta_2$, then the generators in the homotopy class $(p,q)$ have grading $2(p+q), 2(p+q) + 1$, and if $\theta_2 < \theta_1$, then the generators have grading $2(p+q)$, $2(p+q) - 1$ instead.

 \item If $\theta_1 < 0 < \theta_2$ (relabel $L_1, L_2, \theta_1, \theta_2$ appropriately so that this is true, if necessary) then for relatively prime $p,q$:
\[
 CCH^{(p,q)}_*(S \mbox{ rel } H_1 \sqcup H_2) = \left \{ \begin{matrix} \Q^2, & \mbox{if } p > 0 \mbox{ and } \frac{q}{p} \in (\theta_1,\theta_2)  \\ 0, & \mbox{otherwise} \end{matrix} \right.
\]
\noindent  The generators in the homotopy class $(p,q)$ have grading $2(p+q), 2(p+q) + 1$

 \item If $\theta_1, \theta_2$ are both negative, then for $p,q$ relatively prime such that:
\[
p > 0 \mbox{ and } \frac{q}{p} \in (\theta_1,1]; \mbox{ or } q > 0 \mbox{ and } \frac{p}{q} \in (\frac{1}{\theta_2},1]
\]
\noindent we have
\[
 CCH^{(p,q)}_*(S \mbox{ rel } H_1 \sqcup H_2) \cong \Q^2
\]
\noindent and for all other $(p,q)$ it is zero.  The generators in the homotopy class $(p,q)$ have grading $2(p+q), 2(p+q) + 1$.
\end{enumerate}

\begin{example}
 \label{example-3}  We consider the same form $\lambda'$ as above, and (for simplicity) assume that $0 < \theta_1 < \theta_2$.  Consider instead the two-component link $T = T_1 \sqcup T_2$ whose components are 
\begin{itemize}
 \item $T_1$ is the one leaf of the foliation of the torus $T_{(p,q)}$ by the Reeb vector field.
 \item $T_2$ is the one leaf of the foliation of the torus $T_{(p',q')}$ by the Reeb vector field.
\end{itemize}
\noindent where both pairs $(p,q)$ and $(p',q')$ are positive, relatively prime and
\[
\theta_1 < \frac{q'}{p'} < \frac{q}{p} < \theta_2
\]
\noindent Choose any relatively prime pair $(p'',q'')$ such that
\[
 \frac{p}{q} < \frac{p''}{q''} < \frac{p'}{q'}
\]
\noindent Let $[a] = [T_{(p'',q'')}]$ be the homotopy type of a leaf of the foliation of the torus $T_{(p'',q'')}$ by the Reeb vector field considered as an element in the space of loops in $S^3 \backslash L$.  Then $[T_{(p'',q'')}]$ is a simple homotopy class and $(\lambda',T = T_1 \sqcup T_2,[T_{(p'',q'')}])$ satisfies $(PLC)$ at least if $q'' \neq q,  p' \neq p''$.
\end{example}

To see this, first compute that if we have $K_{(l,m)} \subset T_{(l,m)}$ and $K_{(l',m')} \subset T_{(l',m')}$ simple closed Reeb orbits and $l / m < l' / m'$ then the linking number
\[
 \ell(K_{(l,m)},K_{(l',m')}) = l' \cdot m
\]
\noindent  This shows that the linking numbers 
\begin{align*}
\ell(T_1,[T_{(p'',q'')}]) = p'' \cdot q  & \neq \ell(T_1,T_2) = p' \cdot q \\ 
\ell(T_2,[T_{(p'',q'')}]) = p' \cdot q'' &\neq \ell(T_2,T_1) = p' \cdot q
\end{align*}
\noindent so $[T_{(p'',q'')}]$ cannot be homotopic to either $T_1$ or $T_2$ in the complement of $T_1 \sqcup T_2$ since these linking numbers are homotopy invariant in $S^3 \backslash T_1 \sqcup T_2$.  The linking numbers of the orbits in $T_{(p'',q'')}$ with $T_1, T_2$ show that orbits in different orbit sets do in fact lie in different homotopy classes of loops, so (as in the previous example) the Morse-Bott complex is simply that of a Morse function on the orbit set.  Therefore for $[T_{(p'',q'')}]$
\[
 CCH^{[a]= [T_{(p'',q'')}]}_*(\xi \mbox{ rel } T) \cong \left\{ 
\begin{matrix}
 \Q, & * = 2\cdot(p'' + q'') \hbox{ or } 2\cdot(p'' + q'') + 1 \\ 
 0 & \hbox{otherwise}
\end{matrix}  \right.
\]

Using the above calculation we derive the following corollary.  Note that we are not assuming that the contact form $\lambda$ is non-degenerate.

\begin{corollary} \label{cor-forcing2}
Let $\lambda$ be a tight contact form on the $3$-sphere.  Suppose there is a pair of knots $L' = L_1' \sqcup L_2'$ such that there is a contactomorphism taking this pair to the pair of torus knots described in Example \ref{example-3}:
\[
 T = T_1 \sqcup T_2
\]
\noindent Then for any $p'',q''$ relatively prime and such that
\[ 
 \frac{p}{q} < \frac{p''}{q''} < \frac{p'}{q'} \mbox{, } q'' \neq q' \mbox{, and }  p \neq p'' 
\]
\noindent there is a closed Reeb orbit in the homotopy class $[a]= [T_{(p'',q'')}]$ of $S^3 \backslash T$.
\end{corollary}

\begin{proof}
Applying the contactomorphism we can assume without loss of generality that $(\lambda,L') \sim (\lambda',T)$ from Example \ref{example-3}.  Then this follows from Theorem \ref{thm-forcing2} and the computation in Example \ref{example-3}, since the homotopy class $[a]$ of $K$ (which has the homotopy type of the leafs of the foliation of $T_{(p'',q'')}$ by the Reeb vector field in that example) is simple and a proper link class relative to $T_1 \sqcup T_2$.  The hypotheses $q'' \neq q',  p \neq p''$ imply that that the linking numbers $\ell(T_1,[a]) \neq \ell(T_1,T_2)$ and $\ell(T_2,[a]) \neq \ell(T_2,T_1)$ which ensures that $[a]$ is a proper link class, though $[a]$ might be a proper knot class even if these inequalities are not satisfied.
\end{proof}

\begin{remark}
 Similarly, one can consider the cases where at least one of $\theta_1, \theta_2$ is negative and obtain an analogous result for certain cases where $p,p',q,q'$ may be negative.
\end{remark}

It is also true that $(\lambda', T, [H_i])$ satisfies $(PLC)$ for the knots $H_1, H_2$ in Example \ref{example-2} (the components of the Hopf link) if $p' \neq 1$ and $q \neq 1$.  In this case, it is easiest to compute $CCH^{[H_i]}_*$ by choosing $\theta_1, \theta_2 > 0$ of $\lambda'$ appropriately.  Select them such that
\begin{align*}
 \left \lfloor \frac{p}{q}  \right \rfloor < \frac{1}{\theta_2} <  \frac{p}{q}  \\
 \left \lfloor \frac{q'}{p'} \right \rfloor < \theta_1 < \frac{ q'}{ p'}
\end{align*}
\noindent  This ensures that the only orbits in the homotopy classes $[H_1], [H_2]$ are the orbits $H_1, H_2$ themselves\footnote{This can be checked by a linking argument which we omit.} and therefore
\begin{align*}
 CCH^{[a] = [H_1]}_*(\xi \mbox{ rel } T) \cong \Q \cdot q_1, & CCH^{[a] = [H_2]}_*(\xi \mbox{ rel } T) \cong \Q \cdot q_2
\end{align*}
\noindent where the generators $q_1, q_2$ have degrees 
\[
 |q_1| = 2\left\lfloor 1 + \frac{q'}{p'} \right \rfloor, |q_2| = 2\left\lfloor 1 + \frac{p}{q} \right \rfloor 
\]

\subsection{Application to reversible Finsler metrics on $S^2$}
\label{sec-FinslerMetrics}
The geometric set-up described in the next three paragraphs is well-known; a clear treatment can be found in section $4$ of \cite{Harris-Paternain}.

Let us recall the double-covering map $S^3 \cong SU(2) \rightarrow SO(3)$:
\small
\[
 \begin{bmatrix}
  a + \qi b & c + \qi d \\ -c + d \qi & a - \qi b
 \end{bmatrix} \mapsto \begin{bmatrix}
 a^2 + b^2 - c^2 - d^2 & 2(ad + bc ) & 2(-ac + bd) \\
 2(-ad + bc ) & (a^2 - b^2 + c^2 - d^2) & 2(ab+cd)  \\
 2(ac + bd)   & 2(-ab + cd) & (a^2 - b^2 - c^2 + d^2) \\  
 \end{bmatrix}
\] \normalsize
\noindent  By taking the second and third columns\footnote{One could take different pairs of columns instead.}, we obtain a double-covering map $G: S^3 \equiv SU(2) \rightarrow M_0$, where $M_0 = \{(x,v) \in \R^3 \times \R^3 | |x|^2 = |v|^2 = 1, x \cdot v = 0 \}$ is the unit tangent bundle of $S^2$ with the round metric
\[
\begin{bmatrix}
 a \\ b \\ c \\d
\end{bmatrix}
  \mapsto \begin{bmatrix}
 2(ad + bc ) & 2(-ac + bd) \\
 (a^2 - b^2 + c^2 - d^2) & 2(ab+cd)  \\
 2(-ab + cd) & (a^2 - b^2 - c^2 + d^2) \\  
 \end{bmatrix}
\]
\noindent which gives an explicit formula for the double-covering map $S^3 \rightarrow S^1 T S^2$.

Let
\begin{itemize}
 \item $\ell_0: T S^2 \backslash 0 \rightarrow T^* S^2 \backslash 0$ denote the Legendre transform with respect to the round Riemannian metric of unit curvature.  In general, $\ell_F$ denotes the one obtained from a Finsler metric $F$.
 \item $F$ is a Finsler metric on $S^2$, while $F^*$ is the co-metric obtained via $F \circ \ell_F^{-1}$.
 \item Set $M_{F}^* = \{F^*(p) = 1 \}$ be the unit co-sphere bundle with respect to $F^*$, and let $M_{0}^*$ be the unit co-sphere bundle for $F = | \cdot |$ the norm for the round Riemannian metric. 
 \item $g_F: T^* S^2 \backslash 0 \rightarrow (0,\infty)$ is such that $F^*(x,p) = g_F(x,p)|p|$ i.e.~$g_F(x,p)$ is the homogeneous degree $0$ part of $F^*$.  Also, let $r_F:M_0^* \rightarrow M_F^*$ be the map $r_F(x,p) = (x,p/g_F(x,p))$.
 \item $G$ is the map $SU(2) \rightarrow S^1 TS^2 = M_0$ considered above.
 \item $\lambda$ is the Liouville form on $T^* S^2$, while $\lambda_0$ is the standard contact form on $S^3$.  Note that $M_F^*$ is contact type.
\end{itemize}

\begin{lemma} \cite{Harris-Paternain}
 If $f = \frac{1}{g_F} \circ \ell_0$, then
\[
 G^* \ell_0^* r_F^* (\lambda|_{M_F^*}) = 2 (f \circ G) \lambda_0 
\]
The Reeb flow of $2(f \circ G) \lambda_0$ is smoothly conjugate to the double cover of the geodesic flow on $M_F$.
\end{lemma}

There is a global frame $X_1, X_2, X_3$ on $M^*$ with the property that $X_1$ is the Reeb vector field, and $X_2, X_3$ trivialize $\xi = \ker \lambda|_{M_F^*}$.  This can all be pulled back via the double-cover $r_F \circ \ell_0 \circ G:S^3 \rightarrow M_F^*$, so we may use this frame to compute Conley-Zehnder indices.  The linearized flow along a closed Reeb orbit $z(t)$ can be computed with the following result (see e.g.~\cite{Harris-Paternain}).  Let $\phi_t$ denote the Reeb flow.  Let $x_0 X_2(z(0)) + y_0 X_3(z(0))$ be in $\xi_{z(0)}$.  Let $x(t), y(t)$ solve
\[
 d\phi_t(z(0)) = x(t) X_2(z(t)) + y(t) X_3(z(t)), \qquad x(0) = x_0, y(0) = y_0
\]
\noindent Then $x(t), y(t)$ solves the linear system
\[
 \left\{ \begin{matrix} \dot{x} &=& -K(\pi \circ z(t)) y \\ \dot{y} &=& x \end{matrix} \right.
\]
\noindent where $K:S^2 \rightarrow \R$ is determined by the Finsler structure and agrees with the curvature in the Riemannian case.

Poincar\'{e}'s rotation number of a geodesic $\gamma$ can be described as \cite{Angenent} 
\[
1/ \rho = \lim_{n \rightarrow \infty} \frac{y_{2n}}{nL}
\]
\noindent where $y_k$ is the $k^{th}$ positive zero of $y(t)$ which solves
\[
 y'' + K(\gamma(t)) \cdot y = 0
\]
\noindent and $L$ is the length of the geodesic.  This limit can also be described by
\[
 \rho = \lim_{N \rightarrow \infty} \frac{L}{2} \frac{\# \{ \mbox{zeros of } y \in [0,N] \}} {N}
\]

Let $z$ be the closed Reeb orbit which double covers $\gamma$, such that $z$ has period $T$.  We will now compare the Conley-Zehnder indices $CZ(z^k)$ to the inverse rotation number $\rho$.  The Conley-Zehnder index of $z^k$ satisfies (see, for example, the Appendix of \cite{HWZ_FEF}) 
\[
| CZ(z^k) - 2\Delta | < 2
\]
\noindent In the above, $2 \pi \Delta$ is the change in argument of the vector $(x(t),y(t))$ over the paths domain $[0,kT]$, where $(x,y)$ satisfies the linearized equation associated to $z$ above.  However, if $(x,y)$ solves this equation, then $y$ solves the equation (recall $G$ denotes the double cover)
\[
 y'' + K(G \circ z(t)) y = 0
\]
\noindent Moreover, if there are $k$ zeros of $y$, then $k -1 < 2\Delta < k+1$.  This claim is a consequence of the following observations.  First, we notice that all crossings of the $x$-axis are transverse: this is because if $y = 0$ and $y' = 0$ then $y \equiv 0$ by uniqueness of solutions.  Second, at any crossing, by the differential equation we see that if $y' > 0$ then $x > 0$, and if $y' < 0$ then $x < 0$; thus each zero contributes a transverse counterclockwise crossing of the path $(x(t),y(t))$ against the $x$-axis.  So $| 2\Delta - k | < 1$; hence (for any $k$):
\[
 | CZ(z|_{[0,kT]}) - \# \{ \mbox{zeros of } y \in [0, kT]\} | < 3
\]

Since $z$ has period $T$, $z^k$ has period $kT$.  Then
\[
0 \leq  \lim_{k \rightarrow \infty} \frac{ | CZ(z^k) - \# \{ \mbox{zeros of } y \in [0, kT]\} | }{k} \leq \lim 3 / k = 0
\]
\noindent so
\begin{align*}
  2\rho &= \lim_{kT \rightarrow \infty} \frac{2L}{2} \frac{\# \{ \mbox{zeros of } y \in [0,kT] \}} {kT} \\
& = \lim_{k \rightarrow \infty} \frac{T}{2} \frac{\# \{ \mbox{zeros of } y \in [0,kT] \}} {kT} \\
& = \lim_{k \rightarrow \infty} \frac{1}{2} \frac{\# \{ \mbox{zeros of } y \in [0,kT] \}} {k} \\
& = \lim_{k \rightarrow \infty} CZ(z^k)/2k.
\end{align*}
\noindent (The factor of $2$ is due to the fact that $z$ is a double cover of the simple geodesic $\gamma$, thus the period $T$ of $z$ is equal to $2L$, where $L$ is the length of the geodesic $\gamma$.)

If $\rho$ is irrational, then the characterization of Proposition \ref{prop-Linear Growth} shows that $z$ must be elliptic with Conley-Zehnder indices
\[
 CZ(z^k) = 2 \lfloor k (1 + (2 \rho -1) ) \rfloor + 1
\]
\noindent  Since the Finsler metric is reversible, the reversed geodesic gives rise to a second closed Reeb orbit $w$ which has the same Conley-Zehnder indices
\[
 CZ(w^k) = 2 \lfloor k (1 + (2 \rho -1) ) \rfloor + 1
\]

One may assume, without loss of generality, that $\gamma$ is an equator on $S^2$.  One sees by explicitly computing its preimage in $S^3$ that $z$ is an unknotted, self-linking number $-1$ transverse knot in $S^3$ (it will be, if $\gamma$ is chosen correctly, $S^3 \cap \C \times \{0 \}$).  Similarly, the second Reeb orbit $w$ corresponding to $\overline{\gamma}$ is also an unknotted, self-linking number $-1$ transverse knot.  Together, the pair forms a Hopf link of self-linking number $0$ (if $\gamma$ is chosen correctly it will be $(\{ 0 \} \times \C \cup \C \times \{ 0 \}) \cap S^3$).  

Since $\rho$ is irrational, $z$ and $w$ are elliptic and Corollary \ref{cor-forcing1} applies.  We find
\[
 \theta_1 = 2\rho - 1, \qquad \theta_2 = 1 / (2 \rho - 1)
\]
\noindent  Hence, if $\rho > \frac{1}{2}$, for each relatively prime pair $(p,q)$ such that $p,q >0$ and
\[
\frac{q}{p} \in \left( 2\rho-1, \frac{1}{2 \rho - 1} \right) \cup  \left(\frac{1}{2 \rho - 1}, 2\rho-1 \right)
\]
\noindent or equivalently
\[
 \frac{q+p}{2p} \mbox{ or } \frac{p+q}{2q} \in (\rho,1] \cup [1,\rho)
\]
\noindent (whichever is non-empty) we find a $(p,q)$-orbit i.e.~a closed Reeb orbit which links $q$ times with $z$ and $p$ times with $w$.  

If $\rho < \frac{1}{2}$, then instead for relatively prime $(p,q)$ such that
\[
p> 0, \frac{q}{p} \in (2 \rho -1,1], \mbox{ or } q> 0, \frac{p}{q} \in \left( 2\rho - 1, 1 \right] 
\]
\noindent or equivalently ($p,q$ can be negative, but the denominator is supposed to be positive)
\[
 \frac{q+p}{2p} \mbox{ or } \frac{p+q}{2q} \in (\rho,1] 
\]
\noindent we find a $(p,q)$-orbit i.e.~a closed Reeb orbit which links $q$ times with $z$ and $p$ times with $w$.  

Closed Reeb orbits correspond to closed (possibly double-covered) geodesics.  We have found an infinite set of closed geodesics distinguished by the linking numbers $p,q$.  We remark that the symmetry between $p,q$ above corresponds to a double-count of geodesics: a closed geodesic has two distinct lifts corresponding to the direction the geodesic is traversed, and reversing the direction of the geodesic changes the homotopy class of the lifted Reeb orbit from $(p,q)$ to $(q,p)$.

To compare with the geodesics found in \cite{Angenent} in the case $F$ is Riemannian (which are found under weaker regularity hypotheses and without any hypothesis about irrationality of $\rho$), we note that there are representatives in this homotopy class which are precisely double-covers (via $G$) of the curve $(\gamma_{p+q, 2p}, \gamma_{p+q, 2p}')$ in the unit tangent bundle, where $\gamma_{p+q,2p}$ is a $(p+q,2p)$-satellite of the geodesic $\gamma$.  However, we cannot assert using Corollary \ref{cor-forcing1} that the geodesics we have found are the $(p+q,2q)$-satellites found in \cite{Angenent}, because there are many different flat knot types lifting to transverse curves in this homotopy class in $S^3 \backslash (w \sqcup z)$.

\subsection{Open book decompositions}

Contact forms supporting an open book provide a vast class of examples.  Let $(\pi,B)$ be an open book decomposition for $(V,\xi)$.  This means that $B$ is a transverse link in $V$ (the binding), $\pi: V \backslash B \rightarrow S^1$ is a fibration, and every component $B' \subset B$ has a neighborhood $U$ such that $\pi|_{U \backslash B'} \cong S^1 \times \dot{D^2}$ has the form $(t,(r,\theta)) \mapsto \theta$ (where $(r,\theta)$ are polar coordinates on the punctured disc).  The open book \emph{supports} a contact structure $\xi$ if there is a contact form $\lambda$ (called a \emph{supporting} contact form) with $\ker (\lambda) = \xi$ such that 
\begin{itemize}
 \item $B$ is a closed for the Reeb vector field
 \item The form $d\lambda$ is symplectic on the \emph{pages} $\pi^{-1}(\theta)$, i.e.~$d\lambda|_{\pi^{-1}(\theta)} > 0, \forall \theta \in S^1$.
\end{itemize}
\noindent In particular, the Reeb vector field is transverse to the pages, so there are no contractible orbits in $V\backslash B$.  It turns out that every contact structure is supported by some open book (in fact, infinitely many \cite{Giroux}).

We first check an elementary topological fact:

\begin{lemma} \label{lem-bindingintersect}
The binding $B$ satisfies the second condition of Theorem \ref{thm-forcing2}, namely
\begin{quote}
 Every disc $F$ with $\partial F \subset B$ and $[\partial F] \neq 0 \in H_1(B)$ intersects $B$ in the interior.
\end{quote}
\noindent except when the page is a disc.
\end{lemma}

\begin{proof}
Let $F$ be any smoothly map $D^2 \rightarrow V$ with $\partial F \subset B$.  Suppose that $F^{-1}(B) \subset \partial D^2$.  Consider the covering map 
\[
 \pi: \R \times int(S) \rightarrow \R \times int(S) / ((t,x) \sim (t-1,h(x)) \cong V \backslash B
\]
\noindent coming from the open book decomposition.  For any $r < 1$ we may lift $F|_{D^2(r)}$ to a map $G:D^2(r) \rightarrow \R \times int(S)$, uniquely determined if we choose once and fix $G(0) \in \pi^{-1} \{ F(0) \}$.  Thus we can in fact define $G: int(D^2) \rightarrow \R \times int(S)$.  We may homotopy $G$ to $G': int(D^2) \rightarrow \{0 \} \times int(S) \cong int(S)$ by scaling the $\R$-coordinate.  Let $N(B)$ be the tubular neighborhood of $B$ such that $N(B) \backslash B \cong S^1 \times int(N(\partial S))$ where $N((\partial S)$ is a collar neighborhood of $\partial S$ chosen small enough so that the monodromy $h|_{N(\partial S)}$ is the identity, so we have an inclusion map $\iota: S^1 \times int(N(\partial S)) \rightarrow N(B)$ and $F|_{\partial D^2(r)} = \iota \circ \pi \circ G|_{\partial D^2(r)}$ (where $r$ is taken large enough so that $F|_{\partial D^2(r)} \subset B$).  Applying the $H_1$-functor and using the fact that $G'$ is homotopic to $G$ and $[F|_{\partial D^2(r)}] = [F|_{\partial D^2}] \neq 0 \in H_1(N(B)) \cong H_1(B)$ we have
\[
(\iota \circ \pi)_*  [G'|_{\partial D^2(r)}] = [F|_{\partial D^2(r)}] \neq 0 \in H_1(N(B)) \cong H_1(B)
\]
\noindent From this it is easy to conclude that $[G'|_{\partial D^2(r)}] = k\cdot [\partial S] \in H_1(N(\partial S))$ for some $k \neq 0$, and therefore is homotopic to a $k$-fold cover of $\partial S$ with $k \neq 0$; hence $G'|_{\partial D^2(r)}$ is homotopic to $k \cdot \partial S$ in $S$.  However, $G'|_{\partial D^2(r)}$ is null-homotopic in $S$, while the $k$-fold cover of $\partial S$ is not if $k \neq 0$ (unless the page $S$ is a disc).  This contradiction shows that we cannot have $F^{-1}(B) \subset \partial D^2$, so $F$ intersects $B$ in the interior.
\end{proof}

\begin{example}  \label{example-4} \emph{Open Book Decompositions}  
Suppose $B$ supports an open book decomposition (other than the open book decomposition with the disc as the page).  Then $(\lambda,B,[a])$ satisfies $(PLC)$ for any simple proper link class $[a]$ (relative to $B$): any disc bounding a contractible Reeb orbit either links non-trivially with the binding (since the Reeb vector field is transverse to the pages), or is asymptotic to a component of the binding (in which case by Lemma \ref{lem-bindingintersect} it has an interior intersection with the binding), and by hypothesis $[a]$ is a proper link class.  If the page is not the disc or the annulus, there should be many such $[a]$.

If the binding is elliptic non-degenerate then it satisfies $(E)$ as well.
\end{example}

It follows, for simple proper link classes $[a]$ (relative to $B$), that $CCH^{[a]}_*(\xi \mbox{ rel } B)$ is well-defined (where $\xi$ is the contact structure supported by the open book).  In particular, Theorem \ref{thm-forcing2} can be applied to any proper link class $[a]$ for the binding $B$.

In \cite{2008arXiv0809.5088C}, the authors achieve the much more difficult task of computing the cylindrical contact homology of the manifold $(V,\xi)$ when it is defined.  For $V \backslash B$ the situation is considerably simpler.  First, cylindrical contact homology of the complement will always defined (at least for simple homotopy classes), and independent of the equivalence classes $[\lambda]$ of $\equiv$ in proper link classes relative to $B$.  Second, we only need to count a small subset of the holomorphic cylinders and can ignore holomorphic planes (since they will all intersect the binding), so the computation is almost always far easier.  Since one can read off computations from \cite{2008arXiv0809.5088C}, we will content ourselves to go through a particular example in some detail to illustrate the approach concretely.

Recall that the Nielsen fixed point classes of $h$ are the equivalence classes of fixed points under the equivalence relation $\sim$, where $x \sim y$ if there is a path $\gamma$ from $x$ to $y$ such that $\gamma$ is homotopic to $h(\gamma)$ rel endpoints.  A fixed point for $h$ (or more generally a periodic point) determines a loop in the mapping torus.  It is a well-known fact that different Nielsen equivalence classes of fixed points determine different free homotopy classes of loops in the mapping torus.  More generally, two loops in the mapping torus arising from two $k$-periodic orbits of $h$ are freely homotopic if and only if there are points on the corresponding orbits which are $h^k$-equivalent (in this case we will call the orbits Nielsen equivalent as well).  Hence, to Nielsen equivalence classes we may assign a unique homotopy class $[a]$ in $V \backslash B \cong_{h.e.} \Sigma(S,h)$.

\subsection{A concrete example on the figure-eight}
\label{example-5}
\subsubsection{A map on the punctured torus and the torus with one boundary component}

Consider the linear map on $\R^2$
\[
 \begin{bmatrix}
  2 & 1 \\ 1 & 1 
 \end{bmatrix} \in SL_2(\Z) \subset GL_2(\R)
\]
\noindent which fixes the integer lattice.  It therefore descends to a map $h:T^2 \rightarrow T^2$ of $T^2 = \R^2 / \Z^2$.  Since, moreover, it fixes the integer lattice it also defines a map on $T^2 \backslash \{0\} = (\R^2 \backslash \Z^2) / \Z^2$ which preserves the area form inherited from $\R^2$.  One may also wish to think of such a map on a torus with one boundary component, which we model by removing a $\epsilon$-disc around the each point in the integer lattice and denote $T^2 \backslash D^2(\epsilon)$.  The map $h$ can be represented by the time-$1$ flow
\[
 \begin{bmatrix}
  x \\ y 
 \end{bmatrix} 
\mapsto
 \begin{bmatrix}
  2 & 1 \\ 1 & 1 
 \end{bmatrix}^t
\begin{bmatrix}
  x \\ y 
\end{bmatrix} 
\]
\noindent
generated by the symplectic vector field
\[
X(x,y) = \log  \begin{bmatrix}
  2 & 1 \\ 1 & 1 
 \end{bmatrix} \cdot \begin{bmatrix}
  x \\ y 
\end{bmatrix}, 
%
\, \log \begin{bmatrix}
  2 & 1 \\ 1 & 1 
 \end{bmatrix} %
= \frac{1}{\sqrt{5}} \log \left( \frac{3 + \sqrt{5}}{2} \right) \begin{bmatrix}
1 & 2 \\ 
2 & - 1 \end{bmatrix}
\in \mathfrak{sl}_2(\R)
\]
\noindent and it is straightforward to compute that it is generated by the Hamiltonian flow of a quadratic form
\[
Q(x,y) = A x^2 + Bxy + C y^2 = \frac{1}{\sqrt{5}} \log \left( \frac{3 + \sqrt{5}}{2} \right) \left( x^2 - xy  - y^2 \right)
\]
\noindent of signature $(1,1)$.  Take $\phi:\R \rightarrow [0,\infty)$ to be any smooth function with the following properties:
\begin{itemize}
 \item $\phi'(x) \geq 0, \forall x \in \R$
 \item $\phi|_{[0,\frac{1}{2}]} \equiv 0$
 \item $\phi|_{[1,\infty)} \equiv 1$
\end{itemize}
\noindent and denote $\phi_{\epsilon} = \phi(\frac{1}{\epsilon} \cdot)$.  Consider the Hamiltonian function
\[
 \phi_{\epsilon^2} (x^2 + y^2) \cdot Q(x,y)
\]
\noindent and denote by $\tilde{h}$ its time-$1$ map.  Since it agrees with $h$ outside of a disc of radius $\epsilon$ we may define $\tilde{h}$ on $T^2 \backslash D^2(\epsilon/2)$ by replacing the map $h$ with $\tilde{h}$ inside the disc of radius $\epsilon$.  In this way we get an area-preserving map $\tilde{h}_{\epsilon}$ on $T^2\backslash D^2(\epsilon/2)$ which is the identity on a neighborhood of the boundary.

This abstract open book gives rise to $S^3$ with binding a figure-eight knot.  The contact structure supported by this open book, however, is over-twisted.

\subsubsection{Realizing the monodromy map by a Reeb flow}

Let $h:S \rightarrow S$ be a map that preserves the area form $\omega$.  Let $\beta$ be a primitive $\omega = d\beta$.  Consider the mapping torus of $h$, which is $[0,2\pi] \times S$ quotiented by $(1,x) \sim (0,h(x))$.  We wish to find a contact form with the Reeb vector field $\partial_t$ (we will use $t$ to denote the $S^1$ ``coordinate''), using $\beta$ to construct it.  

\begin{lemma} (Giroux, see e.g.~\cite{MR2460883})
\label{lem-ReebRealizability1}
Suppose that $[h^* \beta - \beta] = 0$ in $H^1(S;\R)$\footnote{Note that $h^*\beta - \beta$ is always closed because $h^*\omega = \omega$.}.  Then $h$ can be realized by a Reeb flow on the mapping torus of $h$.
\end{lemma}

\begin{lemma} \cite{MR2460883}
\label{lem-ReebRealizability2}  If $h^* - I$ is an isomorphism on $H^1(S;\R)$, we may find a primitive $\beta'$ such that $h^*\beta' - \beta'$ is zero in $H^1(S;\R)$.  Hence by Lemma \ref{lem-ReebRealizability1} $h$ can be realized by a Reeb flow on its mapping torus.
\end{lemma}

Consider specifically the maps $\tilde{h}_{\epsilon}$ on $T^2 \backslash D^2(\epsilon/2)$ with the area form inherited from the area form for $\R^2$ (thus $T^2 \backslash D^2(\epsilon/2)$ has area $1 - \pi \epsilon^2 /4$).  We can assume that the form $\beta$ has the form $(\frac{r^2}{2} - \frac{1}{2\pi})d\theta$ in $D^2(\epsilon)$.

Let $z$ be the $S^1$ coordinate of the mapping torus, so that the contact form $dt + \beta$ in these coordinates is
\[
 \lambda = dz + \left(\frac{r^2}{2} - \frac{1}{2 \pi} \right) d\theta
\]
\noindent  To construct the open book, this is identified with a neighborhood of the boundary of $S^1 \times D^2(\epsilon/2)$ with coordinates $(s,\phi,t) \in D^2(\epsilon/2) \times S^1$ by
\[
 s = r, \quad \phi = z, \quad t = -\theta
\]
\noindent in which the form is
\[
 \lambda = d\phi + \left( \frac{1}{2\pi} - \frac{s^2}{2} \right) dt
\]
\noindent and the pages are $\phi = const$.

We will now extend this form to the interior of $D^2(\epsilon/2) \times S^1$ using a well-known argument \cite{Thurston-Winkelnkemper}.  The contact form will be
\[
 f(s^2) d\phi + g(s^2) dt \Rightarrow X_{\lambda} = \frac{1}{f'(s^2) g(s^2) - f(s^2) g'(s^2)} \left( f'(s^2) \partial_t - g'(s^2) \partial_{\phi} \right)
\]
\noindent where the functions $f,g$ have the following properties (and $T>0$ is a parameter):
\begin{itemize}
 \item $f(s^2) = 1$ and $g(s^2) = \frac{1}{2 \pi} - s^2/2$ for $s \geq \epsilon/2$
 \item $f(s^2) = s^2$ and $g(s^2) = 1 - T s^2$ for $s \leq \delta$ for some $\delta > 0$ as small as required
 \item $f' g - f g' > 0$
 \item $g' < 0$: This ensures that the Reeb vector field is everywhere positively transverse to the pages $\phi = const$ and is the reason we need to assume $T > 0$.  Moreover, then there is some constant $c = c_{\epsilon,T} \geq 0$ such that $-g'(s) \geq c > 0$.  This bounds the first return time of the Reeb flow to a page $\phi = \mbox{const}$ near the binding $B$ by the constant $2 \pi /c$.
\end{itemize}
\noindent The explicit construction of such functions is straightforward but slightly tedious, so we will omit it.  We remark as in \cite{Thurston-Winkelnkemper} that it is equivalent to drawing a smooth curve in $\R^2$ to the point $(1,\frac{1}{2\pi})$ from the point $(0,1)$ with respective velocities at these points $(1, -T)$, and $(0,-1)$,  such that the position vector is never collinear with the velocity vector (and, by the fourth condition, the velocity vector must point downward at all times as well).

We see that the binding has Conley-Zehnder index equal to
\[
CZ(B^k) = 2 \lfloor k T \rfloor +1
\]
\noindent with respect to the framing given by the angular coordinate $\phi$ (i.e. with respect to the vectors $\partial_r ((r,0),t), \partial_r ((r,\pi/2),t)$ at the point $(0,0,t) \in B$).  Therefore we will select $T$ irrational so that $B$ is elliptic non-degenerate.  As constructed, it has an open neighborhood such that no periodic orbit enters the neighborhood.  The Reeb flow thus constructed is degenerate, but after a small perturbation can be made non-degenerate without altering the dynamics too much.  

The pages form a surface of section with bounded return time.  Moreover, there is a page $S \cong T^2 \backslash 0$ such that the restriction to a subsurface $S' \cong T^2 \backslash D^2(\epsilon/2)$ of the first return map is the map $\tilde{h}_{\epsilon}$, and to a further subsurface $S_{\epsilon} \cong T^2\backslash D^2(\epsilon)$ we have $\tilde{h}_{\epsilon}|_{S_{\epsilon}} = h|_{S_{\epsilon}}$.

Given $\epsilon > 0$ sufficiently small to perform the above construction, let $\lambda_{\epsilon}$ be such a contact form.

\begin{lemma} \label{lem-ApproximationLemma0}
For each $k \geq 1$, there is a $\epsilon_k$ such that if $\epsilon < \epsilon_k$ then for each non-trivial Nielsen fixed point class there is a unique non-degenerate closed Reeb orbit for $\lambda_{\epsilon}$ in the corresponding homotopy class $[a]$ in $S^3 \backslash B \cong \Sigma(T^2 \backslash 0,h)$.
\end{lemma}

\begin{proof}
The fixed point set of $h^k$ for each $k$ is a finite set of points.  Therefore, there is a disc $D$ containing the puncture $0$ which contains no fixed points of $h^l$ for $l \leq k$.  Choose an even smaller disc $ 0 \in D'$ such that 
\[
 \left( \bigcup_{l=1}^k h^l (T^2 \backslash D) \right) \cap D' = \emptyset 
\]
\noindent  Then if $D^2(\epsilon) \subset D'$ then the periodic points of period at most $k$ of $\tilde{h}_{\epsilon}$ and $h$ on $T^2 \backslash D$ are the same.

We may also choose $0 \in D'' \subset D'$ such that $h^l(D'') \subset D'$ for $1 \leq l \leq k$.  Select $\epsilon_k$ so that $D^2(\epsilon_k) \subset D''$.  Then, if $\epsilon < \epsilon_k$, all $l$-periodic orbits of $\tilde{h}_{\epsilon}$ contained in $D$ must enter $D''$, and are therefore in fact contained in $D'$.  So we only have to examine the cases $\tilde{x} \in D'$ and $\tilde{y} \in T^2 \backslash D$.

Let $h_{\epsilon}:\R^2 \rightarrow \R^2$ lift $\tilde{h}_{\epsilon}$ in the sense that $\pi \circ h_{\epsilon} = \pi \circ \tilde{h_{\epsilon}}$ (for $\pi: \R^2 \rightarrow T^2$ the universal cover), which is uniquely determined by demanding $h_{\epsilon}(0) = 0$.  Let $\tilde{x}$ be a $l$-periodic point in $D'$ and $\tilde{y}$ a $l$-periodic point in $T^2 \backslash D$, and $\tilde{\gamma}$ a path between these two points, with lifts $x,y,\gamma$ respectively such that $\gamma(0) = x, \gamma(1) = y$ and $x$ is contained in the preimage of $D'$ containing the origin in $\R^2$.  It is easy to see that the orbit $h^i_{\epsilon}(x)$ remains in this component for all $i$ and therefore $h^k_{\epsilon}(x) = x$.  On the other hand, $h^l_{\epsilon}(y) = h^l(y) \neq y$.  After these observations it is straightforward to deduce that $\tilde{x}, \tilde{y}$ are not Nielsen equivalent for $\tilde{h}_{\epsilon}^l$: a homotopy rel endpoints between $\tilde{\gamma}$ and $\tilde{h^l_{\epsilon}} \circ \tilde{\gamma}$ would lift to one between $\gamma$ and $h^l_{\epsilon} \circ \gamma$ rel endpoints, but the endpoints of these two paths do not even coincide so there can be no such homotopy in the first place.  Thus, $\tilde{y}, \tilde{x}$ lie in different equivalence classes.

One can show that the loops corresponding to $l$-periodic points $x$ and $y$ are homotopic in the mapping torus if and only if there are iterates $\tilde{h}^i_{\epsilon}(x)$, $\tilde{h}^j_{\epsilon}(y)$ which are $\tilde{h}_{\epsilon}^l$-equivalent.  Since $x$ is the only $\tilde{h}^l_{\epsilon}$ fixed point in its $\tilde{h}_{\epsilon}^l$-equivalence class (which follows by checking directly that the fixed points of the matrix map $h^l$ lie in different equivalence classes) this shows that the corresponding orbit is the unique closed Reeb orbit in this homotopy class.  It will be hyperbolic and therefore non-degenerate.

\end{proof}

Therefore, if $[a]$ is the homotopy class of a fixed point of $h$, the triple $(\lambda, B, [a])$ satisfies $(PLC)$ (and $(\lambda,B)$ satisfies $(E)$ as well), though $\lambda$ may not be non-degenerate.  Therefore we will need to perturb, and the following Lemma asserts that small non-degenerate perturbations retain important properties:

\begin{lemma}  \label{lem-Approximation_Lemma} (Approximation lemma)  In any $C^{\infty}$ neighborhood of $\lambda_{\epsilon}$, there is a non-degenerate contact form $\lambda_{\epsilon}'$ with the following properties:
\begin{itemize}
 \item Let $\phi^t$ denote the Reeb flow.  There are constants $c, C$ and a page $S$ of the open book decomposition such that for any $x \in S$, there is a least $0 < t$ such that $\phi^t(x) \in S$ and $c < t < C$.  In particular, every closed orbit has non-trivial linking number with $B$.
 \item For each Nielsen fixed point class $[a]$ such that $\ell([a],B) \leq k$, if $\epsilon < \epsilon_k$ then there is a unique non-degenerate closed Reeb orbit for $\lambda_{\epsilon}'$ in $[a]$.
\end{itemize}
\end{lemma}

\begin{proof}
By Lemma \ref{lem-ApproximationLemma0} the (degenerate) contact form $\lambda_{\epsilon}$ has the following properties:
\begin{itemize}
 \item Let $\phi^t$ denote the Reeb flow.  There are constants $c,C$ and a page $S$ of the open book decomposition such that for any $x \in S$, there is a first $t > 0$ such that $\phi^t(x) \in S$ and $c' < t < C'$.
 \item For each Nielsen fixed point class $[a]$ such that $\ell([a],B) \leq k$, if $\epsilon < \epsilon_k$ then there is a unique non-degenerate closed Reeb orbit for $\lambda_{\epsilon}'$ in $[a]$.
\end{itemize}
\noindent  By Lemma \ref{lem-Gdelta}, in any $C^{\infty}$-neighborhood of $\lambda_{\epsilon}$, we can find a non-degenerate contact form $\lambda_{\epsilon}'$ such that $(\lambda_{\epsilon}',B) \equiv (\lambda_{\epsilon},B)$.  In particular, if $\lambda_{\epsilon}'$ is chosen close enough, the page $S$ will still be a surface of section for $\lambda_{\epsilon}'$ because of the Conley-Zehnder index in a neighborhood of the binding, and because $\lambda_{\epsilon}$ has bounded return time on the (compact) complement of this neighborhood in $S$  if $\lambda_{\epsilon}'$ is chosen close enough.  Therefore there are $c,C$ such that if $x \in S$ then for some $c < t < C$ (in fact, the first such $t>0$) we have $\phi^t(x) \in S$.

The second assertion can be deduced from the $\lambda_{\epsilon}'$ non-degeneracy of the orbits which correspond to fixed points of the unperturbed map $h$.
\end{proof}

The following proposition follows immediately:

\begin{proposition}
Let $[a]$ denote a homotopy class in $S^3 \backslash B$ corresponding to a simple Nielsen fixed point class.  Then there is an $\epsilon$ such that for the contact form $\lambda_{\epsilon}'$ in Lemma \ref{lem-Approximation_Lemma} we have $(\lambda_{\epsilon}',B,[a])$ is non-degenerate, satisfies $(PLC)$ and for $J' \in \J_{\mathrm{gen}}$
\[
 CCH^{[a]}_*(\lambda_{\epsilon}', J' \mbox{ rel } B) \cong \Q
\]
\noindent It follows that (see the following Remark \ref{rem-simple2}) $CCH^{[a]}_*$ is defined and
\[
 CCH^{[a]}_*(\xi \mbox{ rel } B) \cong \Q
\]
\end{proposition}

\begin{remark} \label{rem-simple2}
While the closed orbit above is assumed simple, it is conceivable that the homotopy class $[a]$ it is in is not (i.e.~contain multiply covered loops).  However, Remark \ref{remark-simple} explains why $CCH^{[a]}_*([\lambda] \mbox{ rel } B)$ is still well-defined.  Thus the chain maps/homotopies with the representative $\lambda_{\epsilon}'$ make sense if $\epsilon$ is sufficiently small and computes $CCH^{[a]}_*(\mu,J \mbox{ rel } B)$ for $(\mu,B) \sim (\lambda_{\epsilon}',B)$.
\end{remark}

Let us denote by $\xi_o$ the (over-twisted) contact structure supported by this open book.  One may apply Theorem \ref{thm-forcing2} to deduce, given any contact form for $\xi_o$ with a closed Reeb orbit transversely isotopic to $B$, information about the homotopy classes of Reeb orbits in $S^3 \backslash B$.  However, arguing a little more carefully we can also deduce the following growth rate of periodic orbits in the action (which proves Corollary \ref{cor-forcing3}):

\begin{proposition} \label{prop-figure8}
Suppose $\mu$ is a contact form for $\xi_o$, and its Reeb vector field has a closed periodic orbit transversely isotopic to the binding $B$.  Then
\begin{itemize}
 \item the number of geometrically distinct closed Reeb orbits of period at most $N$ grows at least exponentially in $N$, and
 \item the number of closed orbits $x$ such that $\ell(x,B) \leq k$ grows exponentially in $k$.
\end{itemize}
\end{proposition}

\begin{proof}
Choose any $\lambda' = \lambda_{\epsilon}'$ as in Lemma \ref{lem-Approximation_Lemma}, so that $(\lambda',B,[a])$ satisfies $(PLC)$ for each homotopy class $[a]$ corresponding to a non-trivial Nielsen fixed point class for $h$.  Lemma \ref{lem-Approximation_Lemma} asserts that there is a surface of section $S$ such that each point $x \in S$ has first return time $t$ in bounded by $t < C$ for a constant $C$ independent of $x$.  In particular, for each closed Reeb orbit we have a bound $T(x) < C \cdot \ell(x,B)$.  The number of simple non-trivial Nielsen fixed point classes of $h^k$ is bounded below by $A e^{bk}$ for positive constants $A,b$.  These fixed point classes satisfy $1 \leq \ell([a],B) \leq k$ and therefore $T < C \cdot k$.  Therefore, the number of closed Reeb orbits of action at most $Ck$ is at least $Ae^{bk}$.  By scaling $b = b/C$ we have the number of closed Reeb orbits of action at most $k$ is at least $Ae^{bk}$.

Let $\mu$ be a contact form on $S^3$ with kernel $\xi_o$ with an orbit transversely isotopic to $B$, and if necessary pull-back by a contact isotopy to arrange that this orbit in fact coincides with $B$.  First suppose that $\mu \prec \lambda'$.  Arguing as in the proof of Theorem \ref{thm-forcing2}, for each non-trivial simple Nielsen fixed point class $[a]$ such that $\ell([a],B) \leq k$ we obtain a $\mu$ orbit in the same homotopy class with period bounded by its $\lambda'$ period which is itself bounded by $C \cdot k$.  Therefore we have an exponential lower bound on the number of geometrically distinct periodic orbits of period at least $N$ for $\mu$ as well.  Any $\mu$ can be rescaled so that $d \mu \prec \lambda'$, and one establishes lower bounds similarly (replacing $b$ with $d \cdot b$).
\end{proof}

\subsection{Fibered hyperbolic knots} \label{sec-fiberedpseudoAnosov}

The arguments in the above example can be replaced by any open book with one boundary component and pseudo-Anosov monodromy map which is Reeb-realizable.  In most examples one must smooth the prong singularities of the pseudo-Anosov map, a procedure we will not discuss (one way to smooth the singularity is discussed in \cite{Cotton-Clay}, section $3.2$).    Because $\phi^* - I$ is invertible (as we observed in the introduction), we can apply Lemma \ref{lem-ReebRealizability2} cited above to realize the (smoothed) pseudo-Anosov map by a Reeb vector field, continuing the contact form over the binding as in the last section.    The reader should note that the argument that the number of non-trivial Nielsen fixed point classes grows exponentially is generally more complicated.  Moreover, there may be several non-degenerate fixed points in certain Nielsen equivalence classes, but an appeal to Euler characteristics can be used if needed to show that the homology is non-trivial in non-trivial classes.  For details we refer to \cite{2008arXiv0809.5088C}, Section $11.1$.

Consider now specifically a tight fibered hyperbolic knot realizing the Thurston Bennequin bound in the tight $S^3$.  By \cite{Hedden1}, tight fibered knots realizing the Thurston-Bennequin bound are determined by the topological knot type up to transverse isotopy, so we may assume without loss of generality that such a knot is in fact the binding of an open book decomposition supporting the tight $S^3$ with (pseudo-Anosov) monodromy map $\phi$.  By definition, the open book supports the tight contact structure, so the kernel of the constructed contact form (which supports this open book) is the tight contact structure on $S^3$ (up to isotopy).  Thus Corollary \ref{cor-fiberedpseudoAnosov} is proved in a similar manner as Corollary \ref{cor-forcing3}/Proposition \ref{prop-figure8} by considering the pseudo-Anosov monodromy map $\phi$ in place of the matrix map $h$.

We include some general remarks about tight fibered hyperbolic knots.  First, there are approaches to decide whether or not a given knot is fibered (\cite{MR870705}, \cite{Ni}).  Second, a knot forms the binding of an open book decomposition for the tight $S^3$ if and only if it has a transverse representative which realizes the Thurston-Bennequin bound ($2g-1$ where $g$ is the three-genus) in the tight $S^3$ \cite{Hedden1}.  Third, a sufficient (and necessary) condition for the monodromy map to be pseudo-Anosov is that the knot is a hyperbolic knot \cite{Thurston_preprint}.  Finally, to see that the first property in the definition implies the fourth property as we claimed in the introduction, evaluating the Alexander polynomial $\Delta(t) = \det(t\cdot I - h^*)$ at $t=1$ shows that $h^* - I$ is invertible since $\Delta(1) = \pm 1$ for any knot.

Matt Hedden pointed out to the author that there is an abundance of knots satisfying all these properties including, for example, the Fintushel-Stern knot (see Figure \ref{fig-HFK}, $12n_{0242}$), also known as the Pretzel knot $(-2,3,7)$.  An easy way to see that there are infinitely many examples is to notice that one can positively stabilize many times to make the genus $g$ of the page arbitrarily large, thus increasing the self-linking number of the binding (it will be $2g-1$) which distinguishes the knots.  Positive stabilization does not change the manifold or contact structure, so $1,2$ continue to hold.  It is shown in e.g.~\cite{MR2396912} that if one positively stabilizes an open book appropriately enough times that the monodromy can be made pseudo-Anosov (thus will satisfy $3$), so this infinite family will possess infinitely many tight fibered hyperbolic representatives.  Using the search form in \cite{knotinfo} and the observations above, the $35$ tight fibered hyperbolic knots with crossing number at most $12$ were found and are listed in Figure \ref{fig-HFK} (by knot diagram and Alexander-Briggs notation)

\begin{figure}
\centering 
\includegraphics[width=115mm]{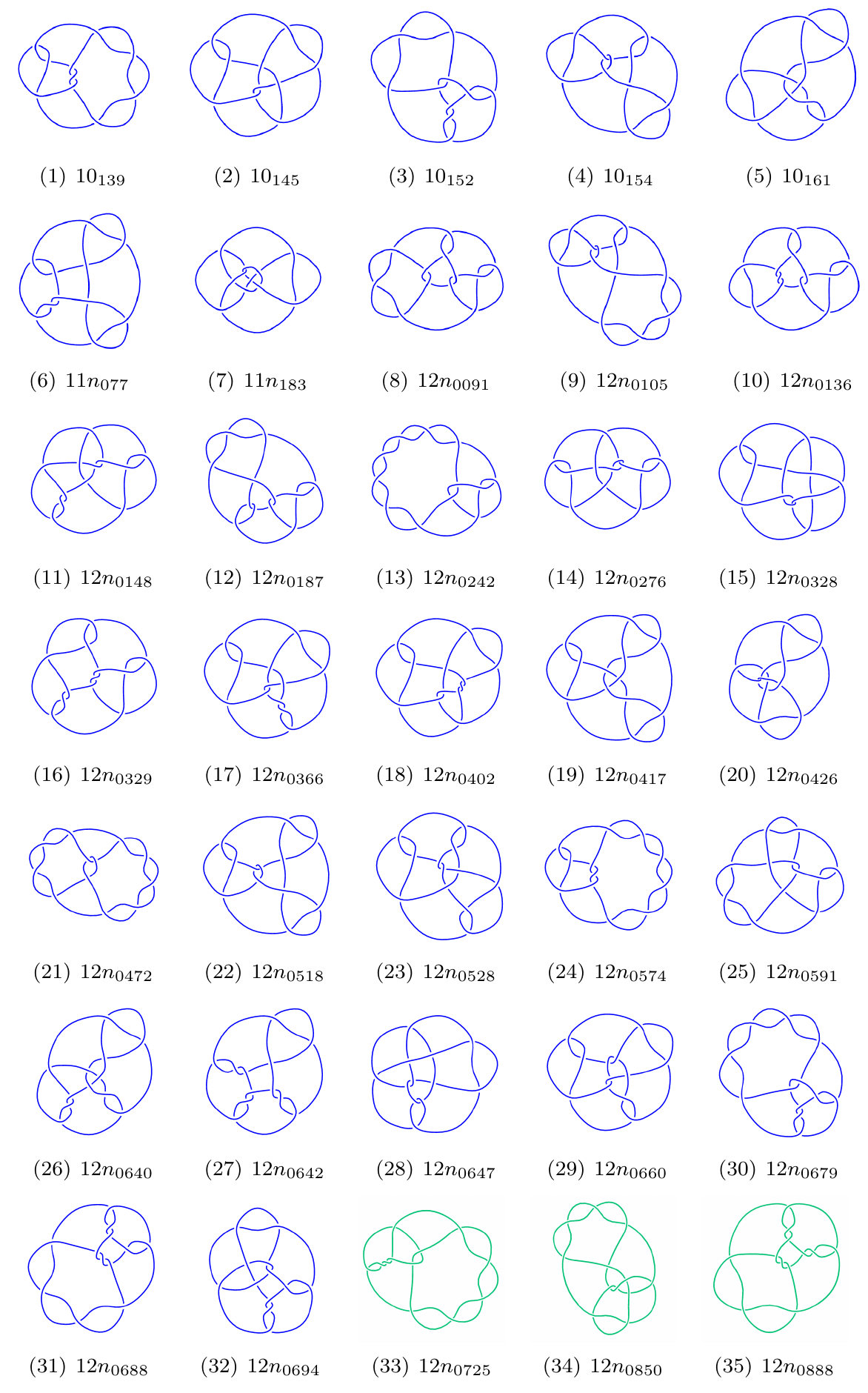}
\caption{The $35$ tight fibered hyperbolic knots with crossing number at most $12$ \cite{knotinfo}}
\label{fig-HFK}  
\end{figure}

\bigskip

\bibliography{/home/user/Dropbox/work/biblio}
\bibliographystyle{amsalpha}

\end{document}